\documentclass[11pt,reqno,letter]{amsart}
\usepackage{amsmath,amsthm,verbatim,amscd,amssymb,graphicx,color}
\usepackage[plainpages=false]{hyperref}

\usepackage{setspace}
\usepackage[english]{babel}
\usepackage{color}
\usepackage{enumerate}

\newtheorem{thm}{Theorem}[section]
\newtheorem{cor}[thm]{Corollary}
\newtheorem{lem}[thm]{Lemma}

\newtheorem{prop}[thm]{Proposition}

\theoremstyle{definition}
\newtheorem{defi}[thm]{Definition}
\newtheorem{rmk}[thm]{Remark}
\newtheorem{ex}[thm]{Example}

\numberwithin{equation}{section}

\def \N {\mathbb N}

\def \C {\mathbb C}
\def \Z {\mathbb Z}
\def \R {\mathbb R}

\def \Q {\mathbb Q}
\def \P {\mathbb P}
\def \H {\mathcal H}
\def \O {\mathcal O}

\def \L {\mathcal L}
\def \m {\mathfrak m}

\def \p {\partial}
\def \bp {\bar{\partial}}

\def \Hilb {\textbf{Hilb}}

\def \WP {\text{WP}}

\renewcommand{\leq}{\leqslant}
\renewcommand{\geq}{\geqslant}

\begin{document}

\title{Calabi-Yau manifolds with isolated conical singularities}

\begin{abstract}
Let $X$ be a complex projective variety with only canonical singularities and with trivial canonical bundle. Let $L$ be an ample line bundle on $X$. Assume that the pair $(X,L)$ is the flat limit of a family of smooth polarized Calabi-Yau manifolds. Assume that for each singular point $x \in X$ there exist a K\"ahler-Einstein Fano manifold $Z$ and a positive integer $q$ dividing $K_Z$ such that $-\frac{1}{q}K_Z$ is very ample and such that the germ $(X,x)$ is locally analytically isomorphic to a neighborhood of the vertex of the blow-down of the zero  section of $\frac{1}{q}K_{Z}$. We prove that up to biholomorphism, the unique weak Ricci-flat K\"ahler metric representing $2\pi c_1(L)$ on $X$ is asymptotic at a polynomial rate near $x$ to the natural Ricci-flat K\"ahler cone metric on $\frac{1}{q}K_Z$ constructed using the Calabi ansatz. In particular, our result applies if $(X, \mathcal{O}(1))$ is a nodal quintic threefold in $\P^4$. This provides the first known examples of compact Ricci-flat manifolds with non-orbifold isolated conical singularities.
\end{abstract}

\author{Hans-Joachim Hein}
\address{Department of Mathematics, Fordham University, Bronx, NY 10458, USA}
\email{hhein@fordham.edu}

\author{Song Sun}
\address{Department of Mathematics, Stony Brook University, Stony Brook, NY 11790, USA}
\email{song.sun@stonybrook.edu}

\thanks{HJH is partially supported by NSF grant DMS-1514709. SS is partially supported by NSF grant DMS-1405832 and an Alfred P. Sloan Fellowship.}

\date{\today}

\maketitle

\markboth{Hans-Joachim Hein and Song Sun}{Calabi-Yau manifolds with isolated conical singularities}

\thispagestyle{empty}

\tableofcontents

\section{Introduction}\label{s:intro}

\subsection{Motivation and background}

\begin{defi}\label{d:cyvariety}
A \emph{Calabi-Yau variety} $(X,L)$ consists of a complex projective variety $X$ with only canonical singularities and with trivial canonical bundle $K_X$, together with an ample line bundle $L$. Writing $\dim X = n$, we may then fix a nowhere vanishing holomorphic $(n, 0)$-form $\Omega$ on $X^{reg}$ with $\int_{X^{reg}} i^{n^2} \Omega\wedge\bar\Omega = (2\pi)^n L^n$. Such a form $\Omega$ is unique up to a global phase.
\end{defi}

When $X$ is smooth (which is automatic if $n = 1$, although in this case there is nothing to prove), Yau's solution of the Calabi conjecture \cite{Yau} asserts that there is a unique K\"ahler metric $\omega \in 2\pi c_1(L)$ with vanishing Ricci curvature. This is equivalent to saying that the Monge-Amp\`ere equation
\begin{equation} \label{eqn0-1}
\omega^n=i^{n^2}\Omega\wedge\bar\Omega
\end{equation} 
 has a unique solution $\omega\in 2\pi c_1(L)$. When $X$ is singular, we know from \cite{DePa, EGZ} that there is a unique K\"ahler current $\omega \in 2\pi c_1(L)$ that has bounded local potentials and restricts to a smooth Ricci-flat K\"ahler metric on $X^{reg}$. This means that (\ref{eqn0-1}) holds pointwise on $X^{reg}$ and also holds globally on $X$ if we interpret both sides as measures on $X$.  We will always refer to such a current $\omega$ as a \emph{singular Calabi-Yau metric}.  It is a basic question in geometric analysis to understand the exact behavior of the singular Calabi-Yau metric $\omega$ near a singular point $x\in X\setminus X^{reg}$. 

When the dimension $n=2$, it is a well-known algebro-geometric fact  that a Calabi-Yau variety has only isolated quotient singularities. These singularities are then necessarily of the form $\C^2/\Gamma$ for a finite group $\Gamma \subset {\rm SU}(2)$ acting freely on $S^3$. A straightforward extension of Yau's technique implies that the singular Calabi-Yau metric $\omega$ also has orbifold singularities, i.e. locally near each quotient singularity of $X$ it lifts to a smooth $\Gamma$-invariant metric on a domain in $\C^2$. 

When $n\geq 3$, a Calabi-Yau variety does not necessarily have only quotient singularities, so  the above question becomes much more interesting and involved. In fact, Yau's technique breaks down in a fundamental way. In order to make our arguments and presentation clear, we will assume that $(X, L)$ lies on the boundary of the moduli space of smooth polarized Calabi-Yau manifolds.

\begin{defi}\label{d:cysmoothable}
We say a Calabi-Yau variety $(X,L)$ is \emph{smoothable} if there exists a flat polarized family $\pi: (\mathcal X, \mathcal L)\rightarrow\Delta$ over the disk such that $(X, L)$ is isomorphic to $\pi^{-1}(0) = (X_0,L_0)$, $\pi^{-1}(s) = (X_s,  L_s)$ is smooth for all $s\neq 0$, and the relative canonical bundle $K_{\mathcal X/\Delta}$ is trivial.  
\end{defi}

By \cite{DS1, RZ,jiansong}, the metric completion of $(X^{reg}, \omega)$ is homeomorphic to $X$, and is isometric to the Gromov-Hausdorff limit of the smooth Calabi-Yau metrics $\omega_s \in 2\pi c_1(L_s)$ as $s \rightarrow 0$. By \cite{DS2}, for all $x \in X \setminus X^{reg}$ the completion has a unique Gromov-Hausdorff tangent cone $C(Y)$ at $x$, which is a normal affine algebraic variety endowed with a singular Calabi-Yau cone metric. (One expects that the germ $(C(Y),o)$, where $o$ denotes the vertex of $C(Y)$, is usually more singular than $(X,x)$---in particular, is not locally homeomorphic to $(X,x)$ \cite{HN}.) It is then natural to ask whether $C(Y)$ can be determined a priori in terms of given information about $(X,L)$ and $x$.

In practice the following special situation occurs quite frequently. Assume that $x$ is an {isolated} singularity of $X$ and that the germ $(X,x)$ is biholomorphic to $(C,o)$, where $C$ is a \emph{known example} of a Calabi-Yau cone and $o$ is its vertex. In this case, our main result in this paper says that $C(Y)$ is isomorphic to $C$ as a Ricci-flat K\"ahler cone, under one additional assumption on $C$.

\begin{defi}\label{d:stronglyregular}
A Calabi-Yau cone $C$ with smooth cross-section and with Ricci-flat K\"ahler cone metric $\omega_C = \frac{i}{2}\partial\bar\partial r^2$ is \emph{regular} if its Reeb field, i.e. the holomorphic Killing field $J(r\partial_r)$, generates a free $S^1$-action on $C \setminus \{o\}$. This exhibits $C$ as the blow-down of the zero section of \begin{small}$\frac{1}{q}$\end{small}$K_Z$ for some K\"ahler-Einstein Fano manifold $Z$ and $q \in \N$. We call $C$ \emph{strongly regular} if $-$\begin{small}$\frac{1}{q}$\end{small}$K_Z$ is very ample.
\end{defi}

\subsection{Main theorem}

\begin{thm} \label{thm0-1}
Let $(X, L)$ be a smoothable Calabi-Yau variety of dimension at least $3$ such that for each $x \in X\setminus X^{reg}$ the germ $(X,x)$ is isomorphic to a neighborhood of the vertex $o$ in a strongly regular Calabi-Yau cone $(C_x, \omega_{C_x})$. Then the metric tangent cone $C(Y)$ at $x$ of the singular Calabi-Yau metric $\omega$ on $(X,L)$ is isomorphic to $(C_x, \omega_{C_x})$ as a Ricci-flat K\"ahler cone. In fact, there exists a biholomorphism $P:  U\rightarrow V$ of neighborhoods $o \in U \subset C_x$ and $x \in V \subset X$ such that
\begin{align}\label{e:polyconv}
|\nabla_{\omega_{C_x}}^j(P^*\omega-\omega_{C_x})|_{\omega_{C_x}} = O(r^{\lambda-j})\end{align}
as $r \to 0$, for some $\lambda > 0$ and for all $j \in \N_0$.
\end{thm}

To illustrate the applications of this result, we consider a well-known family of examples. Let $X$ be a hypersurface of degree $n+2$ in $\P^{n+1}$ with only nodal singularities. Let $L$ be the restriction of $\mathcal O(1)$. If $n \geq 2$, then $(X,L)$ is smoothable Calabi-Yau. A nodal singularity is locally analytically isomorphic to the quadric $Q=\{z \in \C^{n+1}: \sum_{i=1}^{n+1} z_i^2=0\}$. Now $Q$ admits a Ricci-flat K\"ahler cone metric $\omega_Q$ called the \emph{Stenzel metric} \cite{Calabi,cdo,Ste}, which is not flat for $n \geq 3$, and is given by 
\begin{align}\label{e:stenzel}
\omega_{Q}=i\p\bp \left(\sum_{i=1}^{n+1} |z_i|^2\right)^{\frac{n-1}{n}}.
\end{align}
Also, $Q$ is clearly strongly regular, so Theorem \ref{thm0-1} implies the following positive answer to a well-known folklore conjecture; cf. \cite[p.248]{cdo}, \cite{jiansong}. This was known for $n = 2$, but not for any $n \geq 3$.

\begin{cor}\label{c:nodal}
The singular Calabi-Yau metric on a nodal degree $n+2$ hypersurface $(X, \mathcal{O}(1))$ in $\P^{n+1}$ {\rm (}$n\geq 3${\rm )} is polynomially asymptotic to the Stenzel metric at each node up to biholomorphism.
\end{cor}

See Appendix \ref{s:vanishing_cycles} for an immediate application to the construction of special Lagrangians, and see Section \ref{s:discussion} for some further discussion and open problems.

\subsection{Outline of proof}\label{s:outline} We now sketch the main steps in the proof of Theorem \ref{thm0-1}. For the sake of convenience we will assume  that $x$ is the only singularity of $X$. 

\begin{defi}\label{d:conical}
We call a K\"ahler current $\omega$ on $X$ \emph{conical at $x$} if $\omega$ is smooth away from $x$, has a uniformly bounded potential in a neighborhood of $x$, and satisfies \eqref{e:polyconv} with respect to some local isomorphism $P$ as in the statement of Theorem \ref{thm0-1}. Note that we do not specify $P$ a priori.
\end{defi}

By assumption there exists a local isomorphism $I: V \to U$ between a neighborhood $V$ of $x$ and a neighborhood $U$ of the vertex $o$ in $C_x$. Fix a K\"ahler current $\omega_1 \in 2\pi c_1(L)$ that agrees with $I^*\omega_{C_x}$ on $V$ and is smooth everywhere else. Such a current is easy to construct after composing $I$ with a dilation of $C_x$ if needed \cite{ArSpot}, and is obviously conical at $x$ in the above sense. Then
$$\omega_1^n=e^{F} i^{n^2}\Omega\wedge \bar\Omega,$$
where $F \in C^\infty(X^{reg})$ satisfies $i\p\bp F = 0$ on $V \setminus \{x\}$. In fact, $F$ is the real part of a holomorphic function on $V \setminus \{x\}$ and $F = const + O((r \circ I)^\epsilon)$ for some $\epsilon > 0$ (in reality, $\epsilon > 1$), where $r$ is the radius function of $\omega_{C_x}$. Note that we typically cannot arrange directly that $const = 0$.

We then set up a continuity path
\begin{equation}\label{*t}
\omega_t^n= c_t e^{tF} i^{n^2}\Omega\wedge\bar \Omega, \;\,t\in [0, 1],
\end{equation}
to be solved for a family $\omega_t$ of K\"ahler currents in $2\pi c_1(L)$. Here the constant $c_t$ is determined by integrating both sides over $X^{reg}$ and lies in  $[e^{-{\sup F}},e^{-{\inf F}}]$. By \cite{DePa, EGZ} there indeed exists a unique K\"ahler current $\omega_t\in 2\pi c_1(L)$ solving $(\ref{*t})$ that is smooth away from $x$ and has a bounded potential in a neighborhood of $x$. Let $T$ be the set of all $t\in [0, 1]$ such that $\omega_t$ is conical at $x$.  By definition $1\in T$. To prove Theorem \ref{thm0-1} it suffices to show that $T$ is both open and closed in $[0,1]$.

Openness will be proved by an implicit function theorem in Section \ref{s:openness}. There are two key points to this. First, to prove openness at $t$ it is essential that $\omega_t$ satisfies strong polynomial asymptotics of the form \eqref{e:polyconv}. Otherwise it would be very difficult to construct a good theory of the linearized operator $\Delta_{\omega_t}$ in weighted function spaces; cf. \eqref{perturb_error}. Second, even with such a  theory in hand there remains a finite-dimensional space of obstructions to openness coming from harmonic functions of less or equal than quadratic growth on $C_x$. Using computations of Cheeger-Tian \cite{CT}, we are able to prove that these obstructions vanish on \emph{every} Calabi-Yau cone. In fact, our assumptions that $(X,L)$ is smoothable and that $C_x$ is regular or even strongly regular play no role in Section \ref{s:openness}.\footnote{G. Carron and Y. Rollin pointed out to us that they were independently aware of this picture when $(C_x,\omega_{C_x})$ is a $3$-dimensional quadric cone with metric given by \eqref{e:stenzel}. In this special case, the obstructions to openness mentioned above can be computed explicitly using the representation theory of ${\rm SO}(4)$; cf. Example \ref{ex:conifold}.}

At this stage we know the following. We have $T \supset (t_*, 1]$ for some (minimal) $t_* \in [0,1)$. Thus, for each $t \in (t_*,1]$, \eqref{e:polyconv} holds for $\omega = \omega_t$ with respect to some local isomorphism $P_t$. In particular, there exists a (maximal) open neighborhood $V_t$ of $x$ in $X$ such that $|P_t^*(\omega_t|_{V_t\setminus\{x\}}) - \omega_{C_x}|_{\omega_{C_x}} < \frac{1}{10}$. There are now at least two obstructions to passing to a limit as $t \to t_*$. First, $P_t$ may diverge even if its domain of definition converges to an open subset of $C_x$. This is a concern mainly if the Fano manifold $Z$ with $C_x = \frac{1}{q}K_{Z}^\times$ has non-zero holomorphic vector fields.\footnote{D. Joyce pointed this out to us in the case of $3$-dimensional quadric cones with their natural ${\rm SO}(4,\C)$-actions. } Secondly, $V_t$ may cease to be an open neighborhood of $x$ and may even shrink down to $\{x\}$ as $t \to t_*$.

If either of these obstructions occurs, we expect a new Gromov-Hausdorff tangent cone $C(Y) \not\cong (C_x, \omega_{C_x})$ to appear at $t = t_*$. Moreover, we expect to be able to organize all scaled limits of the family $(X,\omega_t)$ into a bubble tree structure, where $C(Y)$ is the root of the tree (the tangent cone at infinity of the scaled limit of minimal asymptotic volume ratio) and $(C_x, \omega_{C_x})$ is one of the tops of the tree. This scenario can be ruled out using Yau's estimates \cite{Yau} only if the bisectional curvature of $(C_x, \omega_x)$ has a uniform one-sided bound as $r \to 0$, i.e. $(X,x)$ is a quotient singularity. 

In Section \ref{s:closedness} we use $K$-stability to rule out bubbling if $(X,L)$ is smoothable and $C_x$ is strongly regular. We will then prove that $T$ is closed, completing the proof of Theorem \ref{thm0-1}.

More precisely, the smoothability hypothesis and the fact that ${\rm Ric}(\omega_t|_V) = 0$ by \eqref{*t} allow us to use  the results of \cite{DS2} to conclude that there is a unique metric tangent cone $C(Y)$ for the limiting metric at $x$, with an algebro-geometric description. A baby version of the bubble tree picture yields that ${\rm Vol}(C(Y))$ $\leq$ ${\rm Vol}(C_x,\omega_{C_x})$ (this is almost clear from Bishop-Gromov volume comparison), with equality if and only if $C(Y)$ $\cong$ $(C_x, \omega_{C_x})$ (this is harder and relies on \cite{Berman,DS2} along with the strong regularity of $C_x$). However, a very recent algebro-geometric result of Li-Liu \cite{LL}, based on earlier work of Berman \cite{Berman} and Fujita \cite{Fujita}, implies that  ${\rm Vol}(C(Y)) \geq {\rm Vol}(C_x,\omega_{C_x})$, using only that $C_x$ is quasi-regular. Thus, $C(Y)$ is indeed isomorphic to $(C_x, \mathcal{\omega}_{C_x})$ as a Ricci-flat K\"ahler cone.

In the last step, also in Section \ref{s:closedness}, we use the $3$-circles arguments of Cheeger-Tian \cite{CT} to upgrade this information to (\ref{e:polyconv}) in a holomorphic gauge $P$. That $P$ can be taken to be holomorphic is not clear from \cite{CT}. We will instead first construct a holomorphic approximation to $P$ using \cite{DS2} and the strongly regular condition, and then use this to streamline and improve \cite{CT} in our setting. The main point of insisting on the holomorphy of $P$ is to avoid having to assume (as in \cite{CT}) that $C_x$ is integrable. Non-integrable deformations of $C_x$ causing logarithmic convergence rates should play a role only if $(X,x)$ is homeomorphic but ``far from isomorphic'' to $(C_x, o)$; cf. Section \ref{s:discussion}.

\subsection{Acknowledgments} The authors would like to thank Simon Donaldson and Mark Haskins for introducing them to this problem and for many fruitful discussions. They would also like to thank Gilles Carron, Dominic Joyce, Yuji Odaka, Yann Rollin, Cristiano Spotti, Richard Thomas, and Terry Wall for useful conversations.
\newpage

\section{Openness}\label{s:openness}

In Section \ref{s:poisson} we invert the Laplacian in weighted H\"older spaces on Riemannian manifolds with isolated conical singularities. This is not new but we find it convenient to have a self-contained proof of the easiest result that suffices for our purposes. As a consequence, if $f$ is a reasonable right-hand side with $f = const + O(r^\epsilon)$ at each singularity, then $\Delta^{-1}f = h_0 + \cdots + h_I + const \cdot r^2 + O(r^{2+\epsilon})$ at each singularity, where each $h_i$ is a harmonic function on the corresponding cone and $h_i \sim r^{\mu_i}$ for some $\mu_i \in [0,2]$. By itself this structure is not sufficient to prove openness in the continuity method: if $\mu_i < 2$ then $i\partial\bar\partial  h_i$ may blow up as $r \to 0$, and if $\mu_i = 2$ then the  $O(1)$ term $i\partial\bar\partial h_i$ might still force us to change our cone model along the continuity path. (The $i\partial\bar\partial r^2$ term from the above expansion corresponds to changing  the cone model  by dilations.) In Section \ref{s:harmonic} we show that on a \emph{Calabi-Yau} cone, these obstructions can be controlled: if $\mu_i < 2$ then $i\partial\bar\partial h_i = 0$, and if $\mu_i = 2$ then $i\partial\bar\partial h_i$ is an infinitesimal automorphism of the cone. Openness will follow easily from this in Section \ref{s:open}.

\subsection{The Laplacian on a Riemannian manifold with isolated conical singularities}\label{s:poisson} 

\begin{defi} The \emph{Riemannian cone} over a given closed connected Riemannian manifold $(Y,g_Y)$ is the Riemannian manifold $(C,g_C)$ where $C = \mathbb{R}^+ \times Y$ and $g_C = dr^2 + r^2 g_Y$, with $r: C \to \R^+$ the projection onto the first factor. We often write $C = C(Y)$ and call $Y$ the \emph{link} of $C$.
\end{defi}

\begin{defi}\label{d:ics_def}
A Riemannian manifold with isolated conical singularities (a \emph{conifold} for short) is a Riemannian manifold $(M,g)$ without boundary such that $M = {M}_0 \sqcup M_1 \sqcup \ldots \sqcup M_L$ ($L \in \N$) such that $M_0$ is the closure of a smoothly bounded domain and such that for every $\ell \in \{1, \ldots, L\}$ there exist a Riemannian cone $C_\ell$ and a diffeomorphism $\Phi_\ell: U_\ell \to M_\ell$, where $U_\ell = \{r < 1\} \subset C_\ell$, such that $|\nabla^j(\Phi_\ell^*g - g_{C_\ell})| \leq c_j r^{\lambda_\ell-j}$ holds with respect to $g_{C_\ell}$ for all $j \in \mathbb{N}_0$. Here $c_j, \lambda_\ell$ are constants and $\lambda_\ell > 0$. By abuse of language we call a conifold $M$ as above \emph{compact} if $M_0$ is compact. 
\end{defi}

Analysis on conifolds is a well-studied subject; see e.g. \cite{Beh,chee,Joy,Maz,Pac,Vert}. In particular, we do not claim originality for any of the results in this section except possibly for some of the proofs. Our main purpose is to show a direct approach to a basic regularity result for the Laplacian acting on scalar functions on a compact conifold (Theorem \ref{t:poisson}) without using any deep machinery.

\subsubsection{The Sobolev inequality on a compact conifold}

\begin{lem}\label{l:ball_sob}
Let ${C}$ be a Riemannian cone of dimension $m \geqslant 3$ and with radius function $r$. Define $V = \{r \leq 1\} \subset C$. Then there exists a constant $c$ such that 
\begin{equation}\label{hardy}
\left(\int_V r^{(m-2)\alpha -m}|u|^{2\alpha}\right)^{\frac{1}{\alpha}} \leq c \int_V |\nabla u|^2
\end{equation}
for all $\alpha \in [1,\frac{m}{m-2}]$ and for all $u \in C^1_{loc}(V)$ with $\frac{u}{r} \in L^2(V)$ and either $u|_{\partial V} = 0$ or $\int_V \frac{u}{r^2}= 0$.
\end{lem}

\begin{proof} 
We consider four different cases.

(1a) $\alpha = 1$, $u|_{\partial V} = 0$. Here we can mimic the proof of Hardy's inequality on $\R^m$. Let $Y$ denote the link of $C$ and write $V = (0,1] \times Y$. Fix $\rho \in (0,\frac{1}{2})$ and a radial cutoff function $\zeta$ with $\zeta = 1$ on $\{2\rho < r < 1\}$, $\zeta = 0$ on $\{r < \rho\}$, and $|\nabla \zeta| \leq \frac{2}{\rho}$. Then, for any $u\in C^1_{loc}(V)$ with $u|_{\partial V} = 0$,
\begin{align*}\begin{split}
 \int_V \zeta\frac{u^2}{r^2} &= \int_0^1 \int_Y \zeta\frac{u^2}{r^2}r^{m-1}\,dy\;dr =  \int_0^1 \int_Y \frac{\partial}{\partial r}\left(-\frac{1}{r}\right) \zeta u^2  r^{m-1} \, dy \, dr \\
&= \int_0^1 \int_Y  \frac{1}{r}\left(2u \frac{\partial u}{\partial r}\zeta + \frac{m-1}{r} u^2\zeta + u^2 \frac{\partial\zeta}{\partial r} \right)r^{m-1}\, dy \,dr.
\end{split}\end{align*}
Since $\frac{u}{r} \in L^2$ and $m \geqslant 3$, the claim follows by letting $\rho \to 0$ and applying Cauchy-Schwarz.

(1b) $\alpha > 1$, $u|_{\partial V} = 0$. Define $r_i = 2^{-i}$ and $A_i = \{r_{i+1} \leq r \leq r_i\}$ for all $i \in \N_0$. Let $u_i$ denote the restriction of $u$ to $A_i$ and let $\bar{u}_i$ denote the average of $u$ over $A_i$. Then, integrating over $A_i$, 
\begin{align*}\begin{split}
\|u_i\|_{2\alpha}^2 \leq 2( \|u_i - \bar{u}_i\|_{2\alpha}^2 + \|\bar{u}_i\|_{2\alpha}^2) \leq c r_i^{2+m(\frac{1}{\alpha}-1)} \|\nabla u_i\|_2^2 + c  r_i^{\frac{m}{\alpha}}\bar{u}_i^2
\end{split}\end{align*}
by the usual Sobolev inequality on the connected domain $A_0$ and by scaling. Summing over $i$,
\begin{align*}\begin{split}
\|r^{-1-\frac{m}{2}(\frac{1}{\alpha}-1)}u\|_{2\alpha}^2 \leq c \sum_{i=0}^\infty r_i^{-2-m(\frac{1}{\alpha} - 1)}\|u_i\|_{2\alpha}^2 \leq  c\|\nabla u\|_2^2 +  c \int_V \frac{u^2}{r^2}.
\end{split}\end{align*}
The claim now follows by applying (1a).

(2a) $\alpha = 1$, $\int_V \frac{u}{r^2}= 0$. By the usual Poincar{\'e} inequality on $A_0$ and by scaling, 
$$\int_V \frac{u^2}{r^2} \leq c\sum_{i=0}^\infty \frac{1}{r_i^2} \int_{A_i} u^2 \leq c\sum_{i=0}^\infty \left( r_i^{m-2}\bar{u}_i^2 + \int_{A_i} |\nabla u|^2\right).$$
Subtraction of the average of a function with respect to some finite measure minimizes the $L^2$-norm of the difference among arbitrary subtractions of constants. Thus we can assume without loss that $\sum r_i^{m-2} \bar{u}_i = 0$ instead of $\int_V \frac{u}{r^2} = 0$. Then, for all possible choices of weights $\alpha_{k\ell i} > 0$, 
\begin{align*}
\sum_{i=0}^\infty r_i^{m-2}\bar{u}_i^2 \leq c\sum_{k < \ell} r_k^{m-2}r_\ell^{m-2}(\bar{u}_k - \bar{u}_\ell)^2 \leq c\sum_{k < \ell} r_k^{m-2}r_\ell^{m-2}\left(\sum_{i= k}^{\ell-1} \alpha_{k\ell i}(\bar{u}_{i} - \bar{u}_{i+1})^2\right)\left(\sum_{j= k}^{\ell-1} \frac{1}{\alpha_{k\ell j}}\right).
\end{align*}
By Poincar{\'e} on $A_0 \cup A_1$ and by scaling,
\begin{align*}\begin{split}
(\bar{u}_i - \bar{u}_{i+1})^2 \leq\frac{1}{|A_i||A_{i+1}|}\int_{A_i}\int_{A_{i+1}} |u(x) - u(y)|^2 \,dx \, dy
\leq  c  r_i^2  \frac{|A_i \cup A_{i+1}|}{|A_i||A_{i+1}|} \int_{A_i \cup A_{i+1}} |\nabla u|^2.
\end{split}\end{align*}
Thus, it suffices to choose the weights $\alpha_{k\ell i}$ in such a way that, for all $i \in \N$, 
$$
 r_i^{2-m}\sum_{k =1}^i \sum_{\ell = i + 1}^{\infty}  r_k^{m-2}r_\ell^{m-2} \alpha_{k\ell i} \left(\sum_{j= k}^{\ell-1} \frac{1}{\alpha_{k\ell j}}\right)  \leq c.
$$
One possible choice is  $\alpha_{k\ell i} = \eta^{i-\ell}$ with a fixed $\eta \in (2^{2-m},1)$. (This method is due to \cite{GSC}.)

(2b) $\alpha > 1$, $\int_V \frac{u}{r^2}= 0$. Define $\phi$ by $\phi^{2\alpha} = r^{(m-2)\alpha-m}$. Choose a function $\zeta \in C^1_{loc}(V)$ such that $\zeta = 1$ on $\{r < \frac{1}{3}\}$ and $\zeta = 0$ on $\{\frac{2}{3} < r < 1\}$. Then
\begin{align*}
\|\phi u\|_{2\alpha} &\leq \|\phi\zeta u\|_{2\alpha} + \|\phi(1-\zeta)u\|_{2\alpha} \leq c\|\nabla(\zeta u)\|_2 + c(\|\nabla((1-\zeta)u)\|_2 + \|(1-\zeta)u\|_{2})\\
&\leq c(\|\nabla u\|_2 + \|u\|_2) \leq c\|\nabla u\|_2,
\end{align*}
using (1b), the standard Sobolev inequality on $\{\frac{1}{3} < r < 1\}$, and (2a).
\end{proof}

\begin{cor}\label{c:ball_sob}
For $C$, $r$, $V$ as in Lemma \ref{l:ball_sob}, there exists a constant $c$ such that 
\begin{equation}\label{hardy}
\left(\int_V |u|^{2\alpha}\right)^{\frac{1}{\alpha}} \leq c \int_V |\nabla u|^2
\end{equation}
for all $\alpha \in [1,\frac{m}{m-2}]$ and for all $u \in C^1_{loc}(V)$ with $\frac{u}{r} \in L^2(V)$ and either $u|_{\partial V} = 0$ or $\int_V u = 0$.
\end{cor}

\begin{proof}
If $u|_{\partial V} = 0$, this is trivial from Lemma \ref{l:ball_sob}. If $\int_V u = 0$, we first need to use the subtraction argument from (2a) above, which allows us to assume 
that $\int_V \frac{u}{r^2} = 0$ instead of $\int_V u = 0$. 
\end{proof}

\begin{prop}\label{p:ics_sob}
 Let $M$ be a compact connected conifold of dimension $m \geq 3$. Then 
\begin{equation}\label{ics_sob_2}
\left(\int_M |u|^{2\alpha}\right)^{\frac{1}{\alpha}} \leq c \int_M |\nabla u|^2
\end{equation}
for all $\alpha \in [1,\frac{m}{m-2}]$ and for all $u \in C^1_{loc}(M)$ with $\frac{u}{r} \in L^2$ at each singularity and $\int_M u = 0$.
\end{prop}

\begin{proof}
We begin by proving this for $\alpha = 1$. Using notation as in Definition \ref{d:ics_def}, observe that $M$ is connected if and only if $M_0$ is. Let $\bar{u}_0$ denote the average of $u$ over $M_0$, and for $\ell \in \{1, \ldots, L\}$ let  $\bar{u}_\ell'$, $\bar{u}_\ell'' $ denote the averages of $u$ over $M_\ell' = \{ r < \frac{1}{2}\}$, $M_\ell'' =\{\frac{1}{2} \leqslant r < 1\} \subset M_\ell$. By the subtraction argument in (2a) in the proof of Lemma \ref{l:ball_sob}, we can assume that $\bar{u}_0 + \sum \bar{u}_\ell' + \sum \bar{u}_\ell'' = 0$ instead of $\int_M u = 0$. Then, using
 the standard Poincar{\'e} inequality on $M_0$, $M_\ell''$, and Corollary \ref{c:ball_sob} on $M_\ell'$,
\begin{align*}
\int_M u^2 &\leq c\int_M |\nabla u|^2 + c\left(\bar{u}_0^2 + \sum_{\ell=1}^L (\bar{u}_\ell')^2 + \sum_{\ell=1}^L (\bar{u}_\ell'')^2 \right)\\
&\leq c\int_M |\nabla u|^2 + c\left(\sum_{\ell=1}^L (\bar{u}_\ell' - \bar{u}_\ell'')^2 + \sum_{\ell=1}^L (\bar{u}_\ell'' - \bar{u}_0)^2 \right).
\end{align*}
Continuing as in (2a) in the proof of Lemma \ref{l:ball_sob}, and applying Corollary \ref{c:ball_sob} on $M_\ell' \cup M_\ell''$ and the standard Poincar{\'e} inequality on $M_\ell'' \cup M_0$, we obtain that 
$$ (\bar{u}_\ell' - \bar{u}_\ell'')^2  \leq  c \int_{M_\ell' \cup M_\ell''} |\nabla u|^2, \;\, (\bar{u}_\ell'' - \bar{u}_0)^2   \leq  c \int_{M_\ell'' \cup M_0} |\nabla u|^2, $$
respectively. This finishes the proof for $\alpha = 1$.

If $\alpha > 1$, we argue as in (2b) in the proof of Lemma \ref{l:ball_sob}, using the $u|_{\partial V} = 0$ case of Corollary \ref{c:ball_sob}, the standard Sobolev inequality on a thickening of $M_0$ inside $M$, and the claim for $\alpha = 1$.
\end{proof}

\subsubsection{Bounded global solutions to the Poisson equation}

\begin{defi}
Let $B = B(x,\rho)$ be a geodesic ball in a Riemannian manifold. Then we define the scaled H{\"o}lder norms of a tensor field $T$ on $B$ by
\begin{equation}
\|T\|_{C^{k,\alpha}_{\rm sc}(B)} = \rho^{k+\alpha}[\nabla^k T]_{C^{0,\alpha}(B)} + \sum_{j = 0}^k \rho^j \|\nabla^j T\|_{L^\infty(B)},
\end{equation}
where $[\;\;\;]_{C^{0,\alpha}}$ denotes the ordinary $C^{0,\alpha}$ seminorm.
\end{defi}

\begin{prop}\label{p:bounded_solns}
Let $M$ be a compact connected conifold of dimension $m \geq 3$. For all $k \in \N_0$ and $\alpha \in (0,1)$ there is a constant $c$ such that the following is true. For all $f \in L^\infty(M) \cap  C^{k,\alpha}_{loc}(M)$ with $\int_M f = 0$ there exists a unique $u \in L^\infty(M)\cap C^{k+2,\alpha}_{loc}(M)$ with $\Delta u = f$ and $\int_M u = 0$. Moreover,
\begin{equation}\label{global_bound}
\|u\|_{L^\infty(M)} \leq c\|f\|_{L^\infty(M)},
\end{equation}
and for all geodesic balls $B \subset M$ such that there exist $C^{k+2,\alpha}$ coordinates $\Phi: B \hookrightarrow \R^m$ with
\begin{equation}\label{harm_coord}
\|\Phi^*g_{\R^m} - g\|_{C^{k+1,\alpha}_{\rm sc}(B)} \leq \frac{1}{10},
\end{equation}
it also holds that
\begin{equation}\label{local_reg}
\|u\|_{C^{k+2,\alpha}_{\rm sc}(\frac{1}{2}B)}\leq c\left(\|u\|_{L^\infty(B)} + \|f\|_{C^{k,\alpha}_{\rm sc}(B)} \right).
\end{equation}

\end{prop}

\begin{proof}
By Schauder theory, \eqref{harm_coord} implies \eqref{local_reg} because the Levi-Civita connection of $\Phi_*g$ is $C^{k,\alpha}$. Then uniqueness follows by multiplying the equation $\Delta u = f$ by $u$ and integrating by parts. Indeed, for some fixed small $\epsilon > 0$, \eqref{harm_coord} holds on all balls of radius $\epsilon r$ at distance $r$ from each singularity, so that $|\nabla u| = O(\frac{1}{r})$ by \eqref{local_reg}, which allows us to drop all boundary terms because $m \geq 3$.

 To prove existence, fix a smoothly bounded domain $\Omega \subset M$ with compact closure. Let $\bar{f}$ denote the average of $f$ over $\Omega$. Then there exists a unique $v \in C^{k+2,\alpha}(\bar{\Omega})$ with $\Delta v = f - \bar{f}$ in $\Omega$ and $v = 0$ on $\partial\Omega$. Extend $v$ by zero to the whole of $M$, denote the average over $M$ of this extension by $\bar{v}$, and define $u = v - \bar{v}$ on $M$. It now suffices to show that $\|u\|_{L^\infty(M)} \leq c\|f\|_{L^\infty(M)}$, where $c$ depends only on $M$. If this is true, then thanks to \eqref{local_reg} we can construct a global
solution by taking a limit  over a sequence of domains $\Omega$ that exhaust $M$, and \eqref{global_bound} will be true by construction.

We use Moser iteration to prove the desired inequality. Given any $p \geq 2$, multiply the equation $\Delta u = f - \bar{f}$ by $u|u|^{p-2}$ and integrate by parts over $\Omega$ to deduce that 
$$\int_\Omega |\nabla|u|^{\frac{p}{2}}|^2 = -\frac{p^2}{4(p-1)}\int_\Omega u|u|^{p-2}(f - \bar{f}).$$
There is no boundary term here because $u$ is constant on $\partial\Omega$ and because $\int_{\partial \Omega} \frac{\partial u}{\partial n} = \int_\Omega (f - \bar{f}) = 0$.
To start the iteration, set $p = 2$. Observe that Proposition \ref{p:ics_sob} (with $\alpha = 1$) applies to $u$ because $u$ has mean value zero over $M$ and, while not usually $C^1_{loc}$, is at least Lipschitz. Thus, 
\begin{align*}\begin{split}
\int_M u^2 \leq c\int_\Omega |\nabla u|^2 \leq c\int_\Omega |u(f - \bar{f})|, 
\end{split}\end{align*}
and hence $ \|u\|_{2} \leq c\|f - \bar{f}\|_2 \leq c\|f\|_\infty$, where the subscripts indicate $L^q(M)$ norms. To continue the iteration, let $\bar{u}_p$ denote the average of $|u|^{\frac{p}{2}}$ over $M$. Then, using Proposition \ref{p:ics_sob} ($\alpha = \frac{m}{m-2}$),
\begin{align*}
\||u|^{\frac{p}{2}} - \bar{u}_p\|_{2\alpha}^2 \leq c\int_\Omega |\nabla|u|^{\frac{p}{2}}|^2 \leq c p \|f\|_{\infty}\|u\|_p^{p-1}.
\end{align*}
Bringing the $\bar{u}_p$ term to the right-hand side,
\begin{align*}
\|u\|_{\alpha p} \leq (c p \|f\|_{\infty})^{\frac{1}{p}}\|u\|_p^{1-\frac{1}{p}}  + c^{\frac{1}{p}}(\bar{u}_p)^{\frac{2}{p}} \leq
(c p \|f\|_{\infty})^{\frac{1}{p}}\|u\|_p^{1-\frac{1}{p}}  + c^{\frac{1}{p}}\|u\|_{\frac{p}{2}}.
\end{align*}
The desired estimate now follows by induction.
\end{proof} 

\subsubsection{Asymptotics of bounded solutions at the singularities} On the regions $M_\ell$ of Definition \ref{d:ics_def}, we can view $\Delta_g$ as a perturbation of $\Delta_{g_{C_\ell}} = \partial_r^2 + \frac{m-1}{r}\partial_r + \frac{1}{r^2}\Delta_{g_{Y_\ell}}$. If $f$ satisfies reasonable decay conditions as $r \to 0$, this allows us to say more about the solution $u$ from Proposition \ref{p:bounded_solns} by using separation of variables and ODE estimates. In this section we assume that $\Phi_\ell^*g = g_{C_\ell}$. In the next section we allow for polynomially decaying errors and prove our main result (Theorem \ref{t:poisson}).

\begin{defi}\label{d:wtd_norms}
Let $U = \{r < 1\}$ in a Riemannian cone $C$. We define weighted H{\"o}lder norms
\begin{equation}
\|u\|_{C^{k,\alpha}_\nu(U)} =\sup_{x \in U}\left( r(x)^{k+\alpha-\nu}[\nabla^k u]_{C^{0,\alpha}(U \cap B(x,\frac{1}{2}r(x)))}\right) + \sum_{j = 0}^k \| r^{j-\nu} \nabla^j u\|_{L^\infty(U)}
\end{equation}
for all functions $u: U \to \R$ and all $k \in \N_0$, $\alpha \in (0,1)$, and $\nu \in \R$.
\end{defi}

We write $0 = \lambda_0 < \lambda_1 \leq \lambda_2 \leq \ldots$ for the eigenvalues of the Laplacian on the link $Y$ of $C$ (listed with multiplicity), $\phi_0$, $\phi_1$, $\phi_2$, $\ldots$ for an associated orthonormal basis of eigenfunctions, and
\begin{align}\label{e:growthrates}\mu_i^\pm = -\frac{m-2}{2} \pm \sqrt{\frac{(m-2)^2}{4} + \lambda_i}\end{align}
for the growth rates of the  corresponding homogeneous harmonic functions $r^{\mu_i^\pm} \phi_i$ on $C$.

\begin{prop}\label{cone_expansion}
Assume that $m \geq 3$. Fix $\alpha \in (0,1)$ and an integer $k > \frac{3m-1}{2}$. Let $\nu \in (-2,\infty)$ be such that $\nu + 2 \neq \mu_i^+$ for all $i \in \N$ and let $\mu_I^+$ be the largest harmonic growth rate below $\nu + 2$. Then there exists a constant $c$
such that if $u \in L^\infty(U)  \cap C^{k+2,\alpha}_{loc}(U) $ and $\Delta u \in C^{k,\alpha}_\nu(U)$, then
\begin{equation} u = h_0 + h_1 + \cdots + h_I + \bar{u},\end{equation}
where for each $i \in \{0,\ldots, I\}$,
\begin{equation}\label{coeff_bounds} h_i = A_i r^{\mu_i^+}\phi_i, \;\, A_i \in \R, \;\; |A_i| \leq \left|\int_U (\Delta u)r^{\mu_i^-}\phi_i\right| + c\left(\|u\|_{L^1(\partial U)} + \|\partial_r u\|_{L^1(\partial U)}\right), \end{equation}
and moreover
\begin{equation}\label{rem_bounds} \|\bar{u}\|_{C^{k+2,\alpha}_{\nu + 2}(\frac{1}{2}U)} \leq c\left(\|\Delta u\|_{C^{k,\alpha}_{\nu}(U)} + \|\nabla_Y^{k+2}u\|_{L^1(\partial U)} + \|\nabla^{k+1}_Y\partial_r u\|_{L^1(\partial U)}\right).\end{equation}
\end{prop}

\begin{proof} 
We begin by proving that the generalized Fourier series
\begin{align}\label{e:fs}
u(r,y) = \sum_{i = 0}^\infty u_i(r)\phi_i(y), \;\, {f}(r,y) = \sum_{i = 0}^\infty {f}_i(r) \phi_i(y),\end{align}
converge in the $C^2_{loc}(U)$ and $C^0_{loc}(U)$ topology, respectively. To see this for $u$, write
\begin{align}\label{e:fouriercoeff}
{u}_i(r) = \int {u}(r,y) \phi_i(y) \, dy = \lambda_i^{-\frac{k+2}{2}} \int ((-\Delta_y)^{\frac{k+2}{2}} {u}(r,y)) \phi_i(y)\,dy
\end{align}
for all $i \in \N$. Also, by Moser iteration and Schauder theory,
\begin{align}\label{eigenfcts}|\phi_i| + \lambda_i^{-\frac{1}{2}} |\nabla \phi_i| + \lambda_i^{-1}|\nabla^2\phi_i| \leq c\lambda_i^{\frac{m-1}{4}},\end{align}
and $\lambda_i \sim i^{\frac{2}{m-1}} $ for $i \gg 1$ from Weyl's law. The Fourier series for $u$ therefore converges in $C^2_{loc}(U)$ as long as $k > \frac{3m-3}{2}$. The proof for $f$ is analogous, requiring that $k > \frac{3m-3}{2}$ as well.

We are now able to differentiate \eqref{e:fs} term by term to derive that
\begin{equation}\label{ODE}u_i'' + \frac{m-1}{r}u_i' - \frac{\lambda_i}{r^2}u_i = {f}_i
\end{equation}
for all $i \in \N_0$. Writing $U_i = (\begin{smallmatrix} u_i \\ u_i' \end{smallmatrix})$, we have a well-known representation formula, 
\begin{align}\label{e:repform}
U_i(r) = W_i(r)\left(W_i^{-1}(1) U_i(1) + \int_1^r W_i(s)^{-1}\begin{pmatrix} 0 \\ f_i(s) \end{pmatrix} \, ds \right),
\end{align}
where, using our assumption that $m \geq 3$, the Wronskian $W_i(s)$ is given by
\begin{align}\label{wronskian}
W_i(s) = \begin{pmatrix} s^{\mu_i^+} & s^{\mu_i^-} \\ \mu_i^+ s^{\mu_i^+-1} & \mu_i^- s^{\mu_i^--1} \end{pmatrix}.
\end{align} 
In order to extract a convenient expression for $u_i$, it is helpful to define $\delta_i = 1$ if $i \in \{0,\ldots,I\}$ and $\delta_i = 0$ otherwise. Using that $u_i$ is bounded as $r \to 0$, we deduce from \eqref{e:repform}, \eqref{wronskian} that
\begin{align}\label{splitting}
u_i(r) = A_i r^{\mu_i^+} +  \bar{u}_i(r),
\end{align}
where the coefficient $A_i$ and remainder $\bar{u}_i(r)$ are given by
\begin{align}\label{def_coeffs}A_i = \frac{1}{\mu_i^+ - \mu_i^-}\left(u_i'(1) -\mu_i^- u_i(1) + \int_{\delta_i}^0 s^{1-\mu_i^+}f_i(s)\,ds\right),\\
\label{remainder}\bar{u}_i(r) = \frac{1}{\mu_i^+ - \mu_i^-}\left( r^{\mu_i^+}\int_{1-\delta_i}^r s^{1-\mu_i^+} {f}_i(s) \, ds - r^{\mu_i^-}\int_0^r s^{1-\mu_i^-}{f}_i(s)\, ds\right).
\end{align}
Our goal is to prove that
\begin{equation}\label{series}
\sum_{i=0}^\infty A_i r^{\mu_i^+}\phi_i(y), \;\, \sum_{i=0}^\infty \bar{u}_i(r) \phi_i(y),
\end{equation}
converge with appropriate estimates.

Regarding the second series of (\ref{series}), it is not hard to see from (\ref{remainder}) and (\ref{ODE}) that 
\begin{equation}\label{naive_bound}
|\bar{u}_i| + r|\bar{u}_i'| + r^2|\bar{u}_i''| \leq c(1 + \lambda_i)^{\frac{1}{2}}(\sup r^{-\nu}|{f}_i|)r^{\nu + 2}.
\end{equation}
Here our assumption that $\nu + 2 \neq \mu_i^+$ ensures that there are no log terms. Also,
\begin{equation}\label{fourier_decay}
\sup r^{-\nu}|{f}_i| \leq c(1 + \lambda_i)^{-\frac{k}{2}} \|{f}\|_{C^{k}_\nu(U)}
\end{equation}
as in \eqref{e:fouriercoeff}. Combining \eqref{naive_bound}, \eqref{fourier_decay} with \eqref{eigenfcts} and Weyl's law, and using that $k > \frac{3m-1}{2}$, we see that the second series of (\ref{series}) converges in
$C^{2}_{loc}(U)$ to a function $\bar{u}: U \to \R$ with 
\begin{equation}\label{pointwise_bound}
|\bar{u}(r,y)| \leq c \|{f}\|_{C^{k}_\nu(U)} r^{\nu + 2}.
\end{equation} 

Using this information and \eqref{splitting}, we conclude that both series in (\ref{series}) converge in the $C^{2}_{loc}(U)$ topology. Moreover, the first series is a harmonic function whereas $\Delta \bar{u} = f$. In particular, 
\begin{equation}\label{wronski}\|\bar{u}\|_{C^{k+2,\alpha}_{\nu + 2}(\frac{1}{2}U)} \leq c\left(\|f\|_{C^{k,\alpha}_{\nu}(U)} + \|\bar{u}\|_{C^0_{\nu+2}(U)}\right) \leq c\|f\|_{C^{k,\alpha}_{\nu}(U)}
\end{equation}
by standard interior elliptic estimates on scaled balls and by (\ref{pointwise_bound}).

To finish the proof, we must bound the coefficients $A_i$. For $i > I$, (\ref{def_coeffs}) tells us that
\begin{align}\label{e:randomeq}
|A_i| \leq c(1 + \lambda_i)^{-\frac{1}{2}}\left(\left|\int \frac{\partial u}{\partial r}(1,y) \phi_i(y)\, dy\right| + (1 + \lambda_i)^{\frac{1}{2}}\left|\int u(1,y) \phi_i(y)\,dy\right|\right).
\end{align}
Using the same trick as in \eqref{e:fouriercoeff}, it then follows that 
\begin{equation}\label{coeffs} |A_i| \leq c(1 + \lambda_i)^{-\frac{k+2}{2}}\left(\|\nabla^{k+2}_Y u\|_{L^1(\partial U)} + \|\nabla_Y^{k}\partial_r u\|_{L^1(\partial U)}\right).\end{equation}
Thus, if we absorb $h = \sum_{i > I} A_i r^{\mu_i^+}\phi_i$ into $\bar{u}$, then (\ref{rem_bounds}) follows from (\ref{wronski}) and the fact that 
$$\|h\|_{C^{k+2,\alpha}_{\nu+2}(\frac{1}{2}U)} \leq c\|h\|_{C^{0}_{\nu+2}(U)} \leq c\left(\|\nabla^{k+2}_Y u\|_{L^1(\partial U)} + \|\nabla_Y^{k}\partial_r u\|_{L^1(\partial U)}\right),
$$
by interior elliptic estimates on scaled balls, \eqref{coeffs}, \eqref{eigenfcts}, and because $k > \frac{3m-5}{2}$. The remaining coefficients $A_i$ for $i \in \{0, \ldots, I\}$ can be estimated directly from (\ref{def_coeffs}) because 
$$\int_0^1 s^{1-\mu_i^+}f_i(s)\,ds = \int_U (\Delta u) r^{\mu_i^-}\phi_i$$
and because $u_i(1)$, $u_i'(1)$ can be written in terms of $u(1,y)$, $\frac{\partial u}{\partial r}(1,y)$ as in \eqref{e:randomeq}.
\end{proof} 

\subsubsection{Regularity of the Laplacian in global weighted H\"older spaces}

\begin{defi}\label{d:globalnorms} Let $(M,g)$ be a compact conifold as in Definition \ref{d:ics_def}. Fix a collar neighborhood $\hat{M}_0$ of $M_0$ in $M$. For $k \in \N_0$, $\alpha \in (0,1)$, $\nu \in \R^L$, and any function $u: M \to \mathbb{R}$, define
\begin{align}
\|u\|_{C^{k,\alpha}_\nu(M)} = \|u\|_{C^{k,\alpha}(\hat{M}_0)} + \sum_{\ell = 1}^L \|u \circ \Phi_\ell\|_{C^{k,\alpha}_{\nu_\ell}(U_\ell)},
\end{align}
where $C^{k,\alpha}_{\nu_\ell}(U_\ell)$ refers to the weighted H\"older norm from Definition \ref{d:wtd_norms}.
\end{defi}

The following theorem is the main result of Section \ref{s:poisson}. 

\begin{thm}\label{t:poisson}
Let $(M,g)$ be a compact connected conifold of dimension $m \geq 3$. Use notation as in Definitions \ref{d:ics_def} and \ref{d:globalnorms}. For all $\ell \in \{1,\ldots,L\}$ and  $i,j \in \N_0$ suppose that 
\begin{align}\label{e:almost_harmonic}
|\nabla^j_{g_{C_\ell}}\Delta_{\Phi_\ell^*g}(r^{\mu_{\ell,i}^+}\phi_{\ell,i})|_{g_{C_\ell}} = O(r^ {\mu_{\ell,i}^+ - 2 + Q_{\ell,i}-j}),
\end{align} 
where the $r^{\mu_{\ell,i}^+}\phi_{\ell,i}$ are a basis of the harmonic functions on $C_\ell$ as in \eqref{e:growthrates} and where $Q_{\ell,i} \in [\lambda_\ell,\infty)$. Fix $\nu_\ell \in (-2,\lambda_\ell]$. Define a sequence $\{\nu_{\ell,j}\}_{j \in \N_0}$ by 
\begin{equation}\label{e:recursion}
\nu_{\ell,0} = \min\{\nu_\ell,\lambda_\ell - 2\}, \;\,
\nu_{\ell, j+1} = \min(\{\nu_\ell, \nu_{\ell, j} + \lambda_\ell\} \cup\{\mu_{\ell,i}^+ - 2 + Q_{\ell,i}: i \in \N_0\}).
\end{equation}
Let $\hat{\nu}_\ell \in (-2,\nu_\ell]$ be the number at which this sequence stabilizes. Suppose $\lambda_\ell, \nu_\ell$ are chosen such that $\nu_{\ell,j} + 2 \neq \mu_{\ell,i}^+$ for all $i,j \in \N_0$.  Let $\mu_{\ell, I_\ell}^+$ be the largest harmonic growth rate below $\hat\nu_\ell + 2$ on $C_\ell$. Fix cut-off functions $\chi_\ell$ on $M$ such that $\chi_{\ell_1} \equiv \delta_{\ell_1\ell_2}$ on $M_{\ell_2}$ for all $\ell_1,\ell_2$. Fix a cut-off function $\hat\chi_\ell$ on $U_\ell$ such that $\hat\chi_\ell \equiv 1$ on $\frac{1}{3}U_\ell$ and $\hat\chi_\ell \equiv 0$ on $U_\ell \setminus \frac{1}{2}U_\ell$. Then for all integers $k > \frac{3m-1}{2}$ and all $\alpha \in (0,1)$ there exists a constant $c$ such that the following holds. Let
\begin{align}
f = \bar{f} + \sum_{\ell=1}^L f_\ell \chi_\ell, \;\,\bar{f} \in C^{k,\alpha}_\nu(M),\;\,f_\ell \in \R, \;\, \int_M f=0.
\end{align}
Then the unique bounded solution $u$ to $\Delta_g u = f$ on $M$ with $\int_M u = 0$ satisfies
\begin{align}\label{e:decomp} u = \bar{u} + \sum_{\ell=1}^L \hat\chi_\ell \left(\frac{f_\ell}{2m} r^2 + \sum_{i=0}^{I_\ell} u_{\ell, i}r^{\mu_{\ell, i}^+}\phi_{\ell, i}\right) \circ \Phi_\ell^{-1},\;\,u_{\ell,i}\in\R,\\
\label{coeff_bounds_2} \|\bar{u}\|_{C^{k+2,\alpha}_{\hat{\nu} +2}(M)} + \sum_{\ell=1}^L \sum_{i=0}^{I_\ell} |u_{\ell,i}| \leq c\left(\|\bar{f}\|_{C^{k,\alpha}_\nu(M)} + \sum_{\ell=1}^L |f_\ell|\right). \end{align}
\end{thm}

\begin{proof} Recall that $u$ exists and is unique by Proposition \ref{p:bounded_solns}. For any fixed $\ell$ it holds on $U_\ell$ that
\begin{align}\label{unperturbed}
\Delta_{\Phi_\ell^*g} (\Phi_\ell^* u)  = \Phi_\ell^*\bar{f} + f_\ell = \Phi_\ell^*\bar{f} + \Delta_{g_{C_\ell}}\left(\frac{f_\ell}{2m}r^2\right).
\end{align}
We begin by considering the special case where $\Phi_\ell^*g = g_{C_\ell}$. In this case, $\lambda_\ell, Q_{\ell,i} = \infty$ and $\hat{\nu}_\ell = \nu_\ell$. Thus, Proposition \ref{cone_expansion} tells us that on $U_\ell$, 
\begin{align}\label{e:reg_10}
\Phi_\ell^*u - \frac{f_\ell}{2m}r^2= \bar{u}_\ell + \sum_{i=0}^{I_\ell} A_{\ell,i}r^{\mu_{\ell,i}^+}\phi_{\ell,i},
\end{align}
where the function $\bar{u}_\ell \in C^{k+2,\alpha}_{\nu_\ell+2}(\frac{1}{2}U_\ell)$ as well as the coefficients $A_{\ell,i}$ can be bounded in terms of $|f_\ell|$ and $\|\bar{f}\|_{C^{k,\alpha}_\nu(M)}$ thanks to \eqref{coeff_bounds}, \eqref{rem_bounds} and \eqref{global_bound}, \eqref{local_reg}. To globalize this result, we set
$$\bar{u} = u + \sum_{\ell=1}^L (\hat\chi_\ell(\bar{u}_\ell-u \circ \Phi_\ell)) \circ \Phi_\ell^{-1}.$$
Then \eqref{e:decomp}, \eqref{coeff_bounds_2} follow easily, using \eqref{global_bound}, \eqref{local_reg} once again.

In general, we first need to translate \eqref{unperturbed} into an equation with $\Delta_{g_{C_\ell}}$ on the left-hand side. For this we set up a bootstrapping scheme. Ignoring the diffeomorphism $\Phi_\ell$, we can write
\begin{align}\label{perturb}
\Delta_{g_{C_\ell}} \left( u - \frac{f_\ell}{2m}r^2\right) = \bar{f} + E_\ell\left(\frac{f_\ell}{2m}r^2\right) + E_\ell \left( u - \frac{f_\ell}{2m}r^2\right), \\
\label{perturb_error} E_\ell = \Delta_{g_{C_\ell}} - \Delta_g = (g_{C_\ell} - g) \ast \nabla^2_{g_{C_\ell}} + \nabla_{g_{C_\ell}}g \ast \nabla_{g_{C_\ell}}.
\end{align}
We will now prove by induction that for all $j \in \N_0$ there exists a constant $c_j$ such that
\begin{equation}\label{ind_claim}\left\|\Delta_{g_{C_\ell}}\left( u - \frac{f_\ell}{2m}r^2\right)\right\|_{C^{k,\alpha}_{\nu_{\ell,j}}(U_\ell)} \leq c_j\left(\|\bar{f}\|_{C^{k,\alpha}_\nu(M)} + |f_\ell|\right).\end{equation}
Since $\{\nu_{\ell,j}\}_{j\in\N_0}$ stabilizes at $\hat{\nu}_\ell$, we can then apply Proposition \ref{cone_expansion} to deduce that \eqref{e:reg_10} holds with $\bar{u}_\ell$ bounded in $C^{k+2,\alpha}_{\hat\nu_\ell + 2}(\frac{1}{2}U_\ell)$, and the statement of Theorem \ref{t:poisson} follows as before.

To prove \eqref{ind_claim} for $j = 0$, notice that the first two terms on the right-hand side of \eqref{perturb} are in $C^{k,\alpha}_{\nu_\ell}(U_\ell)$ because $\nu_\ell \leq \lambda_\ell$, and the third term is in $C^{k,\alpha}_{\lambda_\ell-2}(U_\ell)$ by  (\ref{global_bound}), (\ref{local_reg}). The claim follows. To go from $j$ to $j+1$ in \eqref{ind_claim}, we use Proposition \ref{cone_expansion} to deduce that \eqref{e:reg_10} holds with $I_\ell$ replaced by the index $I_{\ell,j}$ of the largest harmonic growth rate below $\nu_{\ell,j}+2$, and with
\begin{align}\label{e:reg_24}
\|\bar{u}_\ell\|_{C^{k+2,\alpha}_{\nu_{\ell,j}+2}(U_\ell)} + \sum_{i=0}^{I_{\ell,j}} |A_{\ell,i}| \leq c   c_j\left(\|\bar{f}\|_{C^{k,\alpha}_\nu(M)} + |f_\ell|\right).
\end{align}
(Here we were able to write $U_\ell$ instead of $\frac{1}{2}U_\ell$ on the left-hand side, and to suppress the boundary norms of $u$ from \eqref{coeff_bounds}, \eqref{rem_bounds} on the right-hand side, by appealing to \eqref{global_bound}, \eqref{local_reg}.) It remains to feed \eqref{e:reg_24} back into the third term on the right-hand side of \eqref{perturb} and to apply \eqref{e:almost_harmonic}.\end{proof}

\begin{rmk}
Working harder, using $\Phi_\ell^*g$-harmonic rather than $g_{C_\ell}$-harmonic functions in \eqref{e:decomp}, one can avoid having to pass from $\nu$ to $\hat\nu$. In our applications, $\hat\nu$ is a priori equal to $\nu$.
\end{rmk}

\subsection{Harmonic functions on Ricci-flat K\"ahler cones}\label{s:harmonic}
In this section, $Y$ denotes a connected closed Riemannian manifold, $C = C(Y)$ is the Riemannian cone with link $Y$, and $r$ is the distance function to the apex $o \in C$. Then $C$ has a natural scaling vector field $r\partial_r$. A tensor $T$ on $C$ is called $\mu$-\emph{homogeneous} for some $\mu \in \mathbb{R}$ if $\mathcal{L}_{r\partial_r}T = \mu T$. This condition implies that $|\nabla^{j}T| \sim r^{\mu + p-q-j}$ with respect to the cone metric for all $j \in \mathbb{N}_0$ if $T$ is $p$-fold covariant and $q$-fold contravariant. We find it more intuitive to speak about the \emph{growth rate} $\mu + p-q$ rather than the homogeneity $\mu$. We will almost always assume that $C$ is K\"ahler of complex dimension $n$ with parallel complex structure $J$. Then the \emph{Reeb vector field} $\xi = J(r \partial_r)$ is a holomorphic Killing field on $C$ tangent to $Y$.

\begin{lem}\label{reeb} The following hold on every K\"ahler Riemannian cone $C = C(Y)$.

{\rm (1)} Every eigenspace of the Laplacian acting on real-valued scalar functions on $Y$ can be written as an $L^2$-orthogonal direct sum of a space of $\xi$-invariant functions and of planes $\R \phi_1 \oplus \R \phi_2$ where $\phi_1,\phi_2$ are $L^2$-orthonormal and $\xi(\phi_1) = -k\phi_2$, $\xi(\phi_2) = k\phi_1$ for some $k \in \R_{>0}$.

{\rm (2)} If $n > 1$ then every homogeneous pluriharmonic function on $C$ has non-negative growth rate.

{\rm (3)} If a homogeneous real-valued function $u$ of non-negative growth rate on $C$ is pluriharmonic, then $u + iv$ is holomorphic, where $v(x) = \int_o^x d^c u$. Here the integral is taken along radial rays.

{\rm (4)} Let $u$ be a $\mu$-homogeneous complex-valued function on $C$. If $u$ is holomorphic  then $\xi(u) = i\mu u$. If $u$ is harmonic and $\xi(u) = iku$, then $|\mu| \geq |k|$ with equality if and only if $u$ or $\bar{u}$ is holomorphic.
\end{lem}

\begin{proof} (1) This is a standard consequence of the fact that the flow of $\xi$ generates a compact torus $\mathcal{T}$ of holomorphic isometries of $C$. Notice that this fact also implies that the possible values of $k$ lie in a set of the form $\langle x, \Z^r\rangle$ for some vector $x \in \R^r$, where $r = \dim \mathcal{T}$.

(2) Let $u$ be a $\mu$-homogeneous pluriharmonic function. By (1) we can assume that $u$ is real-valued and that $\xi(u) = -kv$, $\xi(v) = ku$ for some $\mu$-homogeneous real-valued pluriharmonic function $v$ and some $k \in \R_{\geq 0}$. Since $Z = r\partial_r - i\xi$ is holomorphic, $Z(\bar{Z}(u)) = 0$. Using the preceding facts, it follows that $(\mu^2 - k^2)u = 0$, hence $\mu^2 = k^2$. On the other hand, writing $u = r^\mu\phi$,
$$-\mu(\mu+2n-2)\phi = \Delta_Y\phi = \xi(\xi(\phi)) + \Delta_Y^b\phi = -k^2\phi + \Delta_Y^b\phi = -\mu^2\phi+ \Delta_Y^b\phi,$$ 
where the \emph{basic Laplacian} $\Delta_Y^b$ is non-positive. Thus, if $n > 1$, then $\mu \geq 0$. (This is a version of the
 Hartogs extension theorem. If $u$ is the real part of a holomorphic function, one could alternatively also use the well-known fact that $C \cup \{o\}$ is naturally a normal affine variety.)

(3) This is clear. The point is that the integral converges because $u$ has non-negative rate.

(4) As in the proof of (2), $\bar{\partial}u = 0$ implies $\bar{Z}(u) = 0$, hence $\xi(u) = i\mu u$. For the reverse direction, note that if $u$ is harmonic, then
$\partial^*\partial u = \bar{\partial}^*\bar{\partial}u = 0$. We integrate these equations against $\bar{u}$ over an annulus $\{1 \leq r \leq R\}$. Writing $\nabla^+ = \partial$ and $\nabla^- =\bar\partial$, this yields that
$$2\int_1^R |\nabla^\pm u|^2 = \left(\oint_{R} - \oint_1\right)(\partial_r u\mp i J \partial_r u)\bar{u} = (\mu \pm k)(R^{2\mu + 2n - 2} - 1)\int_Y |u|^2.$$ Now recall that $\mu < 0$ implies $\mu \leq 2-2n$. The claim is immediate from this.\end{proof}

\begin{thm}\label{t:harmonic_structure} Work on a K\"ahler Riemannian cone with non-negative Ricci curvature.

{\rm (1)} If $u$ is a real-valued $\mu$-homogeneous harmonic function with $\mu > 0$, then $\mu \geq 1$ with equality if and only if $u$ is an $\mathbb{R}$-linear function on $\C^n$. If $1 < \mu < 2$, then $u$ is pluriharmonic. If $\mu = 2$, then $u = u_1 + u_2$ where $u_1,u_2$ are $2$-homogeneous, $u_1$ is pluriharmonic, and $u_2$ is $\xi$-invariant.

{\rm (2)} If a real-valued $\mu$-homogeneous harmonic function $u$ with $\mu > 0$ is $\xi$-invariant, then $\mu \geq 2$ with equality if and only if ${\rm Ric}\,\nabla u = 0$ and $\nabla u$ is a holomorphic vector field.

{\rm (3)} Assume in addition that the cone is Ricci-flat. Then the space of all holomorphic vector fields that commute with $r\partial_r$ can be written as $\mathfrak{p} \oplus J\mathfrak{p}$, where $\mathfrak{p}$ is spanned by $r\partial_r$ and by the gradient fields of the $\xi$-invariant $2$-homogeneous harmonic functions. All elements of $J\mathfrak{p}$ are Killing fields.

\end{thm}

\begin{rmk}
In (1) above, the statement that $\mu \geq 1$ (with equality precisely for a linear function on Euclidean space) holds on all Riemannian cones with ${\rm Ric} \geq 0$ by a well-known reformulation of the Lichnerowicz-Obata theorem; cf. Lemma \ref{cclich}. The fact that $u$ is pluriharmonic if $1 < \mu < 2$ was pointed out and crucially used in \cite[Section 3]{CH1}. (2) and (3) together amount to a cone version of the Lichnerowicz-Matsushima theorem for Fano manifolds; e.g. see \cite[Section 6.3]{ball}.
\end{rmk}

\begin{ex}\label{ex:conifold} Let $C = C(Y)$ be the $3$-dimensional quadric cone $\{z \in \C^4: \sum_{i=1}^4 z_i^2 = 0\}$ together with the Ricci-flat K\"ahler cone metric $i\partial\bar\partial (\sum_{i=1}^4 |z_i|^2)^{2/3}$ of  Calabi \cite{Calabi}, Candelas-de la Ossa \cite{cdo}, and Stenzel \cite{Ste}. Then $Y = ({\rm SU}(2) \times {\rm SU}(2))/{\rm U}(1)$, where ${\rm U}(1)$ is embedded as the anti-diagonal of the maximal torus of ${\rm SU}(2)$ \cite[(2.28)]{cdo}, and where both SU$(2)$ factors carry a Berger deformation with Hopf circle radius $2$$\sqrt{2}$$/3$ of a round sphere of radius $\sqrt{2/3}$ \cite[(2.11)]{cdo}. This picture allows one to compute the spectrum of $\Delta_Y$. The answer is known in the physics literature \cite{gu}. 

One finds that the spectrum can be parametrized by triples 
$(k,n_1,n_2) \in \N_0^3$ with $k \leq n_1 \leq n_2$ and $k \equiv n_1 \equiv n_2\; {\rm mod}\;2$. Here $k$ represents the weight of the eigenfunction under the diagonal action of ${\rm U}(1)$, which corresponds to the Reeb vector field $\xi$ of $C$. The associated eigenvalue is 
\begin{equation}
\lambda =  \frac{3}{2}(n_1^2 + n_2^2) + 3(n_1 + n_2) -\frac{3}{4}k^2,
\end{equation}
and the multiplicity of this eigenvalue is given by
\begin{equation}
m_\lambda = \begin{cases} (n+1)^2 &\textit{if}\;\, k = 0, \; n_1 = n_2 = n,\\
2(n+1)^2 &\textit{if}\;\, k \neq 0, \; n_1 = n_2 = n,\\
2(n_1+1)(n_2+1) &\textit{if}\;\, k = 0, \; n_1 \neq n_2,\\
4(n_1+1)(n_2+ 1) &\textit{if}\;\, k\neq 0, \; n_1 \neq n_2.\end{cases}
\end{equation}

Setting $k = n_1 = n_2 = d$ yields $\lambda = \frac{9}{4}d^2 + 6d$ with multiplicity $m_\lambda = 2(d+1)^2$. The corresponding homogeneous harmonic functions on $C$ are precisely the pluriharmonic ones. They grow at rate $\frac{3}{2}d$ and are obtained by restricting the real and imaginary parts of homogeneous complex polynomials of degree $d$ from $\C^4$ to $C$. In fact, $d = 1$ yields the smallest eigenvalue $\lambda = 8.25$, $m_\lambda = 8$. The next eigenvalue is $\lambda = 12$, $(k,n_1,n_2) = (0,0,2)$, $m_\lambda = 6$, with associated harmonic growth rate $2$. The complexification of this eigenspace can be identified with the Lie algebra of the centralizer  ${\rm SO}(4,\C)$ of the $1$-parameter subgroup generated by $r\partial_r - i\xi$ in the infinite-dimensional group ${\rm Aut}(C)$.
\end{ex}

The proof of Theorem \ref{t:harmonic_structure} relies in an essential way on the following lemma due to Cheeger and Tian \cite[Lemma 7.27]{CT}. We give a detailed proof of this lemma in Appendix \ref{s:CT}.

\begin{lem}[Cheeger-Tian]\label{l:ct1} Let $C = C(Y)$ be a Riemannian cone of dimension $m \geqslant 3$ such that ${\rm Ric}_C \geq 0$. Let $\psi$ be a homogeneous $1$-form on $C$ with growth rate in $[0,1]$. Then $(dd^* + d^*d)\psi = 0$ holds if and only if, up to linear combination, either $\psi = d(r^\mu \phi)$,
where $\phi$ is a $\lambda$-eigenfunction on $Y$ for some $\lambda \in [m-1,2m]$ and $\mu$ is chosen so that $r^\mu\phi$ is a harmonic function on $C$, or $\psi = r^2\eta$, where $L_{r\partial_r}\eta = 0$ and, at $r = 1$, $\eta^\sharp$ is a Killing field on $Y$ with ${\rm Ric}_Y\,\eta^\sharp = (m-2)\eta^\sharp$, or $\psi = r dr$.
\end{lem}

\begin{proof}[Proof of Theorem \ref{t:harmonic_structure}] (1) As mentioned above, the statement that $\mu\geq 1$, with equality precisely for a linear function, is well-known even without the K\"ahler assumption. It follows easily from the closed case of Lemma \ref{cclich}, i.e. the Lichnerowicz-Obata theorem, and \eqref{e:growthrates}. 

To deal with the case $1 \leq \mu < 2$, we only need to observe that $\psi = d^c u$ is a harmonic $1$-form by the Hodge-K\"ahler identities, and that its growth rate lies in $[0,1)$. Then Lemma \ref{l:ct1} tells us that $\psi = dv$ for some other $\mu$-homogeneous harmonic function $v$. Thus, $u + iv$ is holomorphic.

Finally, assume that $\mu = 2$. Write $u = r^2\phi$. By Lemma \ref{l:ct1},
$$d^c(r^2\phi) = d(r^2 \psi) + r^2\eta + d(\alpha r^2),$$
where $\phi,\psi$ are eigenfunctions of $\Delta_Y$ with eigenvalue $4n$, $\eta$ is a scale-invariant $1$-form on $C$ dual to a Killing field on $Y$, and $\alpha$ is a constant. Expanding and dividing by $r$,
\begin{equation}\label{rigid}
2\phi d^c r + r d^c \phi = 2\psi dr + r d\psi + r\eta + 2\alpha dr.
\end{equation}
Dotting \eqref{rigid} with $dr$  yields
$-\xi(\phi) = 2\psi + 2\alpha$, hence $\alpha = 0$ by integrating over $Y$, so $\xi(\phi) = -2\psi$. On the other hand, dotting \eqref{rigid} with $d^cr$ and using that $\alpha = 0$ yields
$$2\phi = \xi(\psi) + \eta(\xi).$$
The function $\eta(\xi)$ is $\xi$-invariant because $\nabla_Y\eta$ is skew-symmetric and $\xi$ has geodesic orbits on $Y$. Also, it is a $4n$-eigenfunction because $\xi(\psi)$ and $\phi$ are.
Thus, the function $v = r^2(\phi - \frac{1}{2}\eta(\xi) + i\psi)$ is harmonic and $2$-homogeneous and satisfies $\xi(v) = 2iv$. By Lemma \ref{reeb}(4), $v$ must be holomorphic. This gives the desired decomposition $u = u_1 + u_2$ with $u_2 = \frac{1}{2}r^2\eta(\xi)$.

(2) We begin by observing that
$$
\nabla^*\nabla^{0,1} \nabla u + {\rm Ric} \,\nabla u = 0.
$$
This follows from the Bochner formula for the harmonic $1$-form $du$ by passing from $1$-forms to vector fields and using that
$\nabla^*(\nabla^{1,0}X - \nabla^{0,1}X) = - \nabla^*(J \circ \nabla X \circ J) = {\rm Ric}\,X$; cf. \cite[(4.80)]{ball}. Integrating this identity against $\nabla u$ and integrating by parts, 
\begin{align*}\int_1^R |\nabla^{0,1}\nabla u|^2 +{\rm Ric}(\nabla u, \nabla u) &= \frac{1}{2}\left(\oint_R - \oint_1\right)(\langle \nabla_{\partial_r}\nabla u, \nabla u\rangle + \langle J\nabla_{J\partial_r}\nabla u, \nabla u\rangle).\end{align*}
Writing $u = r^\mu\phi$, it is easy to see that 
\begin{align*}\nabla u &= r^{\mu-1}(\mu \phi \partial_r + \nabla_Y \phi),\\
\langle \nabla_{\partial_r}\nabla u, \nabla u\rangle = \frac{1}{2}\partial_r|\nabla u|^2 &= (\mu - 1)r^{2\mu-3}(\mu^2 \phi^2 + |\nabla_Y \phi|^2).
\end{align*}
As for the remaining term,
\begin{align*}
J\nabla_{J\partial_r} \nabla u &= r^{\mu-1}(\mu\phi J\nabla_{J\partial_r}\partial_r + J\nabla_{J\partial_r} \nabla_Y\phi) \\
&= -r^{\mu-2}(\mu  \phi \partial_r + \nabla_Y \phi).
\end{align*}
Here we have used that $\phi$ is $\xi$-invariant, that $\nabla_X \partial_r = \frac{1}{r}X$ for every vector $X$ tangent to the slices of the cone, and that $\nabla_\xi \nabla \phi = \nabla_{\nabla\phi}\xi = J\nabla\phi$ because $[\xi,\nabla\phi] = \mathcal{L}_\xi (\nabla \phi) = (\mathcal{L}_\xi (d\phi))^\sharp = \nabla (\xi(\phi)) = 0$.
Combining these computations, we obtain that
\begin{align*}
\int_1^R |\nabla^{0,1}\nabla u|^2 +{\rm Ric}(\nabla u, \nabla u) &= \frac{1}{2} (\mu-2)(R^{2\mu + 2n-4}-1)\int_Y(\mu^2\phi^2 + |\nabla_Y\phi|^2).\end{align*}
The claim is immediate from this.

(3) Assuming only that ${\rm Ric} \geq 0$, (2) shows
that the $0$-homogeneous
holomorphic vector fields on the cone
contain $\mathfrak{p} \oplus J\mathfrak{p}$ and that ${\rm Ric} = 0$ on this space. Also, $J\mathfrak{p}$ clearly consists of Killing fields.

Now assume that ${\rm Ric} = 0$. Let $X$ be a $0$-homogeneous holomorphic vector field. Then the
$1$-form $\psi$ dual to $X$
is $0$-homogeneous harmonic because on every K{\"a}hler manifold, 
\begin{align*}
2\bar{\partial}^*\bar{\partial}X &= \nabla^*\nabla X - {\rm Ric}\, X,\\
(dd^* + d^*d)\psi &= \nabla^*\nabla\psi + {\rm Ric}\,\psi.\end{align*}
(See \cite[(3.35)]{ball} for the first identity.) Thus, by Lemma \ref{l:ct1}, 
\begin{equation}\label{e:matsu3}
X = \nabla(r^2\phi) + Z + \nabla(\alpha r^2),
\end{equation}
where $\phi$ is a $4n$-eigenfunction of $\Delta_Y$, $Z$ is a $0$-homogeneous Killing field, and $\alpha$ is a constant. We can assume that $\alpha = 0$ because $r\partial_r \in \mathfrak{p}$. Then \eqref{e:matsu3} implies that $\nabla^{0,1}\nabla(r^2\phi) + \nabla^{0,1}Z = 0$. Here the first term is a symmetric endomorphism of the real tangent bundle of the cone and the second term is skew-symmetric. This shows that both $\nabla(r^2\phi)$ and $Z$ are holomorphic. 

Repeating this argument with $X$ replaced by $JZ$, we obtain a decomposition
$$JZ = \nabla(r^2\tilde{\phi}) + \tilde{Z} +  \nabla(\tilde{\alpha} r^2)$$
with another $4n$-eigenfunction $\tilde{\phi}$, Killing field $\tilde{Z}$, and constant $\tilde{\alpha}$. We can again assume that $\tilde{\alpha} = 0$ because $r\partial_r \in \mathfrak{p}$, hence $\xi \in J\mathfrak{p}$. Since $\nabla(JZ)$ is a priori symmetric whereas $\nabla\tilde{Z}$ is skew-symmetric, it follows that $\nabla\tilde{Z} = 0$, hence $\tilde{Z} = 0$ (translation vector fields on $\C^n$ are not $0$-homogeneous).

It remains to prove that both $\phi$ and $\tilde{\phi}$ are $\xi$-invariant. We will prove more generally that if $\phi$ is any function on the link of a K\"ahler cone such that $\nabla(r^2\phi)$ is holomorphic, then $\xi(\phi) = 0$. Indeed, if  $\nabla(r^2\phi)$ is holomorphic, then   $\nabla_{\xi}\nabla(r^2\phi)= J\nabla_{r\partial_r}\nabla(r^2\phi)$, hence
\begin{align*} 2\xi(\phi)r\partial_r + 2\phi \nabla_\xi (r\partial_r) + r^2 \nabla_\xi\nabla\phi &= 2\phi\xi  + 2r^2J\nabla\phi+ r^2J\nabla_{r\partial_r}\nabla\phi.\end{align*}
Using that $\nabla_{\xi}\partial_r= \frac{1}{r}\xi$, taking the inner product with $\frac{1}{r}\partial_r$, and using the symmetry of $\nabla^2\phi$,
$$2\xi(\phi) + \langle \nabla_{r\partial_r}\nabla\phi, \xi\rangle = 0.$$
But $\nabla_{r\partial_r}\nabla\phi = -\nabla\phi$ by homogeneity, so $\xi(\phi) = 0$ as desired. 
\end{proof}

\begin{rmk}
In the proof of (3) above, we have used that if a K\"ahler manifold is Ricci-flat, then the metric dual of every (local) holomorphic vector field is a harmonic $1$-form. It is clear from the Bochner formulas in the proof of (3) that the converse of this statement is true as well.
\end{rmk}

\subsection{Openness in the continuity method}\label{s:open} We now return to the continuity method sketched in Section \ref{s:outline} and prove the following openness theorem.

\begin{thm}\label{t:open}
Let $X$ be a compact K\"ahler space of dimension $n\geq 2$ with only isolated canonical singularities and with trivial canonical bundle. Assume that for each $x \in X\setminus X^{reg}$ the germ $(X,x)$ is isomorphic to a neighborhood of the vertex in some Calabi-Yau cone. Fix a holomorphic volume form $\Omega$ on $X$. Let $F: X^{reg} \to \R$ be a smooth function which is pluriharmonic in a neighborhood of each $x \in X \setminus X^{reg}$. Assume that there exists a conical K\"ahler current $\omega$ on $X$ such that
\begin{align}
\omega^n = e^F i^{n^2}\Omega \wedge \bar\Omega.
\end{align} 
Then for some $\delta_0 > 0$ and all $\delta \in (-\delta_0, \delta_0)$ there exists a conical K\"ahler current $\omega_\delta$ representing the same K\"ahler class as $\omega$ such that
\begin{align}\label{e:nearby_eqn}
\omega_\delta^n = c_\delta e^{(1+\delta)F} i^{n^2}\Omega \wedge \bar\Omega.
\end{align}
Here the constant $c_\delta$ is determined by integrating both sides over $X$.
\end{thm}

\begin{rmk}
(1) Each germ $(X,x)$ is assumed to be isomorphic to some Calabi-Yau cone $(C_x,o)$. This cone is allowed to depend on $x$ and it can be regular, quasi-regular, or irregular. The K\"ahler currents $\omega$ and $\omega_\delta$ are understood to be conical at each $x$ in the sense of Definition \ref{d:conical}.

(2) The hypotheses of Theorem \ref{t:open} are strictly weaker than those of Theorem \ref{thm0-1}. In particular, Theorem \ref{t:open} implies that the set $T \subset [0,1]$ defined in Section \ref{s:outline} is open. Imposing the stronger hypotheses of Theorem \ref{thm0-1} would not help to simplify the proof of Theorem \ref{t:open}.
\end{rmk}

\begin{proof}[Proof of Theorem \ref{t:open}]
We will assume for simplicity that $X$ has only one singularity $x$. (All of the analysis so far has been developed for spaces with any number of singularities. This should make it clear that no new ideas are needed to prove Theorem \ref{t:open} in general.)

By assumption we have a biholomorphism $P: U \to V$, where $U$ is a neighborhood of the vertex $o$ in some Calabi-Yau cone $(C_x, \omega_{C_x})$ and $V$ is a neighborhood of $x$ in $X$, such that
$$
|\nabla^j_{\omega_{C_x}}(P^*\omega - \omega_{C_x})|_{\omega_{C_x}} = O(r^{\lambda-j})
$$
as $r \to 0$ for some $\lambda > 0$ and all $j\in\N_0$. In order to make this compatible with Definition \ref{d:ics_def}, we replace $U$ by the region $\{r < 1\}$ and $P$ by a $C^\infty$ embedding of $\{r < 1\}$ into $X$ that coincides with $P$ on $\{r < r_0\}$ for some $0 < r_0 \ll 1$. We continue to write $P: U \to V$ for this modified map.

It is a standard fact in Sasaki geometry that the so-called \emph{transverse automorphisms} of $C_x$, i.e. those holomorphic automorphisms of $C_x$ that commute with the $1$-parameter group generated by $r\partial_r - i\xi$, form a complex Lie group. Let $G$ denote the connected component of the identity of this Lie group. Theorem \ref{t:harmonic_structure}(3) tells us that $Lie(G) = \mathfrak{p} \oplus J\mathfrak{p}$, where $\mathfrak{p}$ is spanned by $r\partial_r$ and by the gradient vector fields of the $\xi$-invariant $2$-homogeneous harmonic functions on $C_x$.  

We are now in position to set up the function spaces required for our implicit function theorem.
Fix a cut-off function $\hat{\chi}$ as in Theorem \ref{t:poisson}, an integer $k \geq 3n$, and a number $\nu\in\R_{>0}$ such that $\nu \leq \min\{\lambda,1\}$ and $C_x$ admits no homogeneous harmonic functions with growth rate in $(2,\nu+2]$. Let $\mathcal{P}$ denote the vector space spanned by the homogeneous pluriharmonic functions with growth rate in $[0,2]$ on $C_x$. Say that a function $u,f: X^{reg} \to \R$ is of type U, F, respectively, if
\begin{align*}
u = \bar{u} + \hat\chi \left(p + \frac{1}{2}(r^2 \circ \Phi - r^2)\right) \circ P^{-1}&\;\,{\rm with}\,\,\bar{u} \in C^{k+2,\alpha}_{\nu+2}(X^{reg}), \, p \in \mathcal{P} , \,\Phi \in G, \tag{U}\\
f = \bar{f} + f_x&\;\,{\rm with}\,\,\bar{f} \in C^{k,\alpha}_\nu(X^{reg}), \, f_x \in \R, \tag{F}
\end{align*}
where the weighted H\"older spaces are defined with respect to the given conifold structure on $X^{reg}$.
Using this terminology, we define two sets of functions on $X^{reg}$ by
\begin{align*}
\mathcal{U} = \left\{u: X^{reg} \to \R: u {\rm \,\,is \,\, of \,\, type \,\, U}, \; \int_{X^{reg}} u \omega^n = 0, \; \omega + i\partial\bar\partial u > 0\right\},\\
\mathcal{F} = \left\{f: X^{reg} \to \R: f{\rm \,\, is \,\, of \,\, type \,\, F}, \; \int_{X^{reg}} (e^f - 1) \omega^n = 0\right\}.
\end{align*}
Finally, we define the complex Monge-Amp{\`e}re operator
\begin{align*}
\mathcal{M}(u) = \log \frac{(\omega+i\partial\bar\partial u)^n}{\omega^n} \in C^0_{loc}(X^{reg})
\end{align*} 
for all $u \in C^2_{loc}(X^{reg})$ such that $\omega + i\partial\bar\partial u > 0$. 

In order to apply the implicit function theorem, we need to check a few basic properties.

$\bullet$ \emph{$\mathcal{U}$ and $\mathcal{F}$ have natural $C^1$ Banach manifold structures.} This is clear for $\mathcal{F}$. For $\mathcal{U}$, notice that the set of all functions of type U can be written as the image of an obvious map
\begin{align*}
\mathcal{S}: C^{k+2,\alpha}_{\nu+2}(X^{reg}) \times \mathcal{P} \times G \to C^{k+2,\alpha}_{\nu+2}(X^{reg}) \oplus \hat\chi  \left(\mathcal{P} \oplus r^2\mathcal{B}\right)\circ P^{-1}. \end{align*}
Here $\mathcal{B}$ denotes any Banach space of functions on the cross-section of $C_x$ such that $\mathcal{B}$ contains the constants and $r^2\mathcal{B}$ is invariant under the pull-back action of $G$ (notice that $r^2 \circ \Phi$ is $2$-homogeneous for all $\Phi \in G$ because $\Phi$ commutes with scaling). It may not be possible to choose $\mathcal{B}$ to be finite-dimensional, hence the $G$-action on $r^2\mathcal{B}$ may never be $C^1$. However, the map $\Phi \mapsto r^2 \circ \Phi$ is clearly $C^1$ if we take $\mathcal{B}$ to be a H\"older or Sobolev space. Then  $\mathcal{S}$ is $C^1$ and the rank theorem \cite[Theorem 2.5.15]{AMR} shows that the image of $\mathcal{S}$ is a $C^1$ submanifold. This theorem applies here because a finite-dimensional subspace of a Banach space is complemented and ${\rm ker}\,d\mathcal{S}|_{(\bar{u},p,\Phi)} = (J\mathfrak{p})\Phi$.

$\bullet$ \emph{$\mathcal{M}$ defines a $C^1$ map from $\mathcal{U}$ to $\mathcal{F}$.} To show that $\mathcal{M}(\mathcal{U})$ is actually contained in $\mathcal{F}$, note that for all $u \in \mathcal{U}$ there exist $\Phi \in G$ and $\bar{u} \in C^{k+2,\alpha}_{\nu+2}(X^{reg})$ such that
\begin{align}\label{e:random42}
P^*(\omega + i\partial\bar\partial u) = \Phi^*\omega_{C_x} + (P^*\omega - \omega_{C_x}) + i\partial\bar\partial(\bar{u} \circ P)
\end{align}
on some sufficiently small neighborhood of the vertex of $C_x$. This holds by the definition of type U. Now $\omega_{C_x}$ is Ricci-flat, so $(\Phi^*\omega_{C_x}^n)/\omega_{C_x}^n = e^h$ for some pluriharmonic function $h$. On the other hand, this quantity is clearly scale-invariant. Thus, $h$ is constant. Using \eqref{e:random42}, it follows that
$$\mathcal{M}(u) \circ P \in C^{k,\alpha}_{\nu}(U) \oplus \R,$$
i.e. $\mathcal{M}(u)$ is of type F. It is clear that $\int_{X^{reg}} (e^{\mathcal{M}(u)} - 1)\omega^n = 0$. Finally, in order for $\mathcal{M}: \mathcal{U} \to \mathcal{F}$ to be $C^1$, we need the Banach space $\mathcal{B}$ used in the previous step to embed into $C^{k+2,\alpha}(Y_x)$, where $Y_x$ is the cross-section of $C_x$. If so, then $\mathcal{M}$ obviously defines a $C^1$ map 
$$C^{k+2,\alpha}_{\nu+2}(X^{reg}) \oplus \hat\chi  \left(\mathcal{P} \oplus r^2\mathcal{B}\right)\circ P^{-1} \to C^{k,\alpha}_{\nu}(X^{reg}) \oplus \hat\chi {C}^{k,\alpha}(Y_x) \circ P^{-1},$$
and $\mathcal{F}$ is a submanifold of the latter space. As usual, $d\mathcal{M}|_u= \frac{1}{2}\Delta_{\omega + i\partial\bar\partial u}$.

$\bullet$ \emph{The differential of $\mathcal{M}$ at $u = 0$ is an isomorphism.} Let $\mathcal{H}$ denote the space of all $\xi$-invariant $2$-homogeneous harmonic functions on $C_x$. It is not hard to see that
\begin{align*}
T_0\mathcal{U} = \left\{u \in  C^{k+2,\alpha}_{\nu+2}(X^{reg}) \oplus \hat{\chi}(\mathcal{P} \oplus \R r^2 \oplus \mathcal{H}) \circ P^{-1}: \int_{X^{reg}}u\omega^n = 0\right\},\\
T_0\mathcal{F} = \left\{f \in  C^{k,\alpha}_\nu(X^{reg}) \oplus\mathbb{R}: \int_{X^{reg}}f\omega^n = 0\right\}.
\end{align*}
Here we use that $Lie(G) = \mathfrak{p} \oplus J\mathfrak{p}$ and $(\nabla h)(r^2) = \langle 2r\phi \partial_r + r^2\nabla\phi, 2r\partial_r\rangle = 4h$ for all $h = r^2\phi \in \mathcal{H}$. Theorem \ref{t:poisson} then tells us that $d\mathcal{M}|_0 = \frac{1}{2}\Delta_\omega$ is an isomorphism from $T_0\mathcal{U}$ to $T_0\mathcal{F}$. Indeed, $\hat\nu = \nu$ because all harmonic functions on $C_x$ with rate in $[0,2)$ are pluriharmonic by Theorem \ref{t:harmonic_structure}(1), so in \eqref{e:almost_harmonic} we may set $Q = \infty$ for all of these functions. Also, all contributions to the $C_x$-harmonic part of \eqref{e:decomp} must already lie in $\hat\chi(\mathcal{P} \oplus \mathcal{H}) \circ P^{-1}$ by Theorem \ref{t:harmonic_structure}(1)(2).

$\bullet$ \emph{The right-hand side of our equation lies in $\mathcal F$.} We are trying to solve $\mathcal{M}(u) = \log(c_\delta e^{\delta F})$. Thus we need to show that $\log(c_\delta e^{\delta F})$ lies in $\mathcal{F}$. But this is clear from the definition of $c_\delta$ and because $F$ is pluriharmonic near $x$, so that $F$ actually lies in $\R \oplus C^\infty_1$ by Lemma \ref{reeb}(2) and Theorem \ref{t:harmonic_structure}(1). It is then also clear that $\log(c_\delta e^{\delta F})$ converges to $0$ in $\mathcal{F}$ as $\delta \to 0$.

The implicit function theorem now tells us that \eqref{e:nearby_eqn} has a solution $\omega_\delta = \omega + i\partial\bar\partial u_\delta$ with $u_\delta \in \mathcal U$ for some $\delta_0 > 0$ and all $\delta \in (-\delta_0, \delta_0)$. (Notice that $\delta_0$ depends on $k,\alpha,\nu$.) It remains to prove that $\omega_\delta$ is conical in the sense of Definition \ref{d:conical} with respect to some local biholomorphism $P_\delta$.

To see this, note that \eqref{e:random42} holds for $u = u_\delta$ with $\Phi = \Phi_\delta \in G$ and $\bar{u} = \bar{u}_\delta \in C^{k+2,\alpha}_{\nu+2}(X^{reg})$. Thus, if we define $P_\delta = P \circ \Phi_\delta^{-1}$, then
$$P_\delta^*\omega_\delta =  \omega_{C_x} + (\Phi_\delta^{-1})^*((P^*\omega - \omega_{C_x}) + i\partial\bar\partial (\bar{u}_\delta \circ P))$$
on some sufficiently small neighborhood of the vertex of $C_x$. Here $P^*\omega - \omega_{C_x} \in C^\infty_{\lambda}$ by assumption and $i\partial\bar\partial(\bar{u}_\delta \circ P) \in C^{k,\alpha}_{\nu}$ by definition. Now for any $\Phi \in G$, the operator $\Phi^*$ preserves these weighted function spaces because $g_{C_x}$ and $\Phi^*g_{C_x}$ are Riemannian cone metrics with the same scaling vector field; cf. \cite[Appendix A]{CH3} for a similar argument. This almost proves the desired decay property of $P_\delta^*\omega_\delta - \omega_{C_x}$ except for the fact that we are currently controlling only $k$ derivatives. Recall here that $k \geq 3n$ was arbitrary but the maximal interval of existence $(-\delta_0, \delta_0)$ depends on $k$.

To overcome this problem, we observe that the function ${u} = \bar{u}_\delta \circ P_\delta \in C^{k+2,\alpha}_{\nu+2}$ satisfies
\begin{align}\label{e:bootstrap}
(\omega_{C_x} + \eta + i\partial\bar\partial {u})^n = e^{f}\omega_{C_x}^n,
\end{align}
where $\eta \in C^\infty_{\lambda}$ is a closed $(1,1)$-form and $f$ is pluriharmonic, hence $f \in C^\infty_1$. Using scaled Schauder estimates on balls of radius $\epsilon r$ at distance $r$ to the vertex (where $\epsilon > 0$ is very small but fixed), the desired property that $u \in C^\infty_{\nu+2}$ now follows by differentiating \eqref{e:bootstrap} and bootstrapping.\end{proof}
\newpage

\section{Closedness}\label{s:closedness}

Now suppose $t_i\in T$ ($T$ is the set defined in Section \ref{s:outline}) and $t_i\rightarrow t_\infty\in [0, 1]$. The goal of this section is to prove that $\omega_{t_\infty}$ is conical at $x$. As explained in Section \ref{s:outline}, the naive hope of proving a generalization of Yau's $C^2$ estimate in this singular setting is obstructed by the unboundedness of the bisectional curvature of each $\omega_{t_i}$ from both sides. Instead we will take an indirect route. In Section \ref{ss:construct_tgt_cone} we first show that the metric completion of $(X\setminus \{x\}, \omega_{t_\infty})$ is homeomorphic to $X$ itself and that there is a unique metric tangent cone at $x$. In Section \ref{section3-2} we show that this tangent cone is the expected one, namely, $C_x$. In Section \ref{section3-3} we prove that $\omega_{t_\infty}$ is polynomially asymptotic to the cone metric $\omega_{C_x}$ in a fixed holomorphic gauge. This means by definition that $\omega_{t_\infty}$ is conical at $x$.

\subsection{Construction of a unique metric tangent cone at $t = t_\infty$}\label{ss:construct_tgt_cone} We first need an appropriate Riemannian convergence theory for the sequence of singular metrics $\omega_{t_i}$. It is very likely that one can directly adapt the Cheeger-Colding theory to this setting, but since in our main theorem we only consider smoothable Calabi-Yau varieties $(X,L)$, we can bypass these technical points using approximation by smooth metrics on a smoothing of $(X, L)$, similar to \cite{SSY}. 

Let  $\pi:(\mathcal X, \mathcal L)\rightarrow\Delta$ be a smoothing of $(X, L)$ with $(X, L)\cong (X_0, L_0)$. By definition, replacing $\mathcal L$ by $\mathcal L^k$ for some $k>0$,  we may embed $\mathcal X$ into $\P^N\times \Delta$ in such a way that  $\mathcal L$ is the restriction of $\O(1)$ over $\P^N$. Denote by $\hat\omega_s\in 2\pi c_1(L_s)$ the restriction of the Fubini-Study metric to $X_s$. Choose a nowhere vanishing relative holomorphic $(n, 0)$-form on $\mathcal X$ and denote  its restriction to $X_s$ by $\Omega_s$. Without loss of generality, we may assume that $\Omega_0$ agrees with the $\Omega$ from Section \ref{s:intro}.

Recall that $F$ is the real part of some holomorphic function $F'$ on a neighborhood of $x$ in $X$. $F'$ extends to a holomorphic function $\mathcal{F}'$ on a neighborhood $\mathcal V$ of $x$ in $\mathcal X$. Let $V_s =\mathcal V \cap X_s$ and note that $V_0$ may be smaller than the $V$ in Theorem \ref{thm0-1}. Then it is easy to find a continuous function $\mathcal F$ on $\mathcal X$ such that $\mathcal F$ is smooth away from $x$, $\mathcal F|_{\mathcal V} = {\rm Re}(\mathcal{F}')$, and $\mathcal{F}|_{X_0} = F$. Define $F_s = \mathcal F|_{X_s}$.

For each $s\in \Delta$ we consider the following continuity path for $\omega_{t, s}\in 2\pi c_1(L_s)$:
\begin{align}\label{e:cont_ext}
\omega_{t, s}^n=c_{t, s} e^{tF_s}i^{n^2}\Omega_s\wedge \bar\Omega_s, \;\,t\in [0, 1].
\end{align}
Again $c_{t, s} >0$ is a constant determined by integrating both sides over $X_{s}$, and $|{\log c_{t, s}}|$ is uniformly bounded by \cite[Theorem B.1]{RZ}. When $s=0$ this is our old path on $X$, so $\omega_{t, 0}=\omega_t$. 
Yau's theorem implies that for all $t\in[0,1]$, $s\neq 0$, there is a unique smooth solution $\omega_{t, s}$ depending smoothly on $(t, s)$.  By \cite{RZ} the Ricci-flat manifolds $(X_s, \omega_{0, s})$ have uniformly bounded diameters, and by \cite{DS1,RZ} they converge in the Gromov-Hausdorff sense to the metric completion of $(X \setminus \{x\}, \omega_0)$ as $s\rightarrow 0$. This completion is actually homeomorphic to $X$ \cite{DS1,jiansong} and we identify it with $(X, \omega_0)$.

Given $D,K>0$ we denote by $G(D,K)$ the set of all $(t, s)$, $t\in [0, 1]$, $0<|s|\leq \frac{1}{2}$, such that 
\begin{equation} \label{eqnIII}
\text{diam}(X_s, \omega_{t, s})\leq D \;\; {\rm and} \;\; K^{-1}\hat\omega_s\leq \omega_{t, s}\leq K\hat\omega_s \;\; {\rm on} \;\; X_s\setminus V_s.
\end{equation}
It is proved in \cite{RZ} that there exist $D_0, K_0$ such that $(0, s)\in G(D_0, K_0)$ for all $0<|s|\leq \frac{1}{2}$. If $D,K$ are given, then ${\rm Ric}(\omega_{t, s})$ is uniformly bounded for all $(t,s)\in G(D,K)$ because $\mathcal F$ is smooth away from $x$ and $\omega_{t, s}$ is Ricci-flat on $V_s$ by \eqref{e:cont_ext}. Thus, given $D,K$, any Gromov-Hausdorff limit of a sequence $(X_{s_i}, \omega_{t_i, s_i})$ with $(t_i, s_i)\in G(D, K)$ is naturally a normal projective variety by \cite{DS1}.

\begin{lem}  \label{lem3-1}
Given a sequence $(t_i, s_i)\in G(D, K)$ such that $s_i\rightarrow 0$ and $t_i\rightarrow t_\infty$,  the following hold.

{\rm (1)} The K\"ahler form $\omega_{t_i, s_i}$ converges smoothly to $\omega_{t_\infty}$ away from the singularity $x$.

{\rm (2)} Any subsequential Gromov-Hausdorff limit of $(X_{s_i}, \omega_{t_i, s_i})$ is naturally isomorphic to $X$ as a  projective variety and is isometric to the metric completion of $(X\setminus \{x\}, \omega_{t_\infty})$.
\end{lem}

\begin{proof}
(1) follows exactly as in \cite[Section 3]{RZ}. The key point is that the bounds on diameter and Ricci curvature provide a uniform Sobolev inequality with respect to $\omega_{t_i,s_i}$, and that $\hat\omega_{s_i}^n$ and $\omega_{t_i,s_i}^n$ can be compared using Hodge theory \cite[Appendix B]{RZ}. Moser iteration then yields a uniform $C^0$ estimate on the relative K\"ahler potential between $\omega_{t_i, s_i}$ and $\hat\omega_{s_i}$, the Chern-Lu inequality yields a $C^2$ bound because ${\rm Ric}(\omega_{t_i,s_i})$ is bounded below and ${\rm Sec}(\hat\omega_{s_i})$ is bounded above, and the rest is standard. (2) follows easily from arguments as in the proof of \cite[Theorem 3.1]{SSY}. The only difference is that \cite{SSY} considers the case of positive K\"ahler-Einstein metrics, where a $C^0$ bound on the potential has to be imposed as an assumption, while in our setting such a bound follows from (1) above.
\end{proof}

\begin{prop} \label{prop3-2}
There are $D,K>0$ such that $(t, s)\in G(D, K)$ for all $t\in [0, 1]$, $0<|s|\leq \frac{1}{2}$. 
\end{prop}

\begin{proof} Recall that there are $D_0, K_0$ such that $(0,s) \in G(D_0, K_0)$ for all $0<|s|\leq \frac{1}{2}$. This was proved in \cite{RZ} (see also \cite{Tos}) using the Ricci-flatness of $(X_s, \omega_{0,s})$ and a uniform volume bound.

We can prove in a similar manner that ${\rm diam}(X \setminus \{x\}, \omega_t) \leq D_0$ for all $t \in [0,1]$, but only \emph{if the set of points added in the completion of $(X\setminus\{x\}, \omega_t)$ has measure zero and if volume comparison holds on the completion}. Indeed, suppose that these hold. The uniform $C^k$ bounds of $\omega_t$ over $X\setminus V_0$ imply that ${\rm diam}(X\setminus V_0, \omega_{t}) \leq D_0$. Now let $p$ be any point in $V_0$ and write $D = d_{\omega_t}(p, \partial V_0)$. Without loss of generality, $D \geq 3$. Let $\gamma$ be a minimal geodesic of length $D$ joining $p$ and $\partial V_0$ in the completion, and let $q$ be the unique point on $\gamma$ such that $r = d(p,q) = D-1$. By volume comparison,
$${\rm Vol}(B_{\omega_{t}}(p, r-1))\geq \frac{1}{(\frac{r+1}{r-1})^{2n}-1} {\rm Vol}(B_{\omega_{t}}(p, r+1)\setminus B_{\omega_{t}}(p, r-1))\geq \frac{r}{C} {\rm Vol}(B_{\omega_{t}}(q, 1))$$
for some dimensional constant $C$. Applying volume comparison once again, 
$${\rm Vol}(B_{\omega_t}(q,1)) \geq \frac{1}{C}{\rm Vol}(B_{\omega_t}(q,D_0+1)) \geq \frac{1}{C}{\rm Vol}(X \setminus V_0) \geq \frac{1}{C},$$
where $C$ now also depends on the uniform $C^k$ bounds of $\omega_t$ over $X\setminus V_0$, which in particular imply uniform Ricci curvature bounds. Since $X \setminus \{x\}$ has bounded volume and since the completion adds no volume by assumption, this yields a uniform upper bound on $r$, hence on $D$. Thus, replacing $D_0$ by $D_0 + C$, the claim that ${\rm diam}(X\setminus\{x\},\omega_t)\leq D_0$ now follows.

It seems hard to eliminate the above technical assumptions about the completion of $\omega_t$ from this argument. In a similar vein, the argument breaks down for $s \neq 0$ because even though $X_s$ is then smooth with ${\rm Ric}(\omega_{t,s})$ compactly supported in $X_s \setminus V_s$, we do not even know a priori that ${\rm Ric}(\omega_{t,s})$ is uniformly bounded with respect to $\omega_{t,s}$. We will overcome these issues arguing by contradiction, using Lemma \ref{lem3-1} (hence \cite{DS1} via \cite{SSY}, see also \cite{jiansong}) and what we have just proved.

Thus, to prove Proposition \ref{prop3-2}, fix $D \geq D_0$ and $K \geq K_0$, and for all $s \in (0,\frac{1}{2}]$ let $t(s)$ denote the largest value of $t \in [0,1]$ such that $(\bar{t},s) \in G(D,K)$ for all $\bar{t} \in [0,t]$.  Notice that $t(s)$ does exist. We may assume that for each $\epsilon \in (0, \frac{1}{2}]$ there exists some $s_\epsilon \in (0,\epsilon]$ such that $t(s_\epsilon) < 1$; indeed, if not, then $[0,1] \times (0,\epsilon_0] \subset G(D,K)$ for some $\epsilon_0 \in (0, \frac{1}{2}]$ and the proposition follows from the compactness of $[0,1] \times [\epsilon_0, \frac{1}{2}]$ by increasing the values of $D,K$ if necessary. Thus, letting $\epsilon = \frac{1}{i}$ for all $i \in \N$, we obtain a sequence $(t_i,s_i) \in [0,1) \times (0, \frac{1}{i}]$ such that $(t,s_i) \in G(D,K)$ for all $t \in [0,t_i]$ and such that at least one of the inequalities of \eqref{eqnIII} is an equality for $(t,s) = (t_i, s_i)$ (because $t_i = t(s_i) < 1$). We can assume that $t_i \to t_\infty \in [0,1]$ and also, by Gromov compactness, that the sequence $(X_{s_i}, \omega_{t_i,s_i})$ has a Gromov-Hausdorff limit. If we assume that equality holds in one of the two $K$-inequalities of \eqref{eqnIII} for infinitely many $i$, then Lemma \ref{lem3-1}(1) gives an immediate contradiction as long as $K > K_0$. Fixing $K > K_0$, we are now forced to conclude that equality holds in the $D$-inequality of \eqref{eqnIII} for infinitely many $i$. Thus,
by Lemma \ref{lem3-1}(2), ${\rm diam}(X \setminus \{x\}, \omega_{t_\infty}) = D$. But Lemma \ref{lem3-1}(2) also implies that the technical assumptions about the completion of $(X\setminus\{x\}, \omega_{t_\infty})$ from our diameter estimate above are indeed satisfied. This leads to a contradiction if we fix $D > D_0$.  \end{proof}

By Proposition \ref{prop3-2} and Lemma \ref{lem3-1}, $(X, \omega_t)$ is the Gromov-Hausdorff limit of $(X_s, \omega_{t, s})$ as $s\rightarrow 0$, 
for any $t \in [0,1]$. In particular, Bishop-Gromov volume comparison holds on $(X, \omega_t)$. Since $\omega_{t,s}$ is Ricci-flat on $V_s$, we can directly apply the results of \cite{DS2} to study the singularity of $\omega_t$ at $x$. 

Now we go back to our original goal. Assume $t_i\in T$ and $t_i\rightarrow t_\infty\in [0, 1]$ as at the beginning of Section \ref{s:closedness}. Then,  by the above discussion, $(X, \omega_{t_\infty})$ is the Gromov-Hausdorff limit of $(X, \omega_{t_i})$. For simplicity we will write $\omega_{t_\infty} = \omega$ in the remainder of Section \ref{s:closedness}. By \cite{DS2} there is a unique tangent cone $C(Y)$ of $\omega$ at $x$, which is an affine algebraic variety with a weak Ricci-flat K\"ahler cone metric $\omega_{C}=\frac{i}{2}\p\bp r^2$. Here $r$ is the distance function to the vertex $o$ in $C(Y)$. In particular, there is a unique holomorphic vector field $\xi$ on $C(Y)$, called the \emph{Reeb vector field}, given by $J(r\p_r)$ on the smooth part of $C(Y)$. $\xi$ generates a holomorphic isometric action of a compact torus $\mathcal T$ as well as a holomorphic action of the complex torus $\mathcal T^{\C}$. These induce a weight decomposition 
$$R(C(Y))=\bigoplus_{d\in \mathcal S} R_d(C(Y))$$
of the coordinate ring of $C(Y)$, where $\mathcal S\subset \R_{\geq 0}$ is a discrete set called the \emph{holomorphic spectrum} of $C(Y)$.  This decomposition is orthogonal with respect to the natural $L^2$-metric over $B_{\omega_C}(o, 1)$. In general for a metric cone $C(Y)$ we define its \emph{volume} by
$${\rm Vol}(C(Y))={\rm Vol}(B(o, 1)). $$
Next, let us recall the algebro-geometric description of  the tangent cone $C(Y)$ to $\omega$ at $x$ from \cite{DS2}. For any non-zero function $f\in \O_x$ we define its \emph{order of vanishing}
\begin{align}\label{e:order}
d_{KE}(f)=\lim_{r\rightarrow 0} \frac{\log \sup_{B_{\omega}(x, r)} |f|}{\log r}.
\end{align}
Then $d_{KE}(f)\in \mathcal S$.  We also make the convention that $d(0)=+\infty$.  List the elements of $\mathcal S$ in order as $0=d_0\leq d_1\leq \cdots$ and for any $k\geq 0$ define an ideal
$$
I_k=\{f\in \mathcal O_x:d_{KE}(f)\geq d_k\}.
$$
We obtain a filtration $\mathcal O_x=I_0\supset I_{1}\supset I_{2}\supset\cdots$ and an associated graded ring
$$R_x=\bigoplus_{k\geq 0} I_k/I_{k+1},$$
where the grading on $I_k/I_{k+1}$ is defined to be $d_k$.  Then $W = {\rm Spec}(R_x)$ is a normal affine algebraic variety, which also admits a natural $\mathcal T^{\C}$-action. Moreover, there exists an equivariant degeneration (test configuration) of $W$ to $C(Y)$ through $\mathcal T^\C$-invariant affine varieties. 

\subsection{The tangent cone at $t = t_\infty$ is the given cone} \label{section3-2}
Recall that the metric $\omega_{t_i}$ is conical, i.e. its tangent cone at $x$ is the given cone $C_x$, with polynomial convergence in a holomorphic gauge. The following theorem represents the first step towards proving that $\omega = \omega_{t_\infty}$ is conical as well.

\begin{thm} \label{thm3-5}
$C(Y)$ is isomorphic to $C_x$ as a Ricci-flat K\"ahler cone, and $W$ is isomorphic to $C_x$ as an affine algebraic cone {\textup (}i.e. as an affine algebraic variety with $\mathcal T^\C$-action{\textup )}.
\end{thm}

\begin{proof}
By definition, we can find two sequences $a_i<b_i$ of positive reals, both converging to $+\infty$, such that $(X, a_i^2\omega_{t_i}, x)$ converges to $C(Y)$ and $(X, b_i^2\omega_{t_i},x)$ converges to $C_x$ in the pointed Gromov-Hausdorff topology. Let $c_i\in [a_i, b_i]$ be any sequence such that $(X,c_i^2\omega_{t_i},x)$ has a pointed Gromov-Hausdorff limit $(Z,p)$. Then, by volume comparison, for all $r > 0$ and $i \gg 1$,
$$
\big(\frac{b_i}{r}\big)^{2n}{\rm Vol}\big(B_{\omega_{t_i}}\big(x,\frac{r}{b_i}\big)\big) \geq \big(\frac{c_i}{r}\big)^{2n}{\rm Vol}\big(B_{\omega_{t_i}}\big(x,\frac{r}{c_i}\big)\big)\geq \big(\frac{a_i}{r}\big)^{2n}{\rm Vol}\big(B_{\omega_{t_i}}\big(x,\frac{r}{a_i}\big)\big).
$$
The Cheeger-Colding volume convergence theorem implies that 
\begin{align}\label{e:bubbletree}
{\rm Vol}(C_x) \geq \frac{1}{r^{2n}}{\rm Vol}(B_Z(p,r))\geq {\rm Vol}(C(Y)).
\end{align}
But ${\rm Vol}(C(Y)) \geq {\rm Vol}(C_x)$ from \cite{LL}; see Appendix \ref{ss:li_liu}. Thus, all inequalities in \eqref{e:bubbletree} are equalities, so $(Z,p)$ is a volume cone with vertex $p$, hence a metric cone by Cheeger-Colding.

Let $\mathcal {LS}$ denote the space of all metric cones $(Z,p)$ arising as limits of sequences $(X, c_i^2\omega_{t_i},x)$ with $c_i \in [a_i, b_i]$ as above. Then each member of $\mathcal {LS}$ is naturally a normal affine algebraic variety by \cite[Section 2.3]{DS2}. Moreover, arguments as in \cite[Section 3.1]{DS2} show that $\mathcal {LS} $ is compact and connected, and that all cones in $\mathcal{LS}$ have the same holomorphic spectrum and Hilbert function. In particular this applies to our two cones $C(Y)$ and $C_x$. Because $C_x$ is strongly regular, the affine algebra $R(C_x)$ is generated by its homogeneous piece of lowest degree with respect to the Reeb vector field of $C_x$, and since $(C_x, o)$ is isomorphic to $(X,x)$, $R(C_x)$ is isomorphic to $\bigoplus_{k=0}^\infty \mathfrak{m}_x^k/\mathfrak{m}_x^{k+1}$ as a graded ring. Here $\mathfrak{m}_x$ denotes the maximal ideal of $\mathcal{O}_x$ and Spec of this ring is the \emph{Zariski tangent cone} of $X$ at $x$. In terms of $C(Y)$ this tells us that $\mathcal{S} = d_1\N_0$ and
\begin{equation} \label{eqnII}
\dim I_k/I_{k+1}=\dim \m_x^k/\m_x^{k+1}. 
\end{equation}
Recall here that $W$ and $C(Y)$ have the same Hilbert function as well. \medskip\

\noindent \emph{Claim 1.}  $I_k=\m_x^k$ for all $k \in \N_0$.\medskip\

\noindent \emph{Proof of Claim 1.} We prove this by induction. This is clear for $k=0$. Suppose it holds for all $k<\ell$. Since $\m_x^\ell =\m_x^{\ell-1}\cdot \m_x$, and $\m_x^{\ell-1}=I_{\ell-1}$, it follows that for every $f\in \m_x^\ell$, $d_{KE}(f)>(\ell-1)d_1$. Thus, $d_{KE}(f)\geq \ell d_1$, so in particular $\m_x^\ell \subset I_\ell$. Therefore by (\ref{eqnII}) we conclude that $I_\ell =\m_x^\ell$. \medskip\

Claim 1 says by definition that $W$ and $C_x$ are isomorphic as affine algebraic varieties. Moreover, the $\mathcal T$-action is the standard $S^1$-action because both $\mathcal{T}$ and $S^1$ act effectively on $R(W) = R(C_x)$ in such a way that the set of weights with respect to each of the two Reeb vector fields is equal to $\mathcal{S}$. Then clearly the Reeb vector field of $C(Y)$ must also generate a $(\mathcal{T}=S^1)$-action.\medskip\

\noindent \emph{Claim 2.} The quotient $C(Y)/\!\!/\mathcal{T}^\C$ has no orbifold singularities in complex codimension $1$.\medskip\

\noindent \emph{Proof of Claim 2.} We have a $\mathcal T$-equivariant test configuration degenerating $W$ to $C(Y)$. If this can be realized inside the Zariski tangent space of $W$, then Claim 2 is obvious because the $\mathcal{T}$-action on $C(Y)$ is then induced by the standard $S^1$-action on this vector space, which has no isotropy except at the origin. In general we must argue differently. If Claim 2 is false, then there exists a polydisk $\Delta^{n-1} \subset C(Y)^{reg}/\mathcal{T}^\C$ such that the ``bad orbifold divisor'' is locally given by $\{0\} \times \Delta^{n-2}$ and such that the total space $T$ of the $\mathcal{T}$-orbibundle over $\Delta^{n-1}$ is equivariantly homeomorphic to $S \times \Delta^{n-2}$ for some solid $3$-torus $S$ with a non-trivial Seifert fibration over $\Delta$. Then $\Delta^{n-1}$ is a smooth limit of polydisks $\Delta_t^{n-1} \subset W_t^{reg}/\mathcal{T}^\C$ as $t \to 0$, where $\{W_t\}_{t\in\C}$ denotes the given degeneration of $W = W_1$ to $C(Y) = W_0$. For $t \neq 0$ the total space $T_t$ of the $\mathcal{T}$-bundle over $\Delta_t^{n-1}$ is equivariantly homeomorphic to $S^1 \times \Delta^{n-1}$ because the $\mathcal{T}$-action on $W_t^{reg}$ is free and $\Delta_t^{n-1}$ is contractible. Since $W_t$ converges to $C(Y)$ without multiplicity, $T_t$ converges to $T$ smoothly without multiplicity. But this is absurd because $S^1 \times \Delta^{n-1}$ and $S \times \Delta^{n-2}$ are not homeomorphic even if we ignore the $S^1$-actions. \medskip\

Taking $\mathcal T^{\C}$-quotients in a $\mathcal T$-equivariant degeneration from $W$ to $C(Y)$ as above, we obtain a test configuration degenerating $W/\!\!/\mathcal T^{\C}$ to $C(Y)/\!\!/\mathcal T^{\C}$ inside some weighted projective space. Now $W/\!\!/\mathcal T^{\C}$ is a smooth Fano manifold by assumption while $C(Y)/\!\!/\mathcal T^{\C}$ may be singular, but Claim 2 says that $C(Y)/\!\!/\mathcal T^{\C}$ is a \emph{$\Q$-Fano variety}, together with a weak K\"ahler-Einstein metric induced by the weak Ricci-flat K\"ahler cone metric on $C(Y)$. By \cite{Berman}, $W/\!\!/\mathcal T^{\C}$ and $C(Y)/\!\!/\mathcal T^{\C}$ are isomorphic, so $W$ and $C(Y)$ are isomorphic as affine algebraic cones. Then $C_x$ and $C(Y)$ are also isomorphic as Ricci-flat K\"ahler cones by the uniqueness of K\"ahler-Einstein metrics on a Fano manifold \cite{BM}. \end{proof}

\subsection{Polynomial convergence in a holomorphic gauge} \label{section3-3} We have proved that the tangent cone $C(Y)$ to $\omega$ at $x$ is the given cone $C_x$, which is locally analytically isomorphic to the germ $(X,x)$. As $C_x$ is strongly regular, the holomorphic functions in $R_{d_1}(C(Y))$ define an embedding of $C(Y)$ as an affine cone into some $\C^N$ such that the Reeb vector field of $C(Y)$ is the restriction of the linear field $2d_1 {\rm Re}(i\sum_{j=1}^N z_j\p_{z_j})$. To show that $\omega$ is conical at $x$ amounts to proving the following.

\begin{thm} \label{thm1}
There is  a complex analytic isomorphism $P:  U\rightarrow V$ between open neighborhoods $U$ of the vertex $o$ in $C(Y)$ and $V$ of $x$ in $X$ such that 
\begin{align}\label{e:conical_thm1}
\sup_{\partial B(o, r)} |\nabla_{\omega_C}^j(P^*\omega-\omega_C)|_{\omega_C}\leq C_j r^{d-j}
\end{align}
for some $d> 0$ and all $j\in\N_0$.
\end{thm}

\begin{rmk}
In the real Einstein setting, the results of Cheeger-Tian \cite{CT} and Colding-Minicozzi \cite{CM} say that if some tangent cone has a smooth link, then one has logarithmic convergence to this cone under a suitable gauge, and if the cone is integrable, then this can be improved to polynomial convergence. In our setting, $(X,x)$ and $(C(Y),o)$ are a priori biholomorphic, but even so it is not a consequence of \cite{CT, CM} that the gauge $P$ in \eqref{e:conical_thm1} can be taken to be holomorphic. Rather, we will use the given biholomorphism to rework the ideas of \cite{CT}. This has the added benefit  that we get polynomial convergence (which is crucial for Section \ref{s:openness}) without having to assume integrability.
\end{rmk}

The proof of Theorem \ref{thm1} takes up the rest of this section.

\subsubsection{Uniform convergence in a broken holomorphic gauge}\label{ss:broken_gauge} We first briefly recall some results from \cite{DS2}. Fix $\lambda\in (0, 1)$. Let $B_i$ denote the unit ball around $x$ in the rescaled metric $\lambda^{-2i}\omega$, which can be identified with the ball $B(x, \lambda^{i})$ in $(X, \omega)$. Let $B_\infty$ be the unit ball around the vertex in $C(Y)$. We have natural maps $\Lambda_i: B_i\rightarrow B_{i-1}$. Fix a metric $\mathbf{d}_i$ on the disjoint union of $B_i$ and $B_\infty$ that realizes the Gromov-Hausdorff convergence of $B_i$ to $B_\infty$. Fix an algebraic embedding $F_\infty$ of $C(Y)$ into $\C^N$ using an $L^2$-orthonormal basis of $R_{d_1}(C(Y))$. Let $G$ be the group of holomorphic automorphisms of $C(Y)$ that commute with radial dilation and let $K$ be the subgroup consisting of isometries. By the definition of $F_\infty$, $G$ may be identified with the subgroup of $GL(N,\C)$ that leaves $C(Y)$ invariant, and $K=G\cap U(N)$. Notice that $Lie(G) = Lie(K)^\C$ by Theorem \ref{t:harmonic_structure}(3).

In \cite[Section 3.3]{DS2} by constructing an ``adapted sequence of bases'' it is proved that for $i$ large there exist holomorphic embeddings $F_i: B_i\rightarrow \C^N$ that converge to $F_\infty$ under the above Gromov-Hausdorff convergence. Moreover, we can arrange that $F_{i-1} \circ \Lambda_i \circ (F_i)^{-1}$ extends to a linear map on $\C^N$, which we denote by $\Lambda_i$ also, and that $\Lambda_i$ converges to $\Lambda=\lambda^{d_1}Id$.  Let $W_i$ denote the weighted tangent cone of $F_i(B_i)$ at $0 \in \C^N$ with respect to the weight $(d_1, \ldots, d_1)$. By \cite{DS2}, $W_i$ is isomorphic to $W$. Since the $\Lambda_i$ commute with $\Lambda$, we have that $W_{i-1}=\Lambda_i(W_i)$. Moreover, the $W_i$ converge to $C(Y)$ as normal affine cones in $\C^N$ (in a certain ``multi-graded Hilbert scheme''  $\Hilb$).

\begin{lem} \label{lem3}
In our case, we may assume that $\Lambda_i\in G$ and that $W_i$ is equal to $C(Y)$ for all $i$.
\end{lem}

\begin{proof}
The key point is that, in our situation, $W_i,W$ are isomorphic to $C(Y)$. Thus they lie in the $GL(N,\C)$-orbit of $C(Y)$ in $\Hilb$, so there exist $g_i\in GL(N,\C)$ with $g_i \to Id$ and $g_i(W_i)=C(Y)$. Replacing $F_i$ with $g_i \circ F_i$, we can therefore arrange that $W_i=C(Y)$ and $\Lambda_i\in G$.
\end{proof}

We are now ready to begin proving Lemma \ref{lem6}. This essentially says that there is a holomorphic embedding $P_i$ of $A_i = \{\lambda^{i+1} < r < \lambda^i\} \subset C(Y)$ into $X$ such that $\lim_{i\to\infty}\sup_{A_i} |P_i^*\omega - \omega_{C}|_{\omega_C} = 0$. The main point of the proof is to improve our abstract biholomorphism
$(X,x)$ $\cong$ $(C(Y),o)$ to one that under the above embeddings $F_i, F_\infty$ becomes the identity plus higher order terms.

More precisely, we can assume that there exists a holomorphic embedding $\Phi$ of $B_1$ onto a domain in $C(Y)$. Then $F_\infty \circ \Phi \circ (F_1)^{-1}$ is an embedding of $F_1(B_1)$ onto a domain in $C(Y) \subset \C^N$. As such, it naturally induces an isomorphism of the associated Zariski tangent cones at $0 \in \C^N$, both of which are equal to\footnote{At this point we are using that $(C_x, \omega_{C_x})$ is strongly regular, i.e. that the holomorphic functions of lowest degree with respect to the $S^1$-action generated by the Reeb vector field already define an affine embedding of $C_x$. However, compared to the proof of Theorem \ref{thm3-5} this seems like a relatively mild use of the strongly regular property.} $C(Y)$ by Lemma \ref{lem3}. By construction, this automorphism of $C(Y)$ commutes with the given $\C^*$-action on $C(Y)$, hence defines an element $g \in G$. By replacing $\Phi$ by $g^{-1} \circ \Phi$, we may thus assume that $F_\infty \circ \Phi \circ (F_1)^{-1}$ induces the identity map on Zariski tangent cones. By the definition of a holomorphic map between singular complex spaces, we can therefore extend $F_\infty \circ \Phi \circ (F_1)^{-1}$ (after shrinking its domain if need be) to a holomorphic map $\tilde\Phi$ on a small polydisk $\mathbb{D} \subset \C^N$ with $\tilde\Phi(z) = z + O(|z|^2)$ for all $z \in \mathbb D$. Thus we have a constant $C$ such that for all $x\in B_1$, 
\begin{equation} \label{eqn1}
|F_1(x)-F_\infty(\Phi(x))|\leq C|F_1(x)|^2.
\end{equation}

Define $\Phi_i = \Phi|_{B_i} =\Phi \circ \Lambda_2 \circ \cdots \circ \Lambda_{i-1} \circ \Lambda_i$. Also, 
define $g_i=(\Lambda_i)^{-1}\circ (\Lambda_{i-1})^{-1}\circ\cdots \circ (\Lambda_2)^{-1}$ as a linear map on $\C^N$, so that $F_i=g_i\circ F_1|_{B_i} = g_i \circ F_1 \circ \Lambda_2 \circ \cdots \circ \Lambda_{i-1} \circ \Lambda_i$. 

\begin{lem} \label{lem4}
For any $\delta \in (0,1)$ there is a constant $C_\delta$ such that for all $i$ and all $x\in B_i$,
\begin{align}|F_i(x)-F_\infty(g_i(\Phi_i(x)))|\leq C_\delta\lambda^{(1-\delta)d_1 i}.\end{align}
\end{lem}

\begin{proof}
Since $\Lambda_i$ converges to $\Lambda=\lambda ^{d_1}Id$, for any $\delta\in (0,1)$ we can find a constant $C_\delta$ such that 
$$C_\delta^{-1}\lambda^{-(1-\delta)d_1i}\leq |g_i|\leq C_\delta \lambda^{-(1+\delta)d_1i}$$
for all $i$.
Using \eqref{eqn1} it follows that for all $i$ and all $x\in B_i$, 
\begin{align*}
|F_i(x)-F_\infty(g_i(\Phi_i(x)))|&=|g_i(F_1(x)-F_\infty(\Phi(x)))| \\
&\leq C_\delta \lambda^{-(1+\delta)d_1i}|F_1(x)|^2 \leq C_\delta^2 \lambda^{-(1+\delta)d_1i + 2(1-\delta)d_1 i}|F_i(x)|^2.
\end{align*}
Now notice that $|F_i(x)| \leq C$ because $F_i$ Gromov-Hausdorff converges to $F_\infty$.
\end{proof}

The previous lemma says that $F_\infty \circ (g_i\circ\Phi_i) \circ (F_i)^{-1}$ realizes the Hausdorff convergence of $F_i(B_i)$ to $F_\infty(B_\infty)$ in $\C^N$. The next lemma says that $g_i \circ \Phi_i$ realizes the Gromov-Hausdorff convergence of $B_i$ to $B_\infty$ with respect to the metric $\mathbf{d}_i$ we fixed in the beginning.

\begin{lem} \label{lem5}
We have that
\begin{align}\lim_{i\to\infty}\sup_{x \in B_i} \mathbf{d}_i(x, g_i(\Phi_i(x))) = 0.\end{align}
\end{lem}

\begin{proof}
 If not, then we can find $\epsilon>0$ and $x_i\in B_i$ such that $\mathbf{d}_i(x_i, g_i(\Phi_i(x_i)))\geq \epsilon$ for infinitely many indices $i$. Passing to a further subsequence, we may find $x\neq y$ in $B_\infty$ such that $\mathbf{d}_i(x_i, x)\rightarrow 0$ and $\mathbf{d}_i(g_i (\Phi_i(x_i)), y)\rightarrow 0$. Correspondingly we have that $F_i(x_i)\rightarrow F_\infty(x)$ and $F_\infty(g_i(\Phi_i(x_i)))\rightarrow F_\infty(y)$. Lemma \ref{lem4} shows that $F_\infty(x)=F_\infty(y)$. This is a contradiction to $F_\infty$ being an embedding.
\end{proof}
 
For simplicity of notation we write $\Psi_i=(g_i \circ \Phi_i)^{-1}$. Then $\Psi_i$ is a biholomorphism from a certain domain in $C(Y)$ onto $B_i$. Notice that Lemma \ref{lem5} implies that for $i$ large enough, the domain of $\Psi_i$ contains a ball of any radius less than $1$ around the vertex in $C(Y)$. Also, for all $k < \ell$ define the annulus $A_{k,\ell} =  B(o, \lambda^k)\setminus B(o, \lambda^{\ell}) \subset C(Y)$. Finally, define $\omega_i=\lambda^{-2i}\Psi_i^*\omega$. 

\begin{lem}\label{lem6} Given any two fixed integers $j\geq 0$ and $\ell \geq 3$, we have that
$$\lim_{i\rightarrow\infty} \sup_{A_{2, \ell}}|\nabla^j_{\omega_C}(\omega_i-\omega_C)|_{\omega_C}=0.$$
\end{lem}

\begin{proof}
By Cheeger-Colding theory we may find a diffeomorphism ${\Psi}_i'$
 from a neighborhood of  $A_{1, \ell+1}$ into $B_i$ such that $({\Psi}_i')^*(J, \lambda^{-2i}\omega)$ converges smoothly to $(J_C, \omega_C)$ and $\sup_{x\in A_{1, \ell+1}} \mathbf{d}_i(x, \Psi_i'(x))\rightarrow 0$. 
Thus, $\sup_{x\in A_{1, \ell+1}} \mathbf{d}_i(\Psi_i(x), \Psi_i'(x))\rightarrow 0$, and it remains to improve this to smooth convergence. To this end, notice that $\Psi_i$ is holomorphic and $\Psi_i'$ is almost holomorphic in a smooth sense. Also, $\Psi_i$ and $\Psi_i'$ are uniformly close to each other, so if $B \subset A_{1,\ell+1}$ is any sufficiently small ball, then both $\Psi_i(B)$ and $\Psi_i'(B)$ will be contained in the same small coordinate ball of $X^{reg}$. Smooth closeness of $\Psi_i$ and $\Psi_i'$ now follows from uniform closeness and standard interior Schauder estimates for the (overdetermined) elliptic operator $\bar\partial$ acting on scalar functions on a ball in $\C^n$.
\end{proof}

Lemma \ref{lem6} has the following useful consequence.

\begin{lem} \label{lem7}
We have that 
\begin{align}
\lim_{i\rightarrow \infty}\sup_{A_{2,\infty}}|\L_{r\p_r}\omega_i-2\omega_i|_{\omega_i}=0.
\end{align}
\end{lem}

\begin{proof}
If $i$ is sufficiently large, then for every $x \in A_{2,\infty}$ there exists some $k=k(i,x) \geq i$ such that $(\Psi_k)^{-1}(\Psi_i(x)) \in A_{1,3}$. Observe that $(\Psi_k)^{-1}(\Psi_i(y)) = 
g_k(g_i^{-1}(y))$ for all $y \in A_{2,\infty}$ for which the left-hand side is defined (e.g. all $y \in A_{2,\infty}$ with $d(o,y) \leq (1+\epsilon)d(o,x)$ for some small $\epsilon = \epsilon(i,x)> 0$). 
Since $g_k \circ g_i^{-1}$ is an element of $G$, hence preserves the vector field $r\p_r$, it follows that 
$$|\L_{r\p_r}\omega_i-2\omega_i|_{\omega_i}(x)=|\L_{r\p_r}\omega_k-2\omega_k|_{\omega_k}((\Psi_k)^{-1}(\Psi_i(x))).$$
Then the claim is immediate from Lemma \ref{lem6} and the fact that $\L_{r\p_r}\omega_C -2\omega_C = 0$.
\end{proof}

Lemmas \ref{lem6} and \ref{lem7} provide a ``broken holomorphic gauge'' with respect to which $\omega$ converges to $\omega_C$ uniformly as $r \to 0$. Following \cite{CT}, some basic properties of the linearized Ricci-flat equation will allow us to upgrade this to polynomial convergence in a fixed holomorphic gauge. Notice that the gauges constructed in \cite{CT} are not holomorphic, and \cite{CT} also need their cones to be integrable. Since here we have more information to begin with (namely, that there is a unique tangent cone by \cite{DS2}, and that annuli in this cone are biholomorphic to annuli in the original space), we are able to overcome both of these issues for the same reason: that, for us, it suffices to consider the linearized Ricci-flat equation only on $J$-invariant (rather than arbitrary) symmetric $2$-tensors.

\subsubsection{The linearized Ricci-flat equation on $(1,1)$-forms}\label{ss:linearizedRF} As we already know from \cite{DePa, EGZ} or \cite{DS1, DS2}, $\omega$ is $i\partial\bar\partial$-exact on $B_i \setminus \{x\} \subset X$ for all large $i$, so $\omega_i$ is $i\p\bp$-exact on $B(o,1-\epsilon)\setminus \{o\}$ $\subset$ $C(Y)$, where $\epsilon \to 0$ as $i\to \infty$. Thus the linearization of the equation ${\rm Ric}(\omega_i) = 0$ at $\omega_C$ is given by 
$$L(\beta)=-i\p\bp {\rm tr}_{\omega_C}\beta$$ for an $i\p\bp$-exact real $(1, 1)$-form $\beta$. By the Hodge-K\"ahler identities with respect to $\omega_C$,
$$L(\beta)=\frac{1}{2}dd^*\beta.$$
 This together with  the obvious equation $d\beta=0$ gives an elliptic system for $\beta$. 

For non-negative integers $k<\ell$ we introduce a weighted $L^2$-norm $\|\beta\|_{k,\ell}$ via
\begin{align*}
\|\beta\|_{k, \ell}^2 =\sum_{p = k}^{\ell-1} \int_{A_{p, p+1}} \lambda^{-2np} |\beta|_{\omega_C}^2.
\end{align*}
Later it will also be convenient to write
\begin{align*}
\|\nabla\beta\|_{k,\ell}^2 = \sum_{p=k}^{\ell-1} \int_{A_{p,p+1}}\lambda^{(2-2n)p} |\nabla\beta|_{\omega_C}^2.
\end{align*}
Finally, we define the standard unweighted $L^\infty$-norm
$$\|\beta\|_{k, \ell; \infty}=\sup_{A_{k, \ell}}|\beta|_{\omega_C}.$$
As in Section \ref{s:harmonic}, we say that a $2$-form $\beta$ defined  on $A_{k, \ell}$  is \emph{$d$-homogeneous} if $\L_{r\p_r}\beta=d\beta$. Then $\beta$ automatically extends to a $d$-homogeneous $2$-form on $C(Y)\setminus\{o\}$ with $|\beta|_{\omega_C} \sim r^{d-2}$. 

\begin{lem}\label{lem8}
There is a discrete set $\Sigma \subset \mathbb R$ with the following property. If $k,\ell$ are given and $\beta$ is an $i\p\bp$-exact real $(1, 1)$-form on $A_{k,\ell}$ with $L(\beta)=0$, then we have a smoothly convergent expansion $\beta=\sum_{d\in \Sigma} \beta_d$, where each $\beta_d$ is 
$d$-homogeneous. This decomposition is orthogonal in $L^2$ on all level sets of the radius function. Moreover, if $\|\beta\|_{k,\ell}<\infty$, then it converges in $L^2$.
\end{lem}

\begin{proof}
By definition $\beta=i\p\bp u$ for some real-valued function $u$, and ${\rm tr}_{\omega_C}\beta = \Delta u$ is pluriharmonic. In particular, $\Delta u$ is harmonic, so there exists a smooth expansion
$\Delta u=\sum_{d\in \Sigma'} v_d$, where $\Sigma' \subset \R$ is discrete and $v_d$ is a $d$-homogeneous harmonic function. Then $i\partial\bar\partial v_d$ is a $d$-homogeneous $2$-form, so every component $v_d$ is itself pluriharmonic. By Lemma \ref{reeb}, $v_d = 0$ for $d < 0$. Thus, using the fact that $\Delta(r^2v_d) = (4d + 2n)v_d$, we obtain that
$$u = \sum_{d \in \Sigma' \cap \R_{\geq 0}} (4d+2n)^{-1}r^2v_d + \sum_{d \in \Sigma'} u_d,$$
where each $u_d$ is a $d$-homogeneous harmonic function. By taking $i\partial\bar\partial$ of this expansion, we obtain a decomposition $\beta = \sum_{d\in\Sigma}\beta_d$ with $\Sigma = (2 + (\Sigma' \cap \R_{\geq 0})) \cup \Sigma'$. More precisely, by construction, 
\begin{align}\label{e:fine_expansion}\beta = \sum_{i=0}^\infty \sum_{j=0}^{j_i} (\beta_{i,j}^+ + \beta_{i,j}^-), \;\,\beta_{i,j}^+ = i\partial\bar\partial(r^{d_{i,j}^+}\phi_{i,j}), 
\;\,\beta_{i,j}^- = i\partial\bar\partial(r^{d_{i}^-}\phi_{i,j}),\end{align}
where $i$ indexes the eigenvalues of $\Delta_Y$ without multiplicity,  \begin{small}$d_{i}^-$\end{small}  is the associated negative harmonic growth rate,  $\phi_{i,j}$  is an $i$-eigenfunction of $\Delta_Y$,  $\phi_{i,j}\perp\phi_{i,j'}$  in $L^2$ for $j \neq j'$,  \begin{small}$d_{i,j}^+$\end{small} $ =$ $2-2n$ $-$ \begin{small}$d_{i}^-$\end{small} if the corresponding harmonic function of positive rate is not pluriharmonic, and  \begin{small}$d_{i,j}^+$\end{small} $=$ $ 4-2n$ $-$  \begin{small}$d_{i}^-$\end{small}
 if it is. In particular,  \begin{small}$d_{i}^-$\end{small} $\leq$ $2-2n$, and \begin{small}$d_{i,j}^+$\end{small} $\geq 2$ by Theorem \ref{t:harmonic_structure}. Thus,  \begin{small}$|\beta^{+}_{i,j}|$\end{small}  grows (strictly for all but finitely many $i$),  \begin{small}$|\beta^-_{i,j}|$\end{small} strictly decays, and one expects that  \begin{small}$\beta_{i,j}^\pm$\end{small} $\perp$  \begin{small}$\beta_{i',j'}^\pm$\end{small}  unless $i = i'$ and $j = j'$.

This orthogonality holds, and this suffices for the proof of Theorem \ref{thm1}; cf. Remark \ref{r:leon}. More surprisingly, it turns out that  \begin{small}$\beta_{i,j}^+$\end{small} $\perp$ \begin{small}$\beta_{i,j}^-$\end{small}  as well. Given this, we will not need to remember \eqref{e:fine_expansion}, and many steps below simplify. We defer the proof of these properties to Lemma \ref{lem-orth}.\end{proof}

\begin{lem}\label{lem-orth} The following hold on every K\"ahler cone $C = C(Y)$.

{\rm (1)} If $\phi \perp \psi$ are eigenfunctions of $\Delta_Y$, then $\int_Y \langle dd^c(r^d\phi), dd^c(r^e\psi) \rangle = 0$ for all $d,e \in \R$.  

{\rm (2)} If $\phi$ is an eigenfunction of $\Delta_Y$, and if $u^\pm = r^{d^\pm}\phi$ are the associated homogeneous harmonic functions on $C$, then $\int_Y \langle dd^c u^+, dd^c u^- \rangle = 0$ and $\int_Y \langle dd^c(r^2 u^+), dd^c u^- \rangle = 0$.
\end{lem}

\begin{proof}
Let $f = f(r)$ be a function of compact support in $\R_{>0}$. Then for all functions $u,v$ on $C$,
\begin{align}\label{e:ibp42}
\int_{C} \langle dd^c u, dd^c v\rangle f = \int_{C} (\Delta u)(\Delta v)f + \int_{C} (\partial_r u)(\Delta v)f' - \int_{C} \langle d^c u, \partial_r \,\lrcorner\,dd^c v\rangle f'.
\end{align}
One proves this by moving one $d$ to the other side, using the Hodge-K\"ahler identity $d^*dd^c = d^cd^*d$ and the obvious identity $\langle d^c\alpha, d^c\beta\rangle = \langle d\alpha,d\beta\rangle$, and moving the outer $d$ back. Now
\begin{align*}
d^c u = (\partial_r u) d^c r - (J\partial_r u)dr + d^c_b u,\\
dd^c v = (\partial_r^2 v)dr \wedge d^c r + (\partial_r v)dd^c r - d(J\partial_r v) \wedge dr + dd^c_b v,\\
\partial_r \,\lrcorner\,dd^c v = (\partial_r^2v + r^{-1}\partial_r v + (J\partial_r)^2v)d^c r+ d_b(J\partial_r v) + \mathcal{L}_{\partial_r }(d^c_b v),
\end{align*}
where a subscript $b$ indicates differentiation in the ``basic'' directions perpendicular to $\partial_r, J\partial_r$. 

To proceed, we require two auxiliary formulas valid for any two functions $\phi,\psi$ on $Y$: 
\begin{align*}
\int_Y \langle d_b \phi, d_b\psi\rangle = \int_Y (d_b^*d_b\phi)\psi = \int_Y (-\Delta_Y\phi + (J\partial_r)^2\phi)\psi,\\
\int_Y \langle d^c_b\phi, d_b\psi\rangle = \int_Y (d_b^*d^c_b\phi)\psi =(2-2n)\int_Y (J\partial_r\phi)\psi.
\end{align*}
These can be proved in several ways, e.g. by completing $d_b, d_b^c$ to $d, d^c$, extending $\phi,\psi$ to be radially constant from $Y$ to $C$, and integrating by parts against a test function on $C$ as in \eqref{e:ibp42}.\footnote{The above also shows that the Hodge-K\"ahler identity $d^*_bd^c_b + d^c_bd^*_b = 0$ is false if we do not restrict to \emph{basic} forms, i.e. $J\partial_r$-horizontal and  $J\partial_r$-invariant ones. Similarly, $d_b^*d_b\phi = 2\bar\partial_b^*\bar\partial_b\phi + (2-2n)i(J\partial_r\phi)$. See also \cite[Remark 5.15]{CH1}. }

We now set $u = r^d\phi, v = r^e\psi$, where $\phi,\psi$ are eigenfunctions of $\Delta_Y$. We can assume that each of $\phi, \psi$ is either $(J\partial_r)$-invariant or belongs to a pair as in Lemma \ref{reeb}(1). Then the lemma follows by direct computation. For the second claim of (2), one needs to use the fact that the radial integral in the third term of \eqref{e:ibp42} reduces to $\int_0^\infty r^{(d^+ + 2)-1}r^{d^- - 2} f'(r)r^{2n-1} dr = 0$.
\end{proof}

Let $\H_2$ denote the space of all $2$-homogeneous $i\p\bp$-exact real $(1, 1)$-forms on $C(Y)$ with $L(\beta)=0$. The elements of $\mathcal{H}_2$ have a special geometric meaning. Recall the group $K$ from Section \ref{ss:broken_gauge}.

\begin{lem} \label{lem9}
There is an isomorphism of vector spaces $Lie(K) \to \H_2$ given by $V\mapsto \L_{JV}\omega_C$.
\end{lem}

\begin{proof}
Notice that $Lie(K)$ agrees with the space $J\mathfrak p$ of Theorem \ref{t:harmonic_structure}. Thus, for any $V\in Lie(K)$ we can write $JV=\nabla u+a r\p_r$ with a $\xi$-invariant $2$-homogeneous harmonic function $u$ and a constant $a\in \R$. Then $\L_{JV}\omega_C=2i\partial\bar\partial u+2a\omega_C$. This has constant trace with respect to $\omega_C$, hence lies in the kernel of $L$. Conversely, if $\beta\in \H_2$, then $\beta=i\p\bp\phi$ for a $2$-homogeneous function $\phi$ by Lemma \ref{lem8}. Since ${\rm tr}_{\omega_C}\beta = \Delta \phi$ is pluriharmonic and $0$-homogeneous, it must be constant. Thus, $\phi=ar^2+\psi$ for some $a\in \R$ and some $2$-homogeneous harmonic function $\psi$. By Theorem \ref{t:harmonic_structure}, $\psi=u_1+u_2$, where $u_1$ is pluriharmonic, $u_2$ is $\xi$-invariant, and $\nabla u_2$ is holomorphic. Thus, $\beta=i\p\bp u_2+2a\omega_C$. It follows that $\beta=\L_{JV}\omega_C$, where $JV=\frac{1}{2}\nabla u_2 + ar\p_r \in \mathfrak{p}$, or in other words $V \in J\mathfrak{p} = Lie(K)$.
\end{proof}

Also recall the group $G$ from Section \ref{ss:broken_gauge}. Fix a Riemannian distance function $d_G$ on $G$.

\begin{lem} \label{lem10}
There exist constants $\epsilon_1\in (0, 1)$ and $M_1>1$ depending only on $\lambda$ such that for all $k\in\N_0$ and all $2$-forms $\beta$ on $A_{k, k+3}$ with $\|\beta\|_{k, k+3} + \|\nabla\beta\|_{k,k+3} \leq\epsilon_1$, there exists some $g\in G$ with $d_G(Id,g)\leq M_1\|\beta\|_{k, k+3}$ such that $g^*(\omega_C+\beta)-\omega_C$ is $L^2$-orthogonal to $\H_2$ over $A_{1, 2}$. 
\end{lem}

\begin{proof}
By scaling invariance, we can assume that $k = 0$. Let $\mathbf{P}$ denote the $L^2$-orthogonal projection onto $\mathcal{H}_2$ over $A_{1,2}$. Define an open neighborhood $U$ of $K$ in $G$ by the condition that $g \in G$ lies in $U$ if and only if $g(\overline{A_{1,2}}) \subset A_{0,3}$. Define a map
$$F: U \to \mathcal{H}_2, \;\,F(g) = \mathbf{P}(g^*(\omega_C + \beta) - \omega_C).$$
Then $|F(Id)| \leq \|\beta\|_{0,3}$. Writing $\mathbf{P}$ in terms of a basis of $\mathcal{H}_2$, which is a finite-dimensional space of smooth harmonic forms, one checks that $F \in C^1$ if $\beta \in L^2_1$. If $\beta$ is actually small in $L^2_1$, then $dF|_{Id}$ is a small perturbation of the projection $Lie(G) = \mathfrak{p} \oplus J\mathfrak{p} $ $\to$ $\mathfrak{p}$ by Lemma \ref{lem9} and Theorem \ref{t:harmonic_structure}, hence is surjective. Similarly, $dF$ is bounded on $U$, with a bounded right inverse. The lemma then follows from the constant rank theorem in a standard manner.
\end{proof}

\begin{lem} \label{lem11}
There exist constants $\epsilon_2\in (0, 1)$ and $M_2>1$ depending only on $\lambda$ such that for all $g \in G$ with $d_G(Id,g)\leq \epsilon_2$, all $k \in \N_0$, and all $2$-forms $\beta$ on $A_{k,k+3}$,
\begin{align}
\label{lem11eq1}\|g^*\omega_C-\omega_C\|_{k,k+1;\infty} &\leq M_2 d_G(Id,g),\\
\label{lem11eq2}\|g^*\beta\|_{k+1, k+2; \infty }&\leq (1+M_2 d_G(Id,g)) \|\beta\|_{k, k+3; \infty},\\
\label{lem11eq3}\|\beta\|_{k+1, k+2; \infty}&\leq (1+M_2 d_G(Id,g)) \|g^*\beta\|_{k, k+3; \infty}.
\end{align}
Moreover, all of these inequalities hold verbatim if we drop all $\infty$ subscripts.
\end{lem}
\begin{proof}
Using the fact that the action of $g$ on $C(Y)$ commutes with radial dilation, one reduces this to the case $k= 0$, which is easy to check directly. For \eqref{lem11eq1} one also uses that $\nabla\omega_C = 0$.
\end{proof}

Fix a real number $\sigma > 0$ less or equal than $\min_{d\in \Sigma\setminus\{2\}}|d-2|$.

\begin{lem} \label{lem12}
Let $\beta$ be an $i\p\bp$-exact real $(1, 1)$-form on $A_{k, k+3}$ with $\|\beta\|_{k, k+3}<\infty$ and  $L(\beta)=0$.

{\rm (1)} We have the three-circles inequality
\begin{equation} \label{eqn4}
\|\beta\|_{k+1, k+2}^2\leq \|\beta\|_{k,k+1}\|\beta\|_{k+2, k+3},
\end{equation}
with equality if and only if $\beta$ is $d$-homogeneous for some $d\in \Sigma$. 

{\rm (2)} If $\beta$ is $L^2$-orthogonal to $\H_2$ over $A_{k+1, k+2}$, then exactly one of the following is true:
\begin{align}
\|\beta\|_{k+2, k+3}\geq \frac{1}{\sqrt{2}}\lambda^{-\sigma} \|\beta\|_{k+1,k+2},\\
\|\beta\|_{k, k+1}> \frac{1}{\sqrt{2}}\lambda^{-\sigma} \|\beta\|_{k+1,k+2}.
\end{align}
\end{lem}

\begin{proof}
By  Lemma \ref{lem8} we have an $L^2$-orthogonal decomposition $\beta=\sum_{d\in \Sigma} \beta_d$ with $\beta_d$ homogeneous
of degree $d$. Writing $p_d=\|\beta_d\|^2_{k+1, k+2}$, we then have for all $\ell \in \Z$ with $k+\ell \geq 0$ that
$$\|\beta\|^2_{k+\ell, k+\ell+1}=\sum_{d\in \Sigma} p_d\lambda^{2(d-2)(\ell-1)}.$$
 $(1)$ now follows easily by applying the Cauchy-Schwarz inequality. To show (2) we first consider the case that $\sum_{d<2}p_d\geq \sum_{d>2}p_d$. In this case,
$$\|\beta\|_{k+2, k+3}^2\geq \lambda^{-2\sigma}\sum_{d<2} p_d\geq \frac{1}{2} \lambda^{-2\sigma} \|\beta\|_{k+1, k+2}^2.$$ The remaining case is similar. 
\end{proof}

\begin{lem} \label{lem14}
There exist $\epsilon_3 \in (0,1)$ and $\epsilon_3(d) \in (0, 1)$ $(d \not\in\Sigma)$ depending only on $\lambda$ such that for all $i\p\bp$-exact real $(1, 1)$-forms $\beta$ on $A_{k, k+3}$ with $\omega_C+\beta>0$, ${\rm Ric}(\omega_C+\beta)=0$, the following hold.

{\rm (1)} If $\|\beta\|_{k, k+3; \infty}\leq \epsilon_3(d)$, then
\begin{align}
\label{e:keep_growing}\|\beta\|_{k+1, k+2}\geq \lambda^{-(d-2)} \|\beta\|_{k,k+1} \;&\Longrightarrow\;\|\beta\|_{k+2,k+3}\geq \lambda^{-(d-2)} \|\beta\|_{k+1, k+2},\\
\label{e:has_been_decaying}\|\beta\|_{k+2, k+3}\leq \lambda^{d-2} \|\beta\|_{k+1,k+2} \;&\Longrightarrow\;\|\beta\|_{k+1,k+2}\leq \lambda^{d-2} \|\beta\|_{k, k+1}.
\end{align}

{\rm (2)} If $\|\beta\|_{k, k+3; \infty}\leq \epsilon_3$ and if $\beta$ is $L^2$-orthogonal to $\H_2$ over $A_{k+1, k+2}$, then
\begin{align}
\|\beta\|_{k+2, k+3}&\geq \frac{1}{\sqrt{2}}\lambda^{-\sigma} \|\beta\|_{k+1,k+2},\;{\rm \textit{or}}\\
\|\beta\|_{k, k+1}&\geq \frac{1}{\sqrt{2}}\lambda^{-\sigma} \|\beta\|_{k+1,k+2}.
\end{align}

\end{lem}
\begin{rmk} (1) is similar to \cite[Proposition 3.7]{DS2} for holomorphic functions. Here $\beta$ satisfies a non-linear equation, but we are able to linearize this equation because $\|\beta\|_{k, k+3; \infty}$ is small.
\end{rmk}

\begin{rmk}\label{r:leon}
Results such as Lemma \ref{lem12} and \ref{lem14} are a standard ingredient in unique tangent cone theorems. In fact, \eqref{eqn4}, which holds because we have an \emph{orthogonal} basis of homogeneous solutions, is vastly stronger than necessary. See \cite[Lemma 5.3]{AV} and \cite[p.549, Lemma 2]{LS}. In \cite[Lemma 5.31]{CT} it is implicitly assumed that $k_{\uparrow}, k_{\downarrow},k_0$ are orthogonal. While this may not always be true, it is not a crucial assumption \cite{AV,LS}: if $k_0 = 0$, the non-orthogonal error can be absorbed; and $k_0$ at worst lies in a finite-dimensional space, where any basis is orthogonal up to constants.
\end{rmk}

\begin{proof}[Proof of Lemma \ref{lem14}] It suffices to prove this for $k=0$. We will only prove \eqref{e:keep_growing} since \eqref{e:has_been_decaying} and (2) are similar. If \eqref{e:keep_growing} fails for $k = 0$, then for some $d \not\in \Sigma$ there is a sequence $\{\beta_i\}$ of $i\p\bp$-exact real $(1,1)$-forms on $A_{0,3}$ such that $\omega_C + \beta_i > 0$, ${\rm Ric}(\omega_C+\beta_i)=0$, $\|\beta_i\|_{0, 3; \infty} \rightarrow 0$, and
\begin{align}\label{equationbeta0}
 \|\beta_i\|_{0,1} \leq \lambda^{d-2} \|\beta_i\|_{1, 2} ,\;\,\|\beta_i\|_{2, 3} < \lambda^{-(d-2)} \|\beta_i\|_{1,2}.
\end{align}

By the Hodge-K\"ahler identities, each $\beta_i$ satisfies a first-order elliptic system 
\begin{align} \label{equationbeta}
d\beta_i=0,\;\,d^*\beta_i = d^c (H(\beta_i) - f_i),\\
f_i=\log( (\omega_C+\beta_i)^n/\omega_C^n),\;\,H(\beta_i)=f_i-{\rm tr}_{\omega_C}\beta_i.
\end{align}
Here $i\partial\bar\partial f_i = 0$ because $\omega_C + \beta_i$ is Ricci-flat, and $H(\beta_i)$ is a power series in $\beta_i$ with no constant or linear terms. Thus, we may write \eqref{equationbeta} as $(d + d^* + a_i \circledast \nabla)\beta_i = -d^c f_i$, where $\|a_i\|_{0,3;\infty} \to 0$. Also, $|f_i| \leq C|\beta_i|$ pointwise since $\|\beta_i\|_{0,3;\infty}$ is small, and $f_i$ is harmonic, so any norm of the right-hand side over any open set $U \subset A_{0,3}$ with $\bar{U} \subset A_{0,3}$ can be controlled by $\|\beta_i\|_{0,3}$. Interior $L^p$ estimates for linear elliptic systems \cite[Theorem 6.2.5]{Morrey} then show that for all $p\in[1,\infty)$,
\begin{align}\label{equationbeta2}
\|\beta_i\|_{L^p_1(U)}\leq C(U,p) \|\beta_i\|_{0, 3}.
\end{align}
Notice that the $L^p_1$ estimate holds with any $L^q$ norm of the solution on the right-hand side.

Now rescale $\beta_i=Q_i \tilde \beta_i$ with $\|\tilde{\beta}_i\|_{1,2}=1$ and $Q_i=\|\beta_i\|_{1, 2}\rightarrow 0$. ($Q_i > 0$ by \eqref{equationbeta0}.) Then on one hand 
$\|\tilde \beta_i\|_{0,3} < \lambda^{d-2}+1+\lambda^{-(d-2)}$ by \eqref{equationbeta0}. By \eqref{equationbeta2} for $p = 2$ and by Rellich, after passing to a subsequence, $\tilde\beta_i$ converges to some limit $\tilde\beta_\infty$ weakly in $L^2_{1,loc}$ and strongly in $L^2_{loc}$ on $A_{0,3}$. Thus,
\begin{align}\label{equationbeta3}
\|\tilde\beta_\infty\|_{1,2} = 1, \;\, \|\tilde \beta_\infty\|_{0,1} \leq \lambda^{d-2}, \;\, \|\tilde\beta_\infty\|_{2,3} \leq \lambda^{-(d-2)}.
\end{align}
On the other hand, (\ref{equationbeta}) implies that $d\tilde\beta_i = 0$ and $|d^*\tilde\beta_i + d^c \tilde{f}_i| \leq C|\beta_i| |\nabla\tilde\beta_i|$. Here $f_i = Q_i \tilde{f}_i$, so $\tilde f_i$ is harmonic, hence is locally uniformly bounded in any norm because $\|\tilde\beta_i\|_{0,3}$ is uniformly bounded. Also, $\nabla \tilde\beta_i$ is locally uniformly bounded in $L^2$ by \eqref{equationbeta2} for $p = 2$, so $|\beta_i||\nabla\tilde\beta_i| \to 0$ in $L^2_{loc}$. Together with the fact that $\tilde\beta_i \to \tilde\beta_\infty$ in $L^2_{loc}$, it follows that $d\tilde\beta_\infty = 0$ and $d^*\tilde\beta_\infty = -d^c \tilde f_\infty$ hold in the sense of distributions on $A_{0,3}$, where $\tilde f_\infty$ is harmonic (in fact, pluriharmonic). Thus, $\tilde\beta_\infty$ is smooth with $0 < \|\tilde\beta\|_{0,3} < \infty$ and $L(\tilde\beta_\infty) = 0$, and equality holds in \eqref{eqn4} by \eqref{equationbeta3}. If we can prove that $\tilde\beta_\infty$ is $i\partial\bar\partial$-exact, then Lemma \ref{lem12} and \eqref{equationbeta3} will imply that $d \in \Sigma$, which is false by assumption.

To prove that $\tilde\beta_\infty$ is $i\partial\bar\partial$-exact, we use that $\tilde\beta_i$ is $i\partial\bar\partial$-exact, hence $d$-exact. Having a $d$-primitive in $L^2_1$ is a closed property in $L^2$ on any compact manifold with boundary \cite[Theorem 7.7.7]{Morrey} (note that the finite-dimensional space $\mathfrak{H}$ of \cite{Morrey} may contain exact forms as well). Thus, for all smaller annuli $U$, $\tilde{\beta}_\infty|_U$ has a $d$-primitive in $L^2_1$, which we can take to be $C^\infty$ \cite[Theorem 7.7.8]{Morrey}. Since the inclusion $U \to A_{0,3}$ is a homotopy equivalence, $\tilde\beta_\infty$ is itself $d$-exact. Then $\tilde\beta_\infty$ is $i\partial\bar\partial$-exact because $H^{0,1}(A_{0,3}) = 0$. This vanishing follows from Andreotti-Grauert type results, which can be applied here because $H^1(B(o,1),\mathcal{O}_{C(Y)}) = 0$, $n \geq 3$, and $C(Y)$ is log-terminal; see \cite[Proposition 4.2]{vanC} and \cite[Appendix A]{CH1} for details, and also see \cite[p.254, Th\'eor\`eme 15]{AG} and \cite[Theorem 3.4.8]{H}.
\end{proof}

\begin{rmk}\label{randomremark}
Using \eqref{equationbeta2} and the Sobolev-Morrey embedding theorem, it follows that there are constants $\epsilon_4 \in (0,1)$ and $M_4 > 1$ depending only on $\lambda$ such that if $\beta$ is an $i\p\bp$-exact real $(1, 1)$-form on $A_{k, k+3}$ with $\omega_C+\beta>0$, ${\rm Ric}(\omega_C+\beta)=0$, and $\|\beta\|_{k, k+3; \infty}\leq \epsilon_4$, then
\begin{equation} \label{extraequation}
\|\beta\|_{k+1, k+2;\infty} + \|\nabla\beta\|_{k+1,k+2}\leq M_4 \|\beta\|_{k, k+3}. 
\end{equation}
In fact we have $L^p_1$ and $C^{0,\alpha}$ estimates for all $p,\alpha$, and will use those as well.
\end{rmk}

\subsubsection{Polynomial convergence in a holomorphic gauge} Based on the analysis of Section \ref{ss:linearizedRF}, the following key Proposition \ref{keyproposition} improves Lemmas \ref{lem6} and \ref{lem7} by showing that $\omega$ remains uniformly arbitrarily close to $\omega_C$ in a fixed holomorphic gauge. Using the above analysis again, Theorem \ref{thm1} will be a relatively straightforward consequence of this fact.

Fix $0 < \sigma \leq \min_{d\in\Sigma}|d-2|$ as above. Then fix $\lambda \in (0,1)$ such that 
\begin{equation} \label{eqn3}
\lambda^{\sigma}\leq 1/60.
\end{equation}
Write $\sqrt{2}\lambda^\sigma = \lambda^{d-2}$. Clearly $d \not\in\Sigma$. Define $\epsilon = \min\{\epsilon_3, \epsilon_3(d),\epsilon_4\}$ and 
\begin{equation}\label{eqn3+}
\tau_0=\epsilon/(15M_4).
\end{equation}

\begin{prop}\label{keyproposition}
For all $\tau\in(0, \tau_0)$ we can find an $I(\tau) \in \N$ such that for all $i\geq I(\tau)$,
\begin{equation}\label{e:583}
\sup_{j\geq 2}\|\omega_i-\omega_C\|_{j, j+1}\leq \tau.
\end{equation}
\end{prop}

\begin{proof}
Suppose for a contradiction that there is some $\tau \in (0,\tau_0)$ such that \eqref{e:583} fails for infinitely many (without loss of generality, for all) $i$. By Lemma \ref{lem6},  $e(i) =\|\omega_i-\omega_C\|_{2, 4; \infty}\rightarrow 0$. Then for all large enough $i$ there exists some $j=j(i)$ such that $\|\omega_i-\omega_C\|_{j, j+1}\leq \tau$ but $\|\omega_i-\omega_C\|_{j+1, j+2}> \tau$. By Lemma \ref{lem7}, we have that $|\L_{r\p_r}\omega_i-2\omega_i|_{\omega_i}\leq \epsilon(i)$  on $A_{j, j+6}$, where, here and in the rest of this proof, $\epsilon(i)$ denotes a function of $i$ with $\epsilon(i) \to 0$ as $i \to \infty$. Integrating this out, we see that for all $k\in [j+1, j+5]$,  $\|\omega_i-\omega_C\|_{k, k+1}\leq \tau+\epsilon(i)$. In particular, we may assume that $\|\omega_i-\omega_C\|_{k, k+1}\leq 2\tau$ for all $k \in [j+1,j+5]$.
We may also assume for all $k \in [2,j+4]$ that
\begin{equation}\label{e:666}
\|\omega_i-\omega_C\|_{k, k+1; \infty}\leq 6M_4\tau \leqslant \epsilon.
\end{equation}
Indeed, if $i$ is large, this obviously holds for $k=2$, and if $\|\omega_i-\omega_C\|_{k, k+1; \infty}\leq 6M_4\tau \leq \frac{\epsilon}{2}$, then again by Lemma \ref{lem7} we know that $\|\omega_i-\omega_C\|_{k, k+3;\infty} \leq \frac{\epsilon}{2} + \epsilon(i) \leq \epsilon$, so that $\|\omega_i-\omega_C\|_{k+1, k+2; \infty}\leq 6M_4\tau$ by (\ref{extraequation}) and because $\|\omega_i-\omega_C\|_{p,p+1} \leq 2\tau$ for $p \in \{k,k+1,k+2\} \subset [2,j+4]$. 

We can now explain the idea. Write $\beta = \omega_i - \omega_C$. Notice that $\|\beta\|_{j+2,j+3}\leq \frac{1}{2} \lambda^{-\sigma} \|\beta\|_{j+1, j+2}$ by \eqref{eqn3}. If $\beta$ happens to be $L^2$-orthogonal to $\mathcal{H}_2$ over $A_{j+1,j+2}$, then this inequality implies by \eqref{e:666} and Lemma \ref{lem14}(2) that $\|\beta\|_{j+1,j+2} \leq \sqrt{2}\lambda^\sigma\|\beta\|_{j,j+1}$. Using \eqref{e:666} and Lemma \ref{lem14}(1) inductively, it follows that $\tau \leq \|\beta\|_{j+1,j+2} \leq (\sqrt{2}\lambda^\sigma)^{j-1}e(i) \leqslant 30^{1-j}e(i)$, which is clearly false for $i$ large. This suggests that we should  first apply Lemma \ref{lem10} to force the required orthogonality. However, the change of gauge provided by Lemma \ref{lem10}, while bounded, may not be small enough to ensure that the above ``slow growth'' from $A_{j+1,j+2}$ to $A_{j+2,j+3}$ still holds after the gauge has been fixed.

To fix this, let $\delta =(20M_1M_2)^{-1} \min\{\epsilon_1, \epsilon_2\}$, and choose $p=p(i)$ to be minimal with
$$
\|\omega_i-\omega_C\|_{p+1, p+2}> \delta\tau. 
$$
Notice that for $i$ large, $p$ is always strictly smaller than the minimal $j$ we could have chosen above. As before, Lemma \ref{lem7} yields that $\|\omega_i-\omega_C\|_{q,q+1} \leq 2\delta\tau$ for all $q \in [p+1,p+3]$ if $i$ is large enough. Using this, it follows from \eqref{e:666} and \eqref{extraequation} that
$$\|\omega_i - \omega_C\|_{p,p+3}+\|\nabla(\omega_i-\omega_C)\|_{p,p+3} \leq (1 + 3M_4)\|\omega_i-\omega_C\|_{p-1,p+4} \leq (4M_4)(8\delta\tau) \leq \epsilon_1.$$
Lemma \ref{lem10} now yields a change of gauge $g\in G$ such that $d_G(Id, g)\leq M_1(5\delta\tau) \leq \epsilon_2$ and $g^*\omega_i-\omega_C$ is $L^2$-orthogonal to $\H_2$ over $A_{p+1, p+2}$. Moreover, for all $k \in [3, j+3]$, 
\begin{align}\label{e:777}
\|g^*\omega_i-\omega_C\|_{k,k+1;\infty} &\leq \|g^*(\omega_i - \omega_C)\|_{k,k+1;\infty} + \|g^*\omega_C - \omega_C\|_{k,k+1;\infty}\leq \epsilon
\end{align}
by Lemma \ref{lem11}, \eqref{eqn3+}, and \eqref{e:666}.

Thanks to \eqref{e:777} and the $L^2$-orthogonality over $A_{p+1,p+2}$, Lemma \ref{lem14}(2) tells us that
\begin{align}\label{eqn9}
\|g^*\omega_i-\omega_C\|_{p+2, p+3}&\geq\frac{1}{\sqrt{2}}\lambda^{-\sigma} \|g^*\omega_i-\omega_C\|_{p+1, p+2},\;{\rm {or}}\\
\label{eqn10}\|g^*\omega_i-\omega_C\|_{p, p+1}&\geq\frac{1}{\sqrt{2}}\lambda^{-\sigma} \|g^*\omega_i-\omega_C\|_{p+1, p+2}.
\end{align}
If \eqref{eqn10} holds, we are done because applying \eqref{e:777} and Lemma \ref{lem14}(1) inductively shows that $\delta\tau$ $\leq$ $\|g^*\omega_i-\omega_C\|_{p+1,p+2} \leq (\sqrt{2} \lambda^\sigma)^{p-2}e(i) \leq 30^{2-p}e(i)$, which is clearly false for $i$ large. It remains to rule out \eqref{eqn9}. If \eqref{eqn9} was true, then again using \eqref{e:777} and Lemma \ref{lem14}(1) inductively we would get that $\|g^*\omega_i - \omega_C\|_{k+1,k+2} \geq (\sqrt{2}\lambda^\sigma)^{-1}\|g^*\omega_i - \omega_C\|_{k,k+1}$ for all $k \in [p+1, j+2]$.  We would now like to derive a contradiction to the above ``slow growth'' of $\omega_i - \omega_C$ from $A_{j+1,j+2}$ to $A_{j+2,j+3}$.

To this end, we apply Lemma \ref{lem11} twice. On one hand, since $d_G(Id,g) \leq M_1(5\delta\tau)\leq\epsilon_2$,
\begin{align*}
\|g^*\omega_i-\omega_C\|_{j,j+3} &\geq \|g^*(\omega_i - \omega_C)\|_{j,j+3} - \|g^*\omega_C - \omega_C\|_{j,j+3} \\
&\geq \frac{1}{1 + M_2d_G(Id,g)}\|\omega_i-\omega_C\|_{j+1,j+2} - M_2 d_G(Id,g) \geq \frac{1}{2}\tau.
\end{align*}
Thus, $\|g^*\omega_i-\omega_C\|_{q,q+1} \geq \frac{1}{6}\tau$ for some $q \in \{j,j+1,j+2\}$. On the other hand, for this $q$,
\begin{align*}
\|g^*\omega_i-\omega_C\|_{q+1,q+2} &\leq \|g^*(\omega_i - \omega_C)\|_{q+1,q+2} + \|g^*\omega_C - \omega_C\|_{q+1,q+2} \\
&\leq  (1 + M_2d_G(Id,g))6\tau + M_2d_G(Id,g) \leq 7\tau.
\end{align*}
Comparing this ``slow growth'' estimate from $A_{q,q+1}$ to $A_{q+1,q+2}$ with the ``rapid growth'' estimate for $k = q$ from the preceding paragraph, we obtain that $\lambda^\sigma \geq \frac{1}{42\sqrt{2}}$, contradicting \eqref{eqn3}.
\end{proof}

\begin{proof}[Proof of Theorem \ref{thm1}]  Let $\tau=\delta\tau_0$ with $\delta$ as in the proof of Proposition \ref{keyproposition}. Then we know for all $i \geq I(\tau)$ as in the statement of the proposition that 
\begin{equation}\label{eqn5}
\sup_{j\geq 2}\|\omega_i-\omega_C\|_{j, j+1}\leq \tau. 
\end{equation}
Moreover, increasing $i$ if necessary, the argument used to prove \eqref{e:666} tells us that
\begin{equation}\label{eqn5+}
\sup_{j\geq 2} \|\omega_i-\omega_C\|_{j, j+1; \infty}\leq 3M_4\tau\leq \epsilon.
\end{equation}

Fix any $j \geq 3$. As in the proof of Proposition \ref{keyproposition}, we can use \eqref{eqn5+},  \eqref{extraequation}, \eqref{eqn5}, the fact that $(4M_4)(5\tau)=20M_4\delta\tau_0 \leq \epsilon_1$, and Lemma \ref{lem10} to construct a $g_j\in G$ such that $d_G(Id,g_j)$ $\leq$ $M_1(3\tau)$ $=$ $3M_1\delta\tau_0 \leq \epsilon_2$ and $g_j^*\omega_i-\omega_C$ is $L^2$-orthogonal to $\H_2$ over $A_{j+1, j+2}$. Then for all $k \geq 3$,
\begin{align}\label{eqn5++}
\|g_j^*\omega_i-\omega_C\|_{k, k+1;\infty}\leq\epsilon,
\end{align}
by Lemma \ref{lem11} and \eqref{eqn5+}. By \eqref{eqn5++} and Lemma \ref{lem14}, at least one of the following holds:
\begin{align}\label{eqn11}
\|g_j^*\omega_i-\omega_C\|_{j+2, j+3}&\geq\frac{1}{\sqrt{2}}\lambda^{-\sigma} \|g_j^*\omega_i-\omega_C\|_{j+1, j+2},\\
\|g_j^*\omega_i-\omega_C\|_{j, j+1}&\geq\frac{1}{\sqrt{2}}\lambda^{-\sigma} \|g_j^*\omega_i-\omega_C\|_{j+1, j+2}.\label{eqn12}
\end{align}
If (\ref{eqn11}) is true, then, by \eqref{eqn5++} and Lemma \ref{lem14}, for all $k\geq j+1$,
$$\|g_j^*\omega_i-\omega_C\|_{k, k+1}\geq 30^{k-j-1}\|g_j^*\omega_i-\omega_C\|_{j+1, j+2}.$$
Fixing $j$ and letting $k \to \infty$, this contradicts \eqref{eqn5} unless $g_j^*\omega_i = \omega_C$ on $A_{j+1,j+2}$, but if so (which is almost certainly impossible) then \eqref{eqn12} is trivially true. Thus,  (\ref{eqn12}) holds in any case, and so \eqref{eqn5++} and Lemma \ref{lem14} imply that for all $k \in [3, j]$, 
 \begin{equation} \label{eqn13}
\|g_j^*\omega_i-\omega_C\|_{k+1, k+2}\leq \sqrt{2}\lambda^\sigma \|g_j^*\omega_i-\omega_C\|_{k, k+1}.
 \end{equation}

We need some inequality such as \eqref{eqn13} to hold for \emph{all} $k \geq 3$. Thus, we now let $j$ tend to infinity. Since 
$d_G(Id,g_j) \leq \epsilon_2$, we may assume that $g_j\rightarrow g_\infty$. By (\ref{eqn13}), for all $k\geq 3$, 
 $$ \|g_\infty^*\omega_i-\omega_C\|_{k+1, k+2}\leq \sqrt{2}\lambda^\sigma \|g_\infty^*\omega_i-\omega_C \|_{k, k+1}.$$
Iterating this inequality and writing $\sqrt{2}\lambda^\sigma = \lambda^d$, it follows that for all $k \geq 0$,
\begin{align}\label{e:final_decay}
\|g_\infty^*\omega_i-\omega_C\|_{k, k+1}\leq C_i\lambda^{dk}.
\end{align}
Fix $i$ large enough as explained above. Let $P=\Psi_i\circ g_\infty \circ D^{-i}$, where $D(x) = \lambda x$ for all $x \in C(Y)$. 
Then $\beta = P^*\omega - \omega_C$ is defined on some small neighborhood of $o\in C(Y)$ and satisfies $\beta \in C^{0,\alpha}_{d}$ by \eqref{e:final_decay} and Remark \ref{randomremark}. By the 
proof of Lemma \ref{lem14}, $\beta$ satisfies a first-order elliptic system of the form $(d + d^* + a \circledast \nabla)\beta \in C^\infty_{e}$, where $e > 0$ and $a$ is a convergent power series in $\beta$ with $a \in C^{0,\alpha}_{d}$. Theorem \ref{thm1} now follows by bootstrapping, using standard scaled Schauder estimates. 
\end{proof}
\newpage

\section{Discussion}\label{s:discussion}

The statement of Theorem \ref{thm0-1} is not set to be optimal. Calabi-Yau and canonical singularities can very likely be relaxed to K\"ahler-Einstein and log-terminal singularities. The only point of the polarized and smoothable conditions is  to make direct use of the results of \cite{DS1, DS2}, and one would expect that these results can be generalized to a purely local setting of K\"ahler-Einstein metrics on germs of log-terminal singularities. The condition that $(X,x) \cong (C_x, o)$ for a Ricci-flat K\"ahler cone $C_x$ is more crucial, although one could certainly allow for $K_{C_x}$ to be torsion and it also seems likely that it suffices for $(X,x)$ to be a ``deformation of positive weight'' of $(C_x,o)$ (as in \cite[Remark 5.3]{CH1}, $t_i = t_{ij} = t_{ijk} = \epsilon = 0$). It would be highly desirable to relax $C_x$ being strongly regular to regular or quasi-regular. This may be quite difficult in Section \ref{section3-2}, where the strong regularity of $C_x$ is used to quickly determine the weighted tangent cone $W$, but perhaps less so in Section \ref{ss:broken_gauge}.

A more ambitious goal would be to prove the conjecture of \cite{DS2} that for a local K\"ahler-Einstein metric on a log-terminal germ $(X,x)$, the metric tangent cone $C(Y)$ at $x$ depends only on $(X,x)$ and can be characterized in terms of the $K$-stability of $C(Y)$. One possible path to this was proposed by Li \cite{Li}, motivated by \cite{DS2} and by the theory of {volume minimization} in Sasaki geometry \cite{MSY}. According to \cite{DS2} there is a metrically defined filtration of the local ring $\O_x$ whose associated graded ring defines a certain weighted tangent cone $W$, and there is an equivariant degeneration of $W$ to $C(Y)$, which should determine $C(Y)$ uniquely in terms of $W$. Li \cite{Li} conjectured that this filtration minimizes a certain \emph{normalized volume functional} $\widehat{\rm Vol}$, \eqref{e:li-liu}, on the space of valuations of $\mathcal{O}_x$ and that minimizers are unique. \cite{LL, LX} prove that if $(X,x)$ admits a local K\"ahler-Einstein metric with $C(Y)$ quasi-regular, then the corresponding metric filtration minimizes $\widehat{\rm Vol}$. \cite{LX} show that there is then no other minimizer among quasi-regular valuations. Currently the main missing ingredient in this approach seems to be a good understanding of $\widehat{\rm Vol}$ at irregular valuations.

Our work shows that the continuity method (connecting two local K\"ahler-Einstein metrics $\omega_0, \omega_1$ by a path $\omega_t$ of local K\"ahler-Einstein metrics, e.g. in the space of boundary values for the Monge-Amp{\`e}re equation) may provide a substantial shortcut to proving that $C(Y)$ depends only on $(X,x)$: if the valuation associated with $\omega_1$ minimizes $\widehat{\rm Vol}$, and if the set of all $t$ such that $C(Y_t) = C(Y_1)$ is open, then $C(Y_t) = C(Y_1)$ for every $t$ because both ${\rm Vol}(C(Y_t))$ and\footnote{This we currently know only if $C(Y_1)$ is quasi-regular; compare Remark \ref{rmk:logdisc}.} the log-discrepancy can only decrease as $t$ approaches a time where $C(Y_t)$ jumps. In fact, Section \ref{s:openness} suggests that $C(Y_t)$ may be locally constant, hence constant, along \emph{every} path even without using $\widehat{\rm Vol}$ (the first-named author has proposed this idea in several talks over the past few years). However, attempting to prove such strong versions of openness leads into difficult linear PDE problems on singular spaces with poor rates of convergence, including jumps of topology, towards their tangent cones. An interesting test case would be $(X,x)$ locally isomorphic to the cone over a Mukai-Umemura $3$-fold without K\"ahler-Einstein metrics. Then we expect $C(Y)$ is the cone over a Mukai-Umemura $3$-fold that does admit K\"ahler-Einstein metrics; cf. \cite{Wang}. Then $(X,x)$ and $(C(Y),o)$ are homeomorphic but the convergence rate should only be logarithmic, showing that the result of Colding-Minicozzi \cite{CM} is sharp.

To conclude we mention a different question motivated by Theorem \ref{thm0-1}. The theorem produces an almost canonical isomorphism between each germ $(X,x)$ and the model $(C_x,o)$. This isomorphism is unique only modulo automorphisms of $(C_x, o)$ that preserve $\omega_{C_x}$ to leading order, so the most interesting case is when $C_x = \frac{1}{q}K_Z^\times$ for a Fano manifold $Z$ with continuous automorphisms. We need to make an initial choice of isomorphism to start the continuity method, and the PDE then produces a path of automorphisms of $Z$ connecting the chosen isomorphism to the canonical one to leading order. Can one understand this path of automorphisms more systematically?

\appendix

\section{Special Lagrangian vanishing cycles}\label{s:vanishing_cycles}

Let $(X,L)$ be a smoothable Calabi-Yau variety as in Definitions \ref{d:cyvariety}--\ref{d:cysmoothable}. Suppose that $X$ has at worst nodal singularities, as discussed after Theorem \ref{thm0-1}. Thus, by Theorem \ref{thm0-1}, the unique weak Calabi-Yau metric $\omega$ on $(X,L)$ is polynomially asymptotic to the model metric \eqref{e:stenzel} at each node up to biholomorphism. In this section, we show that this implies that any smoothing family $(\mathcal X, \mathcal L)$ as in Definition \ref{d:cysmoothable} has \emph{special Lagrangian vanishing cycles}. Experts have long been aware  that such an implication holds in some form. The point is that knowing the exact behavior of $\omega$ at the nodes allows one to ``glue in'' scaled copies of Stenzel's asymptotically conical Calabi-Yau metric on $T^*S^n$ \cite{Ste}, and that the zero section of $T^*S^n$ is special Lagrangian with respect to Stenzel's metric.

For example, the above remarks together with Chan's work \cite[Theorem 4.31 and Theorem 6.1]{Chan} immediately tell us the following: if $n =  3$, then for every class in the image of the restriction map $H^3(X^{reg},\C) \to \bigoplus_{i=1}^k H^3(Y_i, \C)$ ($k$ is the number of nodes of $X$ and $Y_i \cong S^2 \times S^3$ is the link of the $i$-th one) whose projection to $H^3(Y_i,\C)$ is non-zero for all $i$, there is an associated real $1$-parameter family of smooth Calabi-Yau $3$-folds degenerating to $X$ with special Lagrangian vanishing cycles.

One of the key features of Chan's work is that he was able to use a differential-geometric gluing construction to obtain a concrete criterion for the complex structure of a nodal Calabi-Yau $3$-fold to be smoothable (assuming the metric asymptotics of \eqref{e:polyconv}). However, precisely because of this, it is not immediately obvious \emph{which} smooth Calabi-Yau $3$-folds are being produced. To illustrate this point, let $X$ be a quintic $3$-fold with one node. Then the $H^3$ condition can be satisfied because $X$ is smoothable. The locus of $1$-nodal quintics is a smooth hypersurface of the space of all quintics \cite{ShTyo}, and a small flat deformation of a $1$-nodal quintic is again a quintic \cite{Wehler}. By letting $X$ vary over all $1$-nodal quintics, it then seems very likely that Chan's construction sweeps out an {open set} in the space of smooth quintics. However, to be precise one would need to check that Chan's families do indeed extend to flat holomorphic $1$-parameter families transverse to the $1$-nodal locus.

To avoid such technical issues, and also to cover the general $n$-dimensional case, one can use a slightly different gluing construction due to Biquard-Rollin \cite{BiRo} and Spotti \cite{Spot} in the surface case. In these papers, a flat holomorphic family is \emph{given}, and the gluing is carried out purely at the level of K\"ahler potentials. With the help of the following lemma (which is also needed in order to be able to use the metrics of Theorem \ref{thm0-1} as input for the gluing construction of \cite{ArSpot}), it is possible to extend Spotti's construction \cite{Spot} step by step to the $n$-dimensional Calabi-Yau case.

\begin{lem}
In Theorem \ref{thm0-1}, $P^*\omega - \omega_{C_x} = i\partial\bar\partial u$ for some function $u \in C^\infty_{2+\lambda}$. 
\end{lem}

\begin{proof}
We know that $P^*\omega - \omega_{C_x}$ belongs to $C^\infty_{\lambda}$ and can be written as $i\partial\bar\partial u$ for some smooth and uniformly bounded function $u$. In particular, $\Delta u \in C^\infty_\lambda$. By Proposition \ref{cone_expansion}, $u = h + \bar{u}$, where $h$ is a finite sum of homogeneous harmonic functions with rates in $[0,2+\lambda)$,  and $\bar{u} \in C^\infty_{2+\lambda}$. Thus, $i\partial\bar\partial h \in C^\infty_{\lambda}$, and hence $i\partial\bar\partial h = 0$ by homogeneity. The claim follows by replacing $u$ by $\bar{u}$.
\end{proof}

Spotti \cite{Spot} in addition requires the given flat family to be versal for each singularity of its central fiber, but this is not a restrictive condition; for example, it is satisfied for all generic smoothings of a $1$-nodal hypersurface \cite{ShTyo}. The same general perturbation theorem of Joyce \cite[Theorem 5.3]{JoycePart3} that Chan used in \cite{Chan} then yields special Lagrangian vanishing cycles.

\begin{cor} If a smoothing $\mathcal{X}$ as above is versal for each node of its central fiber $X$, then each node of $X$ is the limit of vanishing special Lagrangian $n$-spheres in the nearby fibers of $\mathcal{X}$. 
\end{cor} 

Let us point out in closing that Theorem \ref{thm0-1} also produces Calabi-Yau $3$-folds with singularities modeled on cones over \emph{cubic} surfaces. The symplectic geometry of smoothings of cubic  cones was studied in \cite{SmithThomas}. It would be interesting to see if the Lagrangian vanishing cycles described in \cite{SmithThomas} can be made special with respect to the natural Calabi-Yau structures on such smoothings.

\section{Harmonic $1$-forms on cones with ${\rm Ric} \geq 0$}\label{s:CT}

The aim of this appendix is to give a detailed proof of the following lemma of Cheeger and Tian \cite[Lemma 7.27]{CT}. This was crucially used in Section \ref{s:harmonic} above, as well as in \cite[Section 3]{CH1}.

\begin{lem}[Cheeger-Tian]\label{l:ct2} Let $C = C(Y)$ be a Riemannian cone of dimension $m \geqslant 3$ such that ${\rm Ric}_C \geq 0$. Let $\psi$ be a homogeneous $1$-form on $C$ with growth rate in $[0,1]$. Then $(dd^* + d^*d)\psi = 0$ holds if and only if, up to linear combination, either $\psi = d(r^\mu \phi)$,
where $\phi$ is a $\lambda$-eigenfunction on $Y$ for some $\lambda \in [m-1,2m]$ and $\mu$ is chosen so that $r^\mu\phi$ is a harmonic function on $C$, or $\psi = r^2\eta$, where $\mathcal{L}_{r\partial_r}\eta = 0$ and, at $r = 1$, $\eta^\sharp$ is a Killing field on $Y$ with ${\rm Ric}_Y\,\eta^\sharp = (m-2)\eta^\sharp$, or $\psi = r dr$.
\end{lem}

The proof of Lemma \ref{l:ct2} depends on the following basic fact about the Hodge Laplacian acting on $1$-forms on a manifold of positive Ricci curvature. The case of closed forms is nothing but the classical Lichnerowicz-Obata theorem. The case of co-closed forms has been used by various authors but seems to be less well-known; e.g. see \cite[(0.22)]{CT} and the discussion there, or \cite[p.18]{Uhl}.

\begin{lem}\label{cclich} Let $Y$ be a closed and connected Riemannian manifold of dimension $k \geqslant 2$ such that ${\rm Ric}_Y \geq (k-1)g_Y$. Then the first eigenvalue of $dd^* + d^*d$ acting on closed $1$-forms on $Y$ is at least $k$. Equality holds if and only if $Y$ is a round sphere of radius $1$, and then the first eigenspace consists precisely of the differentials of linear functions on $\R^{k+1}$ restricted to $Y$. On the other hand, the first eigenvalue of $dd^* + d^*d$ on co-closed $1$-forms is always at least $2(k-1)$
and the $2(k-1)$-eigenspace consists precisely of the duals of those Killing fields $Z$ on $Y$ for which ${\rm Ric}_Y\,Z = (k-1)Z$.
\end{lem}

\begin{proof}
Let $\eta$ be an eigen-$1$-form on $Y$ with eigenvalue $\lambda$. Then, by the Bochner formula,
$$\int_Y \lambda|\eta|^2 = \int_Y |d\eta|^2  +  |d^*\eta|^2 = \int_Y (|\nabla\eta|^2 + {\rm Ric}(\eta,\eta))\geq \int_Y (|\nabla\eta|^2 + (k-1)|\eta|^2).$$
Thus, $\lambda \geq k-1$, but equality would imply $\nabla \eta = 0$, so $\eta^\sharp$ would split off as an isometric $\R$-factor on the universal cover of $Y$, which is impossible because $Y$ has finite fundamental group. To obtain an improvement we decompose $\nabla\eta$ into its trace-free symmetric, trace, and skew-symmetric parts:
$$|\nabla\eta|^2 = |\nabla^{\rm sym}_0\eta|^2 + \left|\frac{{\rm tr}(\nabla\eta)}{k}g\right|^2 + |\nabla^{\rm skew}\eta|^2 = |\nabla^{\rm sym}_0\eta|^2 + \frac{1}{k} |d^*\eta|^2 + \frac{1}{2}|d\eta|^2.$$
Notice that $\nabla^{\rm skew}\eta = \sum_{i<j} (\nabla_i \eta_j - \nabla_j \eta_i) \frac{1}{2} (e_i^\flat \otimes e_j^\flat - e_j^\flat \otimes e_i^\flat)$ and $|\frac{1}{2}(e_i^\flat \otimes e_j^\flat - e_j^\flat \otimes e_i^\flat)|^2 = \frac{1}{2}$ for $i < j$, whereas $d\eta = \sum_{i<j} (\nabla_i \eta_j - \nabla_j \eta_i)e_i^\flat \wedge e_j^\flat$ and $|e_i^\flat \wedge e_j^\flat|^2 = 1$ for $i < j$.

Thus, if $\eta$ is closed, then
$$\int_Y \lambda |\eta|^2 = \int_Y |d^*\eta|^2 \geq \int_Y \left(\frac{1}{k}|d^*\eta|^2 + (k-1)|\eta|^2\right),$$
and hence $\lambda \geq k$. Equality implies $\nabla^{\rm sym}_0\eta = 0$ and ${\rm Ric}(\eta,\eta) = (k-1)|\eta|^2$. Passing to the
universal cover $\tilde{Y}$ we have $\eta = d\phi$, and then $\nabla^2\phi + \phi g = 0$. A well-known elementary argument due to Obata now shows that $\tilde{Y}$ is the unit sphere in $\mathbb{R}^{k+1}$ and $\phi$ is the restriction of a linear function. Since $\eta$ $=$ $d\phi$ is invariant under the $\pi_1(Y)$-action on $\tilde{Y}$, it follows that $\pi_1(Y) = \{1\}$ and $Y = \tilde{Y}$.

On the other hand, if $\eta$ is co-closed, then
$$\int_Y \lambda|\eta|^2 = \int_Y |d\eta|^2 \geq \int_Y \left(\frac{1}{2}|d\eta|^2 + (k-1)|\eta|^2\right),$$
and hence $\lambda \geq 2(k-1)$. Equality again implies $\nabla^{\rm sym}_0\eta = 0$ and ${\rm Ric}(\eta,\eta) = (k-1)|\eta|^2$. As $d^*\eta = 0$, the dual of $\eta$ must be a Killing field. On every closed manifold, the dual of a co-closed $1$-form $\eta$ is a Killing field if and only if $\nabla^*\nabla \eta = {\rm Ric}(\eta)$. Thus, in our situation, ${\rm Ric}(\eta) = (k-1)\eta$, and conversely any Killing field with this property defines a co-closed $2(k-1)$-eigenform.
\end{proof}

\begin{proof}[Proof of Lemma \ref{l:ct2}] We start by recording separation-of-variables formulas for $dd^*$ and $d^*d$ acting on $1$-forms on cones. We use a prime to denote the operator $\mathcal{L}_{\partial_r}$ and a tilde to denote operators on $Y$ with its unit-size metric. (Thus, on $\mathbb{R}^2\setminus \{0\} = \mathbb{R}^+_r \times S^1_\theta$, $(d\theta)' = 0$, but $\nabla_{\partial_r}(d\theta) = - \frac{1}{r} d\theta$ because $|d\theta| = \frac{1}{r}$.) With these conventions, if a $1$-form $\eta$ on $C$ is tangent to $Y$ everywhere, then
\begin{align}\label{ddstareta}dd^*\eta &= \left(-\frac{2}{r^3}\tilde{d}^*\eta + \frac{1}{r^2}\tilde{d}^*\eta' \right)dr + \frac{1}{r^2} \tilde{d}\tilde{d}^*\eta,\\
\label{dstardeta}d^*d\eta &= -\eta'' - \frac{m-3}{r}\eta' - \frac{1}{r^2}(\tilde{d}^*\eta')dr + \frac{1}{r^2}\tilde{d}^*\tilde{d}\eta,\end{align}
and if $\kappa$ is a function on $C$, then
\begin{align}\label{ddstarkappa} dd^*(\kappa dr) &= -\kappa'' dr + \left(\frac{m-1}{r^2}\kappa - \frac{m-1}{r}\kappa'\right)dr - \tilde{d}\kappa' - \frac{m-1}{r}\tilde{d}\kappa,\\
\label{dstardkappa}d^*d(\kappa dr) &= \tilde{d}\kappa' + \frac{1}{r^2}(\tilde{d}^*\tilde{d}\kappa)dr + \frac{m-3}{r}\tilde{d}\kappa.
\end{align} 
These formulas can be found in \cite[(2.14), (2.15)]{CT}; see also \cite[p.586]{chee}. To prove them, one can use that $d = \sum e_i^\flat \wedge \nabla_{e_i}$ and $d^* = -\sum e_i \,\lrcorner\, \nabla_{e_i}$, and that $\nabla_{\partial_r}\eta = \eta'  -\frac{1}{r}\eta$ and $\nabla_{Z}\eta = \tilde\nabla_{Z}\eta - \frac{1}{r}\eta(Z)dr$ for all vector fields $Z$ tangent to $Y$. (Hence in particular $d\eta = dr \wedge \eta' + \tilde{d}\eta$.)

The following simple observation is key: if $\psi = \kappa dr + \eta$ is any smooth $1$-form on $C$, then
\begin{align*}
\kappa(r,y) &= \sum f_\lambda(r) \kappa_\lambda(y), \quad\quad\quad\quad\quad\quad\quad\;\tilde{d}^*\tilde{d}\kappa_\lambda = \lambda \kappa_\lambda , \\
 \eta(r,y) &= \sum g_\lambda(r)\tilde{d}\kappa_\lambda(y) + h_\lambda(r) \eta_\lambda(y), \quad \tilde{d}^*\tilde{d}\eta_\lambda = \lambda \eta_\lambda, \; \tilde{d}^*{\eta}_\lambda = 0.
\end{align*} 
Here the sums formally run over all $\lambda \in \R_{\geq 0}$, but the terms added can be non-zero only if $\lambda$ is an eigenvalue of the Hodge Laplacian on exact or co-closed $1$-forms. Also, $f_\lambda, g_\lambda$ are $C^\infty_{\rm loc}$ and the sums converge in the $C^\infty_{\rm loc}$ topology because $\psi$ is smooth. Thus, any smooth $1$-form can be written as a smoothly convergent sum of smooth $1$-forms of one of the following two types:
\begin{align*}f(r)\kappa(y) dr + g(r) \tilde{d}\kappa(y), \quad &\tilde{d}^*\tilde{d}\kappa = \lambda\kappa, \tag{\rm A}\\
h(r) \eta(y), \quad &\tilde{d}^*\tilde{d}\eta = \lambda\eta, \; \tilde{d}^*\eta = 0. \tag{B} \end{align*}
Moreover, both types are invariant under $dd^* + d^*d$ by \eqref{ddstareta}--\eqref{dstardkappa}, so that every Hodge harmonic $1$-form on $C$ must be an infinite sum of Hodge harmonic $1$-forms of type A and type B. Moreover, if the given form is homogeneous, then each of these components is homogeneous as well.\\

\noindent {\bf Harmonic $1$-forms of type A.} Define two auxiliary quantities $D,E$ by setting
$$D = g'' + \frac{m-1}{r}g' - \frac{\lambda}{r^2}g, \;\, E = g' - f. $$
Then the harmonic equation is equivalent to
\begin{align}\label{e:reduce}
D = \frac{2}{r}E,\;\,
D' = E'' + \frac{m-1}{r}E' - \frac{\lambda + m-1}{r^2}E.\end{align}
Inserting the first equation of \eqref{e:reduce} into the second yields
$$E'' + \frac{m-3}{r}E' - \frac{\lambda + m-3}{r^2}E = 0.$$
The general solution here is a linear combination of powers $r^\nu$, where
$$2\nu = -(m-4) \pm \sqrt{(m-4)^2 + 4(\lambda + m-3)}.$$
Homogeneity implies that $E = const \cdot r^\nu$ for one of these roots. Inserting this into the first equation of \eqref{e:reduce} and solving for $g$ leads to the following three possibilities.\newpage

$\bullet$ If $E = 0$, then $g = const \cdot r^\mu$, where
\begin{align}\label{e:scalarharmonic}
2\mu = -(m-2) \pm \sqrt{(m-2)^2 +4 \lambda}.
\end{align}

$\bullet$ If $E \neq 0$ and $\mu \neq \nu + 1$ for both values of $\mu$, then homogeneity implies that $g = const \cdot r^{\nu + 1}$. 

$\bullet$ If $E \neq 0$ and $\mu = \nu + 1$ for one value of $\mu$, then no homogeneous solutions exist.\\

\noindent {\bf Harmonic $1$-forms of type B.} A form of type B is co-closed, hence is harmonic if and only if
$$h'' + \frac{m-3}{r}h' - \frac{\lambda}{r^2}h = 0.$$
Homogeneity implies that $h = const \cdot r^\tau$, where
\begin{equation}\label{ccrate}2\tau = -(m-4) \pm \sqrt{(m-4)^2 + 4\lambda}.\end{equation}

\vskip4mm

Now let $\psi$ be a homogeneous harmonic $1$-form of growth rate in $[0,1]$ as in the statement of the lemma. Then up to linear combinations the above leaves three options: 
either $\psi = d(r^{\mu}\kappa)$, where $\kappa$ 
is a $\lambda$-eigenfunction of the Laplacian on $Y$
such that $\lambda \in [m-1,2m]$ and $r^\mu\kappa$ is harmonic on $C$ (i.e. $\mu$ is the plus solution in \eqref{e:scalarharmonic}), or $\psi = rdr$, or else $\psi = r^\tau \eta$, where $\eta$ is a co-closed $\lambda$-eigen-$1$-form on $Y$, $\lambda \in [m-3, 2m-4]$, 
and $\tau$ is the plus solution in (\ref{ccrate}).

Everything so far holds on every Riemannian cone of dimension $m \geqslant 3$. If in addition ${\rm Ric}_C \geq 0$ (equivalently, ${\rm Ric}_Y \geq (m-2)g_Y$), then all eigenvalues of the Laplacian on co-closed $1$-forms on $Y$ are greater or equal than $2m-4$ by Lemma \ref{cclich}, and the $(2m-4)$-eigenspace consists precisely of the duals of those Killing fields $Z$ on $Y$ that satisfy ${\rm Ric}_Y\,Z = (m-2)Z$. \end{proof}

\begin{rmk}
From the proof of Lemma \ref{l:ct2}, if the expansion of a harmonic $1$-form on a cone of dimension $m \geq 3$ contains any log terms at all, then these log terms must take the following shape: either $m = 3$, $\psi(r,y) = r^{-\frac{1}{2}}(\log r)(\kappa(y) dr + 2r \tilde{d}\kappa(y))$, and $\kappa$ is a $\frac{3}{4}$-eigenfunction on $Y$; or $m = 4$, $\psi(r,y) = (\log r)\eta(y)$, and $\eta$ is a harmonic $1$-form on $Y$. By Lemma \ref{cclich}, neither of these can occur if the cone has ${\rm Ric} \geq 0$. Compare \cite[Remark 3.5]{CH1} and \cite[p.546]{CT}.
\end{rmk}

\section{Volume of Calabi-Yau cones}\label{ss:li_liu}

We adopt the notation of Section \ref{s:closedness}. The goal of this appendix is to explain the inequality
\begin{equation}\label{eqnB0}
{\rm Vol}(C(Y))\geq {\rm Vol}(C_x),
\end{equation}
which was crucially used in the proof of Theorem \ref{thm3-5}, using the recent algebro-geometric result \cite{LL}. Recall the order of vanishing $d_{KE}: \mathcal{O}_x \to \mathbb{R}_{\geq 0}$ associated with our K\"ahler-Einstein metric $\omega_{t_\infty}$ on $(X,x)$. This was defined in \eqref{e:order}. Verbatim the same definition applies to $\omega_{t_i}$ for every $i$, and the resulting function $d_{KE}': \mathcal{O}_x \to \mathbb{R}_{\geq 0}$ is clearly independent of $i$. It is implicit in \cite{DS2} that $d_{KE}$ and $d_{KE}'$ are valuations on $\mathcal{O}_x$ (see also \cite{Do}\footnote{\cite{Do} states the slightly suboptimal property $d_{KE}(fg) \geq d_{KE}(f) + d_{KE}(g)$ in addition to the other defining properties of a valuation. But it is easy to see that equality indeed holds, using the definition of $W$ and the fact proved in \cite{DS2} that $W$ is normal, hence is reduced and irreducible. }).
Based on work of Berman \cite{Berman} and Fujita \cite{Fujita},  Li-Liu \cite{LL} proved that $d_{KE}'$ minimizes the \emph{normalized volume} among all such valuations $\nu$,  
\begin{align}\label{e:li-liu}A(\nu)^n{\rm Vol}(\nu) \geq A(d_{KE}')^n {\rm Vol}(d_{KE}').\end{align}
Here the \emph{volume} of $\nu$ is defined by
$${\rm Vol}(\nu) =  \lim_{d\rightarrow\infty} \frac{n!}{d^n}\dim\mathcal O_x/\{f \in \mathcal O_x: \nu(f)\geq d\}, $$
and $A(\nu)$ is the \emph{log-discrepancy} of $\nu$ introduced in \cite{JM}. We will not recall the precise definition of $A(\nu)$ in \cite{JM}, but only list a few properties relevant to the sequel.

\begin{enumerate}[(a)]
\item If $\nu_1=c\nu_2$ for some $c\in \R_{>0}$, then $A(\nu_1)=cA(\nu_2)$.\smallskip
\item\label{itb} Given a proper birational morphism $\pi: \hat X \rightarrow X$ with $\hat X$ normal and a prime divisor $E$ in $\hat X$, let $\nu_E$ be the \emph{divisorial} valuation $\nu_E(f) = {\rm ord}_E(f \circ \pi)$ on $\mathcal O_x$. Then $A(\nu_E)$ is given by
\begin{equation} \label{eqnB4}
K_{\hat X}=\pi^*K_{X}+(A(\nu_E)-1)E.
\end{equation}
More concretely, if one pulls back a rational section of $K_{X}$ from $X$ to $\hat X$, then $A(\nu_E)-1$ is the pole or vanishing order of the pull-back section along $E$.\smallskip
\item $A(\nu)$ is lower semi-continuous: if a family of valuations $\nu_i$ converges to $\nu$ in the weak topology, i.e. $\nu(f)=\lim_{i\rightarrow\infty} \nu_i(f)$ for all $f\in \mathcal O_x$, then $A(\nu)\leq \liminf_{i\rightarrow\infty}A(\nu_i)$. 
\end{enumerate}

\noindent Our goal in this appendix is to derive \eqref{eqnB0} from \eqref{e:li-liu} applied to $\nu = d_{KE}$.

The first step will be to show that ${\rm Vol}(d_{KE}')$, ${\rm Vol}(d_{KE})$ are equal to ${\rm Vol}(C_x)$, ${\rm Vol}(C(Y))$ up to a dimensional factor. This is essentially well-known. We only write the proof for $C(Y)$. Define
$$F(t)=\sum_{d\in \mathcal S} e^{-td} \dim R_d(C(Y)).$$
As explained in \cite[Proposition 4.1]{CS} and \cite[Lemma 4.2]{DS2}, we know that as $t \to 0^+$, 
$$F(t)= \frac{{\rm Vol}(C(Y))}{{\rm Vol}(\C^n)} \frac{1}{t^n}+O\left(\frac{1}{t^{n-1}}\right).$$
Thus, by standard Tauberian theory (see for example \cite[Theorem 2.42]{BGV}),
$$\sum_{e\leq d} \dim R_e(C(Y))=\frac{{\rm Vol}(C(Y))}{{\rm Vol}(\C^n)}\frac{d^n}{n!}+o(d^n) $$
as $d \to \infty$.
By the discussion in Section \ref{ss:construct_tgt_cone},
$$\sum_{e\leq d} \dim R_e(C(Y))=\sum_{d_k\leq d} \dim I_{k}/I_{k+1}=\dim \mathcal O_x/\{f\in \mathcal O_x:d_{KE}(f)\geq d\},$$
and this clearly establishes the claim.

The second, more delicate step consists of the following proposition. 

\begin{prop}  \label{propB2}
We have $A(d_{KE})\leq n$ with equality if $C(Y)$ is quasi-regular, i.e. $\dim \mathcal T = 1$.
\end{prop}

Since this applies verbatim to $d_{KE}'$, $C_x$ in place of $d_{KE}$, $C(Y)$, and since $C_x$ is actually strongly regular, it is now clear that \eqref{e:li-liu} with $\nu = d_{KE}$ indeed implies \eqref{eqnB0} as desired.

The rest of this appendix is dedicated to proving Proposition \ref{propB2}. Recall from \cite{DS2} that we can embed $C(Y)$ as an affine algebraic subvariety of some $\C^N$  such that the Reeb vector field $\xi$ is the restriction of $2{\rm Re}(i\sum_{j=1}^N w_j z_j\p_{z_j})$ from $\C^N$  for some weight $w=(w_1, \ldots, w_N)\in (\R_{>0})^N$. There is also an embedding of $(X, x)$ into $(\C^N, 0)$ such that $W$ is isomorphic to the \emph{weighted tangent cone} of $(X, x)$ with respect to $w$. Moreover, $d_{KE}$ is equal to the \emph{weighted degree function} $d_w$. The latter is defined as follows. Given any $f\in \mathcal O_x$, choose an extension $\tilde f$ to a neighborhood of $0$ in $\C^N$, with Taylor expansion $\tilde f(y)=\sum_{\alpha\in\N_0^N } c_\alpha y^\alpha$. Then  the weighted degree of $\tilde f$ is defined by
$$d_w(\tilde f)=\inf \left\{\sum_{i=1}^N w_i\alpha_i: \alpha \in \N_0^N, \,c_\alpha\neq 0\right\},$$
and we define $d_w(f)$ to be the biggest $d_w(\tilde f)$ among all such extensions $\tilde f$. Also recall from \cite{DS2} the underlying reason why $d_{KE} = d_w$. This is the inequality
\begin{equation} \label{eqnB2}
c_\epsilon^{-1} r_w(y)^{1+\epsilon} \leq r_{KE}(y) \leq c_\epsilon r_w(y)^{1-\epsilon}
\end{equation}
for all $\epsilon > 0$ and $y \in (X,x)$, where $r_{KE}(y)$ denotes the $\omega_{t_\infty}$-distance of $y$ and $x$ and 
\begin{align}\label{d:weighted_radius}
r_w(y)^2=\sum_{i=1}^N |y_i|^{\frac{2}{w_i}}.
\end{align}

\noindent {\bf Case 1: $C(Y)$ is quasi-regular.} Then we can uniquely write $w_i=cw'_i$ with $c\in \Q_{>0}$, $w'_i\in \Z_{>0}$, and $w_1', \ldots, w_N'$ (not necessarily pairwise) coprime. By (a) above, $A(d_{KE}) = A(d_w) = c A(d_{w'})$. We will prove the claim, $A(d_{KE}) = n$, by writing $d_{w'} = \nu_E$ for some prime divisor $E$ over $X$ as defined in (b) above and applying \eqref{eqnB4}. As a basic warning, finding a pair $(\hat{X}, E)$ as in (b) such that $\hat{X}$ is normal, $E$ is prime, and ${\rm ord}_E$ of the relative canonical divisor can be readily calculated is {not easy} in general. Here this will crucially depend on the fact from \cite{DS2} that $W$ is {normal}.

We will construct $\hat{X}$ as a weighted blow-up of $X$. For this we require some preliminaries. For a non-zero vector $v\in \C^N$ we denote the $\C^*$-orbit of $v$ by $[v]$. Let $\WP$ denote the quotient $\C^N/\!\!/\C^*$. This is a \emph{weighted projective space}. It has a natural orbifold structure and the underlying projective variety of this orbifold is normal. We will now describe local affine charts for this variety.\footnote{By their construction as $\C^{N-1}/\Z_{w_i'}$, these charts carry natural orbifold structures as well. However, simple examples such as $\WP(1,1,2)$ and $\WP(1,2,2)$ show that the codimension-$1$ orbifold structure of such a chart may or may not agree with that of $\WP$. Since we are interested only in the underlying varieties, this does not matter for us.} For any $i$, consider the map $p_i: \C^{N-1}\rightarrow \WP$ given by $p_i(z_1, \ldots, \hat z_i, \ldots, z_N) = [(z_1, \ldots, 1, \ldots, z_N)]$. As a map of varieties, $p_i$ factors through an embedding of  \begin{small}$\C^{N-1}/\Z_{w'_i}$\end{small} onto the open set $\{v_i\neq 0\}\subset \WP$, where $\Z_{w'_i}$ acts via 
$(\zeta.z)_j$ $=$ \begin{small}$\zeta^{w'_j}$\end{small}$z_j$ for all $w_i'$-th roots of unity $\zeta$ and all $j \in \{1, \ldots, \hat\imath, \ldots, N\}$.

Let $\pi: \C^N\times \WP \rightarrow \C^N$ be the obvious projection and let $\hat \C^N \subset \C^N\times \WP$ be the subvariety of 
all pairs $(v, [u])$ with $v\in \overline{[u]}$.  Then $\pi$ induces a proper birational morphism $\pi: \hat\C^N\rightarrow \C^N$, which is an isomorphism away from its exceptional set $Exc(\pi) = \pi^{-1}(0) = \WP$. This morphism is called the \emph{weighted blow-up} of $\C^N$ at $0$ with weight $w'$. We can again describe the variety $\hat\C^N$ locally using charts. For any $i$, consider the map $q_i: \C^N\rightarrow \C^N$ with components $q_i(z)_i$ $=$ \begin{footnotesize}$z_i^{w_i'}$\end{footnotesize} and $q_i(z)_j$ $=$ \begin{footnotesize}$z_j z_i^{w_j'}$\end{footnotesize} for all $j \neq i$. Then $\pi^{-1} \circ q_i|_{\{z_i \neq 0\}}$ extends to a map $\hat q_i: \C^N \to \hat\C^N$ such that $\hat q_i$ factors through an embedding of \begin{small}$\C^N/\Z_{w_i'}$\end{small} onto the
open set $\{(v,[u])$ $\in$ \begin{small}$\hat\C^N$\end{small}$: u_i \neq 0\}$, where \begin{small}$\Z_{w_i'}$\end{small} acts via $(\zeta. z)_i = \zeta^{-1} z_i$ and $(\zeta. z)_j$ $=$ \begin{small}$\zeta^{w_j'}$\end{small}$z_j$ for $j\neq i$. Under this embedding, $\{z_i = 0\}/\Z_{w_i'}$ maps onto $\hat{q}_i(\C^N) \cap Exc(\pi)$.

The \begin{small}$\Z_{w_i'}$\end{small}-action on $\C^N$ has no isotropy in codimension $1$ because ${\rm gcd}(w_1', \ldots, w_N') = 1$. Thus, the Weil divisor $Exc(\pi)$ is locally cut out by the \begin{small}$w_i'$\end{small}-fold multi-valued function \begin{small}$z_i \circ \hat{q}_i^{-1}$\end{small}, which we denote by $z_i$ as well. In other words, the single-valued function \begin{footnotesize}$z_i^{w_i'}$\end{footnotesize} vanishes to order \begin{small}$w_i'$\end{small} along the reduced analytic set $Exc(\pi)$. Moreover, it follows from the definition \eqref{d:weighted_radius} of $r_{w'}$ that each point of $Exc(\pi)$ contained in $\hat{q}_i(\C^N) \cap (\hat\C^N)^{reg}$ has a small neighborhood $K$ such that on $K$,
\begin{equation} \label{eqnB3}
C_K^{-1} |z_i|\leq r_{w'} \circ \pi \leq C_K|z_i|
\end{equation}
for some $C_K > 1$. In particular, $r_{w'} \circ \pi$ vanishes to first order along $Exc(\pi) \cap (\hat{\C}^N)^{reg}$.

Abusing notation, write $X$ for the germ $(X,x) \subset (\C^N, 0)$. Let $\hat{X}$ denote the {proper transform} of $X$ under $\pi: \hat\C^N \to \C^N$. Abusing notation again, write $\pi$ for the restriction of $\pi$ to $\hat{X}$. The proper birational morphism $\pi: \hat X \to X$ is then called the \emph{weighted blow-up} of $X$ at $x$ with weight $w'$. Let $E$ denote the exceptional divisor of this morphism. This is a subscheme of $\hat{X}$. The following lemma allows us to use $(\hat{X},E)$ to compute $A(d_{w'})$ via (b) above. This lemma would be completely false for a generic choice of weights $\tilde{w}, \tilde{w}'$ in place of $w, w'$, even for the same base variety $X$. 

\begin{lem}\label{lem:normality}
In our situation, $\hat{X}$ is normal and $E$ is reduced and irreducible.
\end{lem}

\begin{proof} In fact we will prove that $\hat{X}$ and ${E}$ are normal, using that $W$ is normal thanks to \cite{DS2}.\medskip\

\noindent \emph{Step 1}: \emph{The weighted blow-up $\hat{W}$ and its exceptional divisor $D$ are normal.} Choose $w'$-homogeneous generators  $f_1, \ldots, f_q$ of the ideal of $W$ in $\C[y_1, \ldots, y_N]$. By definition, the affine ring of $\hat W_i$ (i.e. the piece of $\hat W$ in the $i$-th affine chart of $\hat \C^N$) is the $\Z_{w_i'}$-invariant part of 
\begin{align}\label{e:ring_W_hat}
\C[z_1, \ldots, z_i, \ldots, z_N]/(f_1(z_1, \ldots, 1, \ldots, z_N), \ldots, f_q(z_1, \ldots, 1, \ldots, z_N)),
\end{align}
and the affine ring of $D_i$ (the piece of $D$ in the $i$-th chart) is the $\Z_{w_i'}$-invariant part of 
$$R_i = \C[z_1, \ldots, \hat{z}_i, \ldots, z_N]/(f_1(z_1, \ldots, 1, \ldots, z_N), \ldots, f_q(z_1, \ldots, 1, \ldots, z_N)).$$
Thus, $\hat W_i = (\tilde D_{i}\times \C)/$\begin{small}$\Z_{w_i'}$\end{small} and $D_i = \tilde D_{i}/$\begin{small}$\Z_{w_i'}$\end{small}, where $\tilde{D}_{i} = {\rm Spec}\;R_i$ and $\zeta$ $\in$ \begin{small}$\Z_{w_i'}$\end{small} acts on the $\C$-factor as multiplication by $\zeta^{-1}$. Hence a neighborhood of $D_i$ in \begin{small}$\hat W_i$\end{small} fibers over the disk with general fiber $\tilde D_i$ and central fiber $w_i' D_i$. Considering a neighborhood of the general fiber, it follows that $\tilde D_i \times \Delta$ ($\Delta$ is a disk) embeds as an open subset into $\hat W_i \setminus D_i \subset W \setminus \{0\}$. Since $W$ is normal, $\tilde{D}_i \times \Delta$ is too. Using the characterization of normality in terms of extending bounded holomorphic functions, one now easily checks that $\hat W_i = (\tilde D_i \times \C)/\Z_{w_i'}$, $\tilde D_i$, and $D_i = \tilde D_i/\Z_{w_i'}$ are all normal.\medskip\

\noindent \emph{Step 2}: \emph{$\hat{X}$ and $E$ are normal.} We will reduce this to Step 1 using a degeneration argument, which essentially amounts to deforming $\hat X$ to the normal cone of $E$.

 Let $\mathcal I \subset \mathcal O_{\C^N, 0}$ be the ideal of $(X,x) \subset (\C^N,0)$. By \cite{DS2}, the ideal of $W$ in $\C[y_1,\ldots,y_N]$ agrees with the so-called  $w'$-\emph{initial ideal} $\mathcal I_{w'}$ of $\mathcal I$, i.e. is generated by those $w'$-homogeneous polynomials $f\in\C[y_1,\ldots,y_N]$ that satisfy $f|_X = g|_X$ for some $g\in \mathcal O_{\C^N, 0}$ with $d_{w'}(g)>d_{w'}(f)$. Since $\mathcal{O}_{\C^N,0}$ is Noetherian, we can assume that our generators $f_1, \ldots, f_q \in \mathcal{I}_{w'}$ are chosen so that $f_p = g_p + h_p$ for all $p \in \{1, \ldots, q\}$, where $d_{w'}(g_p) > d_{w'}(f_p)$ and $h_1, \ldots, h_q \in \mathcal{I}$ generate $\mathcal{I}$ in $\mathcal{O}_{\C^N,0}$. Observe that $d_{w'}(f_p) = d_{w'}(h_p)$ and call this number $d_p \in \N$. Fix $i \in \{1, \ldots, N\}$ and any $z^* \in \C^{N}$ with $z^*_i = 0$. For any $t \in \C^*$ let $\mathcal{I}_t$ be the ideal of $\mathcal{O}_{\C^N, z^*}$ generated by
\begin{align}\label{e:defn_normal_cone}
f_p(z_1, \ldots, 1, \ldots, z_N) - (tz_i)^{-d_p}g_p(z_1 (tz_i)^{w_1'},\ldots, (tz_i)^{w_i'}, \ldots, z_N (tz_1)^{w_N'})
\end{align}
for $p = 1, \ldots, q$. These are holomorphic functions because $d_{w'}(g_p) > d_p$. Also, $\mathcal{I}_t$ is invariant under the stabilizer $\Gamma$ of $z^*$ in \begin{small}$\Z_{w_i'}$\end{small}, and by construction  $\mathcal{I}_1^\Gamma$ is the ideal of $(\hat{X}, \hat{q}_i(z^*))$ in $(\C^N, z^*)/\Gamma$ with respect to the $i$-th chart of \begin{small}$\hat\C^N$\end{small}. Moreover, $\mathcal{I}_s$ is $\Gamma$-equivariantly isomorphic to $\mathcal I_t$ for all $s,t \in \C^*$. 

From \eqref{e:defn_normal_cone}, the flat limit of $\mathcal{I}_t^\Gamma$ as $t \to 0$ contains the $\Gamma$-invariant part of the ideal generated by $f_p(z_1, \ldots, 1, \ldots, z_N)$ ($p = 1, \ldots, q$) in $\mathcal{O}_{\C^N,z^*}$. By \eqref{e:ring_W_hat}, this is the ideal of the germ $(\hat{W}, \hat{q}_i(z^*))$ in the $i$-th chart of \begin{small}$\hat\C^N$\end{small}. But this germ is reduced and irreducible because \begin{small}$\hat{W}$\end{small} is normal by Step 1, so it contains no proper $n$-dimensional sub-germs. Thus, $\lim_{t\to 0} \mathcal{I}_t^\Gamma$ is equal to the ideal of $(\hat{W}, \hat{q}_i(z^*))$. It then follows from the openness of normality in flat families \cite{banica} that $(\hat{X}, \hat{q}_i(z^*))$ is normal.

Finally, notice that $D \cong E$ by \eqref{e:defn_normal_cone}, so the normality of $E$ is trivial from Step 1.
\end{proof}

We are now in position to prove that $A(d_{KE}) = A(d_w) = cA(d_{w'}) = n$, where $w = cw'$ ($c \in \Q_{>0}$). Indeed, $d_{w'}=\nu_E$ by \eqref{eqnB3}, so $A(d_{w'}) - 1$ is the vanishing order of $\pi^*\Omega$ along $E$ by \eqref{eqnB4}. Again by \eqref{eqnB3}, it is then clear that $A(d_{w'}) = a(d_{w'})$, where by definition
\begin{align}\label{newdiscdef}
a(d_{\tilde{w}})=\sup\left\{a>0:\int_{V} r_{\tilde{w}}^{-2a} i^{n^2}\Omega\wedge\bar \Omega<\infty\right\}
\end{align}
for a fixed small neighborhood $V$ of $x$ in $X$ and all weights $\tilde w \in (\R_{>0})^N$. But $r_{w'}=r_w^c$, so $A(d_{w})$ $=$ $a(d_w)$. By (\ref{eqnB2}), we may replace $r_{\tilde{w}}$ by $r_{KE}$ $=$ \begin{small}$d_{\omega_{t_\infty}}(x,-)$\end{small} in \eqref{newdiscdef} if $\tilde{w} = w$. Since $\omega_{t_\infty}$ is Ricci-flat with volume form \begin{small}$i^{n^2}$\end{small}$\Omega\wedge\bar\Omega$ on $V$, we may also replace \begin{small}$i^{n^2}$\end{small}$\Omega \wedge\bar\Omega$ by $\omega_{t_\infty}^n$ in \eqref{newdiscdef} if $\tilde{w} = w$. Applying Bishop-Gromov volume comparison on dyadic annuli in $(V,\omega_{t_\infty})$, we get that $a(d_w) = n$.\hfill\vskip3mm

\noindent {\bf Case 2: $C(Y)$ is irregular.} Pick rational Reeb vector fields $\xi^{(i)} \in Lie(\mathcal T)$ converging to $\xi$. These induce weight vectors $w^{(i)} \in (\mathbb{Q}_{>0})^N$ through their action on $\C^N$, with $w^{(i)}$ converging to $w$. 

\begin{lem}
For $i$ large enough, the weighted tangent cone $W^{(i)}$  of $(X, x) \subset (\C^N,0)$ with respect to the weight vector $w^{(i)}$ is equal to $W$ as an affine subscheme of $\C^N$. 
\end{lem}

\begin{proof}
We have $W = {\rm Spec}\;\C[y_1,\ldots,y_N]/\mathcal{I}_w$ and $W^{(i)} = {\rm Spec}\;\C[y_1,\ldots,y_N]/\mathcal{I}_{w^{(i)}}$ as in the proof of Lemma \ref{lem:normality}. Pick a finite set of generators $f_1, \ldots, f_q$ for $\mathcal{I}_w$. These are $w$-homogeneous, but we can even assume that they are eigenvectors of the $\mathcal T$-action on $R(W)$, hence are $w^{(i)}$-homogeneous for all $i$ because $\xi^{(i)} \in Lie(\mathcal T)$. It follows by continuity that $f_1, \ldots, f_q \in \mathcal I_{w^{(i)}}$ for all large enough $i$, so $\mathcal I_{w^{(i)}}\supset \mathcal I_w$ and $W^{(i)}$ is an affine subscheme of $W$. (This is generically false if $\xi^{(i)} \not\in Lie(\mathcal{T})$.) Now $W^{(i)}$ is a weighted tangent cone of $X$, so by definition there exists a family of analytic sets over the disc with general fiber $X$ and central fiber $W^{(i)}$. By the upper semi-continuity of fiber dimensions in such a family, $\dim W^{(i)} \geq n$. 
Then $(W^{(i)})^{red}$ is a subvariety of dimension $\geq n$ of the irreducible $n$-dimensional variety $W^{red} = W$, so $W = (W^{(i)})^{red} \subset W^{(i)} \subset W$ and $W^{(i)} = W$.  
\end{proof}

It is easy to see that there exist $C > 1$ and  $\epsilon_i\rightarrow0$ such that for all $y\in \C^N$, 
$$C^{-1} r_w(y)^{1+\epsilon_i} \leq r_{w^{(i)}}(y) \leq Cr_w(y)^{1-\epsilon_i}.$$
Thus, by the discussion of Case 1, $A(d_{w^{(i)}}) \to n$ as $i \to \infty$. It is also clear that $d_{w^{(i)}}$ converges to $d_w$ in the weak topology of valuations on $\mathcal{O}_x$. Hence $A(d_w)\leq n$ by item (c) above.\vskip3mm

This finishes the proof of Proposition \ref{propB2}, hence of the key inequality \eqref{eqnB0}.

\begin{rmk}\label{rmk:logdisc} It seems natural to conjecture that $A(d_{KE}) = n$ always holds. This would be useful for developing the theory further. An inspection of examples with generic coprime integer weights (where $\hat{X}$ is non-normal and $E$ is non-reduced or reducible) suggests that the key identity $A(d_w) = $ $a(d_w)$, with $a(d_w)$ as in \eqref{newdiscdef}, may in fact hold quite generally even if our proof does not apply. So perhaps there is a way of proving \eqref{e:li-liu} directly with $A$ replaced by $a$ even if $C_x$ is irregular. 
\end{rmk}

\end{document}